\documentclass[reqno]{amsart}

\usepackage[normalem]{ulem}
\usepackage{mathscinet}
\usepackage{array}
\usepackage{booktabs}
\usepackage{numprint}
\usepackage{caption}
\usepackage{subcaption}
\usepackage{combelow}
\usepackage{amsthm}
\usepackage{amsfonts,amssymb,amsmath} 
\usepackage{aliascnt} 
\usepackage{mathrsfs} 
\usepackage{xargs,ifthen} 
\usepackage{enumitem}

\usepackage{bigints}
\usepackage{comment}
\usepackage{color}
\definecolor{lightgrey}{gray}{0.8}

\usepackage{shuffle}

\usepackage[all,cmtip]{xy}
\usepackage{tikz}
\usepackage{tikz-cd}
\usetikzlibrary{decorations.pathreplacing}
\usetikzlibrary{decorations.markings}
\usepgflibrary{arrows.meta}
\usetikzlibrary{calc}

\usepackage{comment}

\usepackage{todonotes}

\usepackage{hyperref}
\definecolor{sxred}{RGB}{153,0,0}
\definecolor{tocolor}{rgb}{.1,.1,.1}
\definecolor{urlcolor}{rgb}{.2,.2,.6}
\definecolor{linkcolor}{rgb}{.1,.1,.5}
\definecolor{citecolor}{rgb}{.4,.2,.1}
\definecolor{lightgrey}{rgb}{.9,.9,.9}
\hypersetup{colorlinks=true, urlcolor=urlcolor, linkcolor=linkcolor, citecolor=citecolor}

\newcommandx{\thdef}[2]{
	\newaliascnt{#1}{theorem}  
	\newtheorem{#1}[#1]{#2}
	\aliascntresetthe{#1}  
	\expandafter\newcommand\expandafter{\csname #1autorefname\endcsname}{#2}
}

\newtheorem{theorem}{Theorem}[section]

\thdef{lemma}{Lemma}
\thdef{corollary}{Corollary}
\thdef{conjecture}{Conjecture}
\thdef{proposition}{Proposition}
\thdef{theoremdefn}{Theorem/Definition}

\theoremstyle{definition}
\thdef{definition}{Definition}

\theoremstyle{remark}
\thdef{remarkx}{Remark}
\thdef{examplex}{Example}

\newenvironment{example}
  {\pushQED{\qed}\examplex}
  {\popQED\endexamplex}

\newenvironment{remark}
  {\pushQED{\qed}\remarkx}
  {\popQED\endremarkx}
  
\newcommand{\weak}{{^\mathbf{V}}}
\newcommand{\an}[1]{{#1}^{\mathrm{an}}}
\newcommand{\defn}[1]{\textbf{\textit{#1}}} 

\newcommand{\CC}{\mathbb{C}}
\newcommand{\KK}{\mathbb{K}}
\newcommand{\KKx}{\mathbb{K}^\times}
\newcommand{\RR}{\mathbb{R}}
\newcommand{\QQ}{\mathbb{Q}}
\newcommand{\NN}{\mathbb{N}}
\newcommand{\ZZ}{\mathbb{Z}}
\newcommand{\AF}{\mathbb{A}}
\newcommand{\PP}{\mathbb{P}}

\newcommand{\halfspc}[1]{[0,\infty)^{#1}}

\DeclareMathOperator{\divf}{div}

\newcommand{\id}[1]{\mathrm{id}_{#1}}

\renewcommand{\ge}{\geqslant}

\newcommand{\iu}{\mathrm{i}}

\newcommand{\tipi}{2\pi\iu}

\renewcommand{\epsilon}{\varepsilon}

\newcommand{\red}{\mathrm{red}}
\newcommand{\gp}{\mathrm{gp}}

\newcommand{\set}[2]{\left\{#1\,\middle|\,#2\right\}}
\newcommand{\rbrac}[1]{\left(#1\right)}
\newcommand{\abrac}[1]{\left\langle#1\right\rangle}

\newcommand{\fibprod}{\mathop{\times}}

\newcommand{\cO}[1]{\mathcal{O}_{#1}}
\newcommand{\cOx}[1]{\cO{#1}^\times}
\newcommand{\cOlog}[1]{\cO{#1}^{\log}}

\newcommand{\cM}{\mathscr{M}}

\newcommand{\cMphan}{\mathscr{M}^{\mathrm{phan}}}

\newcommand{\cMgp}{\mathscr{M}^{\mathrm{gp}}}
\newcommand{\cL}{\mathscr{L}}

\newcommand{\Cont}[1]{\mathscr{C}^0_{#1}}
\newcommand{\Contcirc}[1]{\mathscr{C}^{0, \unitcirc}_{#1}}
\newcommand{\Cinf}[1]{\mathscr{C}^\infty_{#1}}

\newcommand{\Cinfpos}[1]{\mathscr{C}^{\infty, > 0}_{#1}}
\newcommand{\Cinfge}[1]{\mathscr{C}^{\infty, \ge 0}_{#1}}

\newcommand{\dd}{\mathrm{d}}
\newcommand{\dlog}[1]{\frac{\dd#1}{#1}}
\newcommand{\ddlog}[1]{\mathrm{dlog}(#1)}
\newcommand{\wddlog}[1]{\mathrm{dlog}\weak(#1)}

\newcommand{\cvf}[1]{\partial_{#1}}

\newcommand{\uD}{\underline{D}}
\newcommand{\uL}{\underline{L}}
\newcommand{\uE}{\underline{E}}
\newcommand{\uF}{\underline{F}}
\newcommand{\uN}{\underline{N}}

\newcommand{\uX}{\underline{X}}
\newcommand{\uY}{\underline{Y}}
\newcommand{\uZ}{\underline{Z}}
\newcommand{\uS}{\underline{S}}

\newcommand{\uphi}{\underline{\phi}}
\newcommand{\upsi}{\underline{\psi}}

\newcommand{\pt}{*}

\newcommand{\logline}{\AF^1_{\mathrm{log}}}
\newcommand{\logpt}{*_{\mathrm{log}}}

\newcommand{\LogSch}{\mathbf{LogSch}}
\newcommand{\fsLogSch}{\mathbf{LogSch}^{\mathrm{fs}}}

\newcommand{\WLogSch}{\weak\mathbf{LogSch}}
\newcommand{\fsWLogSch}{\weak\mathbf{LogSch}^{\mathrm{fs}}}
\newcommand{\smWLogSch}{\weak\mathbf{LogSch}^{\mathrm{sm}}}
\newcommand{\diffable}{differentiable}

\newcommand{\mapdef}[5]{\begin{array}{ccccc}
#1 &:& #2 &\rightarrow& #3 \\
&& #4 &\mapsto& #5
\end{array}}

\newcommand{\Sig}{\Sigma}

\newcommand{\tb}[2][]{T_{#1}#2}

\newcommand{\KN}[2][]{\mathrm{KN}_{#1}(#2)}

\newcommand{\Spec}[2][]{\mathsf{Spec}_{#1}(#2)}

\newcommand{\rsect}[1]{\mathsf{R\Gamma}(#1)}
\newcommand{\sect}[1]{\mathsf{\Gamma}(#1)}
\newcommand{\coH}[2][\bullet]{\mathsf{H}^{#1}\rbrac{#2}}
\newcommand{\HdR}[2][\bullet]{\mathsf{H}_{\mathrm{dR}}^{#1}(#2)}
\newcommand{\HB}[2][\bullet]{\mathsf{H}_{\mathrm{B}}^{#1}(#2)}

\newcommand{\forms}[2][\bullet]{\Omega^{#1}_{#2}}
\newcommand{\wforms}[2][\bullet]{{^\mathbf{V}}\Omega^{#1}_{#2}}

\newcommand{\Hom}[2][]{\mathsf{Hom}_{#1}\rbrac{#2}}

\newcommand{\Gm}{\mathbb{G}_m}
\newcommand{\Ga}{\mathbb{G}_a}
\newcommand{\unitcirc}{\mathsf{S}^{1}}

\newcommand{\M}[2][]{\mathfrak{M}_{#2}^{#1}}
\newcommand{\bM}[2][]{\overline{\mathfrak{M}}_{#2}^{#1}}
\newcommand{\bbM}[2][]{\partial_{#1}\overline{\mathfrak{M}}_{#2}}

\newcommand{\Conf}{\mathsf{Conf}}

\newcommand{\Crv}[1]{\mathfrak{X}_{#1}}
\newcommand{\Tinf}[2][\infty]{T_{#1}\Crv{#2}}
\newcommand{\Tinfx}[2][\infty]{T^\times_{#1}\Crv{#2}}

\newcommand{\FM}[1]{\mathfrak{FM}_{#1}}
\newcommand{\FMfr}[1]{\FM{#1}^{\mathrm{fr}}}
\newcommand{\FMtop}[1]{FM_{#1}}

\title{Logarithmic morphisms, tangential basepoints, and little disks}

\author{Cl\'ement Dupont}
\address{Institut Montpelli\'erain Alexander Grothendieck, Universit\'e de Montpellier, CNRS,
Montpellier, France}
\email{clement.dupont@umontpellier.fr}

\author{Erik Panzer}
\address{Mathematical Institute, University of Oxford, UK}
\email{erik.panzer@maths.ox.ac.uk}

\author{Brent Pym}
\address{McGill University, Montr\'eal, Canada}
\email{brent.pym@mcgill.ca}

\begin{document}

\begin{abstract}
We develop the theory of ``virtual morphisms'' in logarithmic algebraic geometry, introduced by Howell.  It allows one to give algebro-geometric meaning to various useful maps of topological spaces that do not correspond to morphisms of (log) schemes in the classical sense, while retaining functoriality of key constructions.   In particular, we explain how virtual morphisms provide a natural categorical home for Deligne's theory of tangential basepoints: the latter are simply the virtual morphisms from a point. We also extend Howell's results on the functoriality of Betti and de Rham cohomology.  

Using this framework, we lift the topological operad of little $2$-disks to an operad in log schemes over the integers, whose virtual points are isomorphism classes of stable marked curves of genus zero equipped with a tangential basepoint. The gluing of such curves along marked points is performed using virtual morphisms that transport tangential basepoints around the curves. This builds on Vaintrob's analogous construction for framed little disks, for which the classical notion of morphism in logarithmic geometry sufficed. In this way, we obtain a direct algebro-geometric proof of the formality of the little disks operad, following the strategy envisioned by Beilinson. Furthermore, Bar-Natan's parenthesized braids naturally appear as the fundamental groupoids of our moduli spaces, with all virtual basepoints defined over the integers.
\end{abstract}

\maketitle

\setcounter{tocdepth}{1}
\tableofcontents

\section{Introduction}

\subsection{Overview}
In interactions between algebraic geometry  and other areas, such as topology, number theory and mathematical physics, one often encounters situations where one would like to ``compactify'' an algebraic variety $U$ by adding ``points at infinity'', without changing the homotopy type.  The problem is that if we try to compactify $U$ as a variety, we are forced to ``fill in the punctures'', which changes the topology in a drastic way (e.g.~the complex plane becomes a sphere).

Deligne \cite{Deligne1989} realized that points at infinity should instead be treated by choosing a regular compactification $X$ of $U$ whose boundary $D:=X\setminus U$ is a normal crossing divisor, and decorating a point of $X$ with a tangent vector pointing outwards from $D$. One can think of such ``tangential basepoints'' as analogous to points in a punctured tubular neighbourhood of $D$ in the classical topology. Correspondingly, Deligne proves that tangential basepoints give rise to fibre functors for various Tannakian categories of local systems, vector bundles with integrable connection, etc.\ on $U$. 

The need for tangential basepoints and their generalizations can be recognized in many contexts, e.g.:
\begin{itemize}
    \item To define the limit of the (co)homology groups of a family of varieties $Y_t$, $t \in \AF^1\setminus\{0\}$, one needs to choose a direction to approach $t \to 0$, which is naturally indicated by a tangential basepoint at $t=0$.
    \item In arithmetic geometry, one often considers schemes $U$ that have no rational (or integral) points, like $\mathbb{P}^1\setminus \{0,1,\infty\}$ over $\ZZ$, or have none that are compatible with symmetries or other structure. For instance, this is the case for the collection of moduli spaces of marked curves $\M{g,n}$, for which Grothendieck had already advertised the use of basepoints ``at infinity'' in what later became known as Grothendieck--Teichm\"{u}ller theory 
\cite{GrothendieckEsquisse}.
\item  In anabelian geometry, a variant of Grothendieck's section conjecture predicts that tangential basepoints classify sections of the map from the \'{e}tale fundamental group of a hyperbolic curve to the absolute Galois group \cite{StixSection}.
\item In mathematical physics, e.g.~in the study of Feynman integrals, it is common to encounter logarithmically divergent integrals that need to be assigned finite values by a process of ``regularization''. This requires a choice of scale near the boundary of the integration domain, which as we explain in \cite{DPP:Integrals}, is naturally encoded in a family of tangential basepoints. 
\item Various spaces and maps arising in quantum algebra and topology look like they should come from algebraic geometry, if only one could allow the maps to take values at infinity. A key example is the topological operad of little 2-disks.  The space of $n$ little disks embedded in a bigger one is homotopy equivalent to the configuration space of $n$ distinct points in $\CC$, which is an algebraic variety.  However, the topological operation of gluing smaller disks into bigger ones is not expressible by a map of varieties.  In this paper, building on work of Vaintrob~\cite{Vaintrob2021}, we will explain how one can nevertheless realize the little disks operad algebro-geometrically---even over $\Spec{\ZZ}$---as a ``virtual logarithmic moduli space of genus zero curves with a tangential basepoint'', yielding a direct algebro-geometric proof of its formality, following a strategy of Beilinson~\cite{BeilinsonLetterToKontsevich}. The question of the algebro-geometric nature of the little disks operad was also raised independently by Morava \cite[appendix]{morava}.
\end{itemize}

 The core difficulty in dealing with tangential basepoints is the need to carry around the extra data of the tangent vectors, taking care to ensure that all constructions are compatible with them.  On a case-by-case basis, this is not a major problem, but when many varieties/schemes start to interact via nontrivial morphisms, as in many of the examples above, it becomes more cumbersome. For instance, to construct certain cases of motivic fundamental groups with tangential basepoints, Deligne--Goncharov~\cite{DeligneGoncharov} first introduce extra unnatural basepoints in the interior and must later eliminate them using transcendental techniques, to obtain the desired object.

The underlying conceptual issue is ultimately that tangential basepoints are not truly basepoints in the categorical sense, i.e.~maps from the terminal object $* = \Spec{\KK}$ in the category of schemes over $\KK$. The purpose of this paper is to explain how logarithmic algebraic geometry can be used to solve this problem. However, this requires an addition to the classical theory of log schemes of Fontaine--Illusie and Kato \cite{IllusieLogSpaces, Kato1989a}, which capitalizes on the following remarkable fact:
\begin{quote}
    \textbf{Many constructions in log geometry are functorial for a broader class of morphisms.}
\end{quote}
Namely, following ideas of Howell~\cite{Howell2017} (a variant of which we discovered later but independently in \cite{DPP:Integrals}), we will describe a category of log schemes and ``virtual morphisms'', a natural variant of the classical notion of morphisms. In this category, tangential basepoints are truly basepoints in the categorical sense:
\begin{quote}
    \textbf{Tangential basepoints are simply virtual morphisms from a point.}
\end{quote}
We also have a functorial theory of Betti and de Rham cohomology, facilitating the treatment of the scenarios above.

\subsection{Virtual morphisms in log geometry}

The first step towards realizing tangential basepoints is to enlarge the category of schemes to include more general objects that encode compactifications and degenerations thereof.  Logarithmic geometry, in the sense of Fontaine--Illusie and Kato~\cite{IllusieLogSpaces,Kato1989a}, is a natural such context.  Recall that a logarithmic scheme is a scheme equipped with a factorization of the natural inclusion $\cOx{}\hookrightarrow \cO{}$ of the invertible functions through an intermediate sheaf of (multiplicative) monoids $\cM$.  A compactification $X \supset U$ with divisorial boundary carries a natural such structure, where $\cM$ is the sheaf of regular functions on $X$ that are invertible on $U$.  The open $U$ is then recovered from $(X,\cM)$ as the locus where $\cM=\cOx{}$, so that the logarithmic structure ``remembers'' the fact that $X$ was obtained by compactifying $U$.  This viewpoint is amplified by Kato--Nakayama's construction~\cite{KatoNakayama} of a topological space associated to fs log schemes over $\CC$, which when applied to a compactifying log scheme as above,  produces a compactification of the complex points $U(\CC)$ that preserves the homotopy type (in the analytic topology). 
 
 However, enlarging the class of objects in this way does not fully solve the problem of tangential basepoints.  The reason is that if we use the standard notion of morphisms between log schemes, i.e.~morphisms of schemes and monoids respecting both maps in the factorization $\cOx{}\to\cM\to\cO{}$, then interiors of compactifications are preserved.  Thus, in particular, a categorical basepoint in a logarithmic compactification is forced to be an ordinary basepoint in $U$ itself.  In other words, there are not enough morphisms in the usual category of log schemes to capture the notion of a tangential basepoint.

There are also other situations where one may wish to find additional morphisms in logarithmic geometry.  For instance, if $L \to U$ is a line bundle that extends to a line bundle $\widetilde{L} \to X$ on a compactification $X \supset U$, then there is a log structure on $X$ whose associated topological space has the homotopy type of the unit circle bundle of $L$. However, the isomorphism class of this log structure depends on the isomorphism class of the extension $\widetilde{L}$. In particular, it can happen that $L$ is trivial, but the log structure on $X$ is not, so that the latter does not fully reflect the simplicity of the geometry of $L$.

A crucial observation, which is a central theme of Howell's PhD thesis~\cite{Howell2017}, is that many constructions in logarithmic geometry---particularly those with a cohomological flavour---actually only use the map $\cOx{} \to \cMgp$, where $\cMgp$ is the group completion of the monoid $\cM$; Howell calls this map the associated ``virtual log structure''.  It is therefore profitable to define an alternative notion of morphism between log schemes, by requiring only compatibility with $\cOx{} \to \cMgp$ and dropping compatibility with $\cM{} \to \cO{}$.   In this way, we obtain a category $\WLogSch$ of log schemes and ``virtual'' morphisms, which contains the classical category of schemes as a full subcategory (i.e.~the morphisms between ordinary schemes are unchanged), and a functor $\LogSch \to \WLogSch$ which is faithful on the subcategory of log schemes whose sheaf of monoids is integral (i.e.~an ordinary morphism is determined by the corresponding virtual morphism whenever the map $\cM\to\cMgp$ is injective). We record some basic properties of this category in \autoref{sec:virtual-mors}, building on the results in Howell's thesis.  We include many concrete examples to illustrate the subtle differences between virtual morphisms and ordinary ones.

\subsection{Tangential basepoints as virtual morphisms}
In \autoref{sec:tangential-basepoints}, we unpack the definition of virtual morphism in the case of the log structure associated to a strict normal crossing divisor.  This yields the following result, which justifies the slogan that ``tangential basepoints are just points'':

\begin{theorem}
Let $X$ be the log scheme associated to a strict normal crossing divisor $\uD$ in a smooth variety $\uX$ over a field $\KK$. Let $*$ denote the point $\Spec{\KK}$ equipped with the trivial log structure. 
 Then virtual morphisms $* \to X$ of log schemes over $\KK$  are in natural bijection with tangential basepoints for the pair $(\uX,\uD)$ in the sense of Deligne.
\end{theorem}

We thus propose the concept of virtual morphisms from a point, or simply ``virtual points'' as a natural notion of tangential basepoints in more general settings, e.g.~if $\KK$ is replaced by a more general base ring (such as $\ZZ$), or the smoothness hypotheses on $(\uX,\uD)$ are relaxed.  We give some examples of these scenarios in \autoref{sec:general-basepoints}.

\subsection{Virtual morphisms and cohomology}
The category $\WLogSch$ of virtual morphisms is also well behaved with respect to basic cohomological invariants of log schemes, generalizing the (singular) Betti and (algebraic) de Rham cohomology of varieties.  Here, we restrict our attention to the classes of log schemes considered by Kato--Nakayama in \cite{KatoNakayama}, for which these cohomology theories are defined.  These give full subcategories
\[
\smWLogSch_\KK \subset \fsWLogSch_\KK \subset \WLogSch_\KK
\]
consisting of log schemes over a field $\KK$ satisfying certain  smoothness and finiteness conditions (``ideally smooth'' and ``fs'', respectively).  Kato--Nakayama associate to such log schemes a topological space $\KN{X}$ (when $\KK=\CC$) whose singular cohomology defines the Betti cohomology $\HB{X}$, and an algebraic de Rham complex $\forms{X}$ (for any $\KK$ of characteristic zero), whose hypercohomology defines the de Rham cohomology $\HdR{X}$.  They prove that these are naturally isomorphic as functors on the category of ideally smooth log schemes and ordinary morphisms.

Howell~\cite[\S3.3]{Howell2017} explained that the Kato--Nakayama space, and in some cases the forms, are functorial for virtual morphisms as well.  We review and extend these results in \autoref{sec:Betti} through \autoref{sec:comparison}.

In particular, on the Betti side, the Kato--Nakayama space also has a log structure, now in the $C^\infty$ context of ``positive log differentiable spaces'' as developed by Gillam--Molcho~\cite{GillamMolcho}.  We prove that the positive log structure is also functorial for virtual morphisms, a fact we exploit in \cite{DPP:Integrals} to give cohomological meaning to logarithmically divergent period integrals on algebraic varieties. 

Meanwhile, on the de Rham side, we prove that the log de Rham complex $\forms{}$ is functorial for virtual morphisms $X \to Y$, provided that the scheme underlying $X$ is reduced;  since de Rham cohomology is invariant under nilpotent extensions, its functoriality in general follows.  Howell had proven a similar result under the stronger assumption of integrality. Our proof is similar, but invokes a simple ``continuity principle'' (\autoref{lem:continuity-principle}) that allows one to make arguments involving the map $\cM \to \cO{}$ from knowledge of $\cOx{} \hookrightarrow \cMgp$, by restricting to the locus where a given function is invertible. 

In summary, we have the following extension of the results in \cite[\S3.3]{Howell2017}:
\begin{theorem}\label{thm:BdR}
The following statements hold.
\begin{enumerate}
    \item\label{it:Betti} For $\KK=\CC$, the construction of the positive log differentiable space $\KN{X}$ defines a symmetric monoidal functor on $\fsWLogSch_\CC$ and hence so does the Betti cohomology $\HB{X} := \coH{\KN{X};\ZZ}$.
    \item\label{it:deRham} For any field $\KK$ of characteristic zero, the construction of the logarithmic de Rham cohomology $\HdR{X}$  extends to a symmetric monoidal functor on $\smWLogSch_\KK$.  
    \item\label{it:comp} For $\KK=\CC$, the comparison isomorphism $\HB{X} \otimes \CC \cong \HdR{X}$ of Kato--Nakayama is natural for virtual morphisms, i.e.~it is an isomorphism of symmetric monoidal functors on $\smWLogSch_\CC$.
\end{enumerate}
\end{theorem}

\subsection{Virtual morphisms and regularized pullbacks}
In \autoref{sec:strata} we explain how the functoriality of cohomology plays out for an important class of virtual morphisms, related to inclusions of strata in normal crossing divisors.  We show that the induced maps on cohomology have a purely classical description as the composition of specialization maps for punctured tubular neighbourhoods, and pullbacks along sections of punctured normal bundles. On the level of de Rham cohomology this amounts to a ``regularized pullback'' in which logarithmic poles $\dlog{z}$ are formally set to zero on the divisor $z=0$ via a consistent choice of tangential basepoints.  It follows immediately that  the periods obtained by pairing relative Betti cycles and de Rham cocycles of such morphisms can be expressed explicitly as classical periods in the sense of Kontsevich--Zagier~\cite{KontsevichZagier}.

\subsection{Virtual morphisms and little disks}
Finally, in \autoref{sec:little-disks}, we apply the formalism to give an algebro-geometric construction of the little disks operad via moduli spaces of curves with tangential basepoints, as suggested by Beilinson \cite{BeilinsonLetterToKontsevich}. It is modelled on an analogous construction of Vaintrob~\cite{Vaintrob2021} for the \emph{framed} little disks operad, which uses \emph{ordinary} morphisms.  However, as observed in \S4 of \emph{op.~cit.}, such morphisms cannot be used in the unframed case. Thus, our construction uses virtual morphisms in an essential way.

Namely, we construct log schemes $\FM{n}$, $n \ge 2$, lying over the moduli spaces $\bM{0,n+1}$ of stable marked curves of genus zero.  The virtual points of $\FM{n}$ are in bijection with stable marked curves of genus zero endowed with a tangential basepoint at one of the markings, and a choice of first-order smoothing of every node, in the spirit of Kimura--Stasheff--Voronov \cite{Kimura1995}. We explain how to lift the classical maps $\bM{0,n+1} \times \bM{0,m+1}\to\bM{0,n+m+1}$ that glue marked curves to virtual morphisms $\FM{n}\times\FM{m}\to\FM{n+m-1}$ that are defined via a canonical recipe for transporting nonzero tangent vectors between different points on a genus zero curve.  The result is the following:
\begin{theorem}
    The spaces $\FM{n}$, $n \ge 2$ assemble into an operad in the category $\smWLogSch_\ZZ$, whose operad of Kato--Nakayama spaces is isomorphic to the Fulton--MacPherson model $\FMtop{}$ of the little disks operad. 
\end{theorem}

As an immediate consequence, we obtain in \autoref{sec:formality} a new proof of the following statement.
\begin{corollary}
    The little disks operad is formal.
\end{corollary}

    Earlier proofs of this statement were given by Tamarkin \cite{TamarkinDisks}, Kontsevich \cite{KontsevichOperadsMotives} with details by Lambrechts--Voli\'{c} \cite{LambrechtsVolic}, Petersen \cite{PetersenDisks}, Cirici--Horel \cite{CiriciHorelMHS}, and Drummond-Cole--Horel \cite{DrummondColeHorel}. Our proof follows Beilinson's suggestion \cite{BeilinsonLetterToKontsevich} that formality should be an immediate consequence of the fact that the little disks operad ``comes from an operad in algebraic geometry''. Vaintrob \cite{Vaintrob2021} (see also \cite{Vaintrob2019}) previously managed to carry out that strategy for the \emph{framed} little disks operad using his operad of log schemes and \emph{ordinary} morphisms (\autoref{sec:operads-moduli-spaces}); he then combined it with an extra argument about the behaviour of the cochain $\infty$-functor to deduce formality of the \emph{unframed} little disks operad. Our proof, which needs \emph{virtual} morphisms, is the first direct implementation of Beilinson's suggestion.

\subsection{Open questions and relations to other works}

\subsubsection{Motives}
\autoref{thm:BdR} above should be the shadow of a motivic structure associated to log schemes and virtual morphisms.  There are nowadays several approaches to motives of log schemes, e.g.~\cite{Howell2017,Shuklin,BindaAndCo, ParkA1}, in which various parts of this correspondence are made precise, but the comparison between the different approaches appears to be subtle.  It would be interesting and useful to explicitly formulate the relationships between these constructions. As already noted by Howell \cite{Howell2017} and explained by Shuklin \cite{Shuklin} in the setting of classical motives, virtual morphisms give rise to motivic specialization functors (see \autoref{sec:specialization} below) and shed a new light on the construction of limit motives by Spitzweck \cite{SpitzweckLimit}, Levine \cite{LevineTubular}, and Ayoub \cite{AyoubThesis1, AyoubThesis2}.

\subsubsection{Other Weil cohomology theories}
Relatedly, other Weil cohomology theories, such as \'etale and crystalline cohomology, have been extended to the logarithmic context.  Indeed, all of them must extend by the motivic considerations above.  On the other hand, Kato--Nakayama~\cite{KatoNakayama} proved a direct equivalence between logarithmic \'etale and Betti cohomology (with constructible torsion coefficients).  It thus follows indirectly from \autoref{thm:BdR} that \'etale cohomology for log schemes over $\CC$ must also be functorial for virtual morphisms. It would be instructive to prove the functoriality directly from the definition of \'etale cohomology and virtual morphisms, and to do similarly for other Weil cohomology theories. 

\subsubsection{K\"ahler differentials and smoothness}
In the classical theory of schemes, there are many equivalent definitions of smoothness, but their natural analogues in log geometry  do not all coincide.  For instance, consider the standard log point $\logpt$ over a field $\KK$. 
 Its K\"ahler differentials $\forms[1]{\logpt/\KK}$ are free, and moreover, when $\KK=\CC$ its Kato--Nakayama space is a smooth manifold (the circle $\unitcirc$).    However, $\logpt$ is not considered smooth over $\KK$ (only ``ideally smooth'', which is enough for our purposes).  The reason is that in log geometry, ``smoothness'' is a condition on infinitesimal liftings, which cannot be solved in this case, because there are no ordinary morphisms $\Spec{\KK} \to \logpt$.  In contrast, there are many \emph{virtual} morphisms, which suggests that considering virtual infinitesimal liftings in the definition of smoothness may result in a theory in which the log point is smooth and the connection with K\"ahler differentials is more direct.

\subsubsection{Fundamental groupoids}
As illustrated by the example of parenthesized braids and little disks, the realization of tangential basepoints as virtual morphisms from a point allows one to give natural constructions of the fundamental groupoid with tangential basepoints in its various incarnations (de Rham, Betti, \'etale, etc.). This will be developed further in future work. Note that in the context of $p$-adic Hodge theory, Olsson \cite[Chapter 9]{olssonpadic} already used logarithmic geometry to show how tangential basepoints provide fiber functors for relevant Tannakian categories.

\subsubsection{Riemann--Hilbert correspondence and ``splittings''}
Kato--Nakayama~\cite{KatoNakayama} and Ogus~\cite{OgusRHC} have proven a Riemann--Hilbert correspondence for logarithmic flat connections on ideally smooth log schemes.  This has recently been extended by Achinger~\cite{Achinger2023} to handle connections with regular singularities at infinity.  A key ingredient in Achinger's construction is the notion of a ``splitting'' of a log structure, which  ``behaves as if it were a section of the map from [a log scheme] $X^\sharp$ to its underlying scheme $\uX^\sharp$ even though no such section exists, and in particular it allows one to turn log connections into classical connections by means of a `pull-back'
functor'' (\emph{op.~cit.}, p.~3). Virtual morphisms give a framework in which this intuition becomes precise: a splitting of an ideally smooth log scheme is literally equivalent to a section $\uX^\sharp \to X^\sharp$ in the sense of virtual morphisms; see \autoref{ex:achinger-splittings}.  The pullback on connections is then induced by the natural pullback on forms. 

\subsubsection{Deformation quantization} The present work is part of our ongoing project on the motivic structures arising in deformation quantization of Poisson manifolds, in which the little disks operad plays an important role.  In future work, we will build on these results to treat the period integrals arising in Kontsevich's  formality morphism~\cite{Kontsevich2003} motivically, thus realizing his vision of a motivic Galois action in deformation quantization~\cite{KontsevichOperadsMotives} that can be directly compared with the known action of the Grothendieck--Teichm\"{u}ller group~\cite{Dolgushev2021,Willwacher:GCandGRT}.

\subsection*{Conventions and notation}
Throughout this paper, $\KK$ is a field. All monoids are implicitly commutative, and we almost always write the monoid law multiplicatively. The main exception is the set $\NN$ of natural numbers (i.e.~non-negative integers, including zero), which we view as a monoid under addition.

\subsection*{Acknowledgements}
We thank Piotr Achinger, Federico Binda, Damien Calaque, Danny Gillam, Samouil Molcho,  Sam Payne, Yiannis Sakellaridis, and Matt Satriano
for helpful discussions and correspondence. We are especially grateful to Dan Petersen for a stimulating conversation at CIRM, which helped refine our ideas about little disks and virtual morphisms, and prompted us to include them in the present paper. We thank Xu Gao for identifying an error in the proof of \autoref{prop:virtual-Z-points-of-FM-are-binary} in a previous version of the article, which we have now corrected. This work grew out of discussions held by the authors during a Research in Pairs activity at the Mathematisches Forschungsinstitut Oberwolfach (MFO) in 2022 and a Research in Residence activity at the Centre International de Rencontres Math\'{e}matiques (CIRM) in 2024. We thank both institutes for providing excellent hospitality and working conditions.
C.~D.\ was supported by the Agence Nationale de la Recherche through grants ANR-18-CE40-0017 PERGAMO and ANR-20-CE40-0016 HighAGT.
E.~P.\ was funded as a Royal Society University Research Fellow through grant {URF{\textbackslash}R1{\textbackslash}201473}.
B.~P.\ was supported by an Discovery Grant from the Natural Sciences and Engineering Research Council of Canada, a New university researchers startup grant from the Fonds de recherche du Qu\'ebec -- Nature et technologies (FRQNT), and a startup grant at McGill University.

\section{Virtual morphisms in log geometry}
\label{sec:virtual-mors}

In this section, we describe the notion of virtual morphisms in log geometry and the basic properties of the resulting category. Even though we only deal with log schemes, many aspects of the discussion work equally well for log analytic spaces and other log geometry contexts such as (positive) log differentiable spaces~\cite{GillamMolcho}. 
 Basic references on log algebraic geometry include \cite{Kato1989a, Ogus, Temkin:IntroLog}.

\subsection{Log structures and log schemes}
Recall that a \defn{pre-log structure} on a scheme $X$ is a sheaf of monoids $\cM_X$ on $X$ in the \'{e}tale topology along with a morphism of sheaves of monoids
\[
\alpha_X\colon \cM_X\to \cO{X}.
\]
It is called a \defn{log structure} if the induced morphism
\[
\alpha_X^{-1}(\cOx{X})\to \cOx{X}
\]
is an isomorphism, where $\cOx{X}\subset\cO{X}$ is the subsheaf of invertible elements. A \defn{(pre-)log scheme} is a scheme equipped with a (pre-)log structure. Following the convention in log geometry, we will often abuse notation and simply write $X$ for a (pre-)log scheme $(X,\cM_X,\alpha_X)$, sometimes denoting by $\uX$ the underlying scheme when there is a risk of confusion. If $X$ is a log scheme, we will tacitly identify $\cOx{X}$ with the submonoid $\alpha_X^{-1}(\cOx{X})$ of $\cM_X$, viewing $\alpha_X$ as a factorization of the inclusion of $\cOx{X}$ inside $\cO{X}$:
\[
\cOx{X}\hookrightarrow \cM_X \stackrel{\alpha_X}{\longrightarrow} \cO{X}.
\]

A pre-log structure $(\cM_X,\alpha_X)$ has an associated log structure $(\cM^{\log}_X,\alpha^{\log}_X)$, given by the pushout
\begin{equation}
\cM_X^{\log} := \cOx{X} \underset{\alpha_X^{-1}(\cOx{X})}{\sqcup}\cM_X  \label{eq:logification-pushout}
\end{equation}
together with the map $\alpha^{\log}_X\colon \cM_X^{\log}\to \cO{X}$ induced by the inclusion $\cOx{X}\hookrightarrow \cO{X}$ and the map $\alpha_X \colon \cM_X \to \cO{X}$. This \defn{logification} procedure allows one to define new log structures from old ones, which is convenient because usually the naive operation only produces a pre-log structure.

\begin{example}[The spectrum of a log ring]
Let $M$ be a monoid, $R$ a commutative ring and $\alpha \colon M \to R$ a morphism from $M$ to the multiplicative monoid of $R$; such a triple $(R,M,\alpha)$ is called a \defn{log ring}.  This data defines a pre-log structure on $\uX = \Spec{R}$, for which $\cM_X$ is the constant sheaf with stalk $M$ and $\alpha_X \colon \cM_X\to \cO{X}$ is the canonical map induced by $\alpha$.  The logification of this pre-log structure is called the spectrum of $(R,M,\alpha)$; we shall typically denote it by $\Spec{M \to R}$, the map $\alpha$ being understood.
\end{example}

\begin{example}[Trivial log structure]
A scheme $X$ can be made into a log scheme with the \defn{trivial log structure} $\cM_X=\cOx{X}$, with $\alpha_X\colon \cM_X \to \cO{X}$ the inclusion.
\end{example}

\begin{example}[Divisorial log structure]\label{ex:divisorial-log-structures}
Let $\uX$ be a scheme and $\uD\subset \uX$ an effective Cartier divisor, i.e.~a closed subscheme locally defined by the vanishing of a function which is not a zero divisor. Let $j\colon \uX\setminus \uD\hookrightarrow \uX$ denote the open immersion of its complement. We have a log structure $\cM_{X,D}=j_*\cOx{X\setminus D}\cap \cO{X}$ whose sections are the regular functions on $\uX$ which become invertible when restricted to $\uX\setminus \uD$, with $\alpha_{X,D} \colon \cM_{X,D} \to \cO{X}$ the inclusion. We call this a \defn{divisorial log structure}.
\end{example}
Given a divisorial log structure we may pull it back to a closed subscheme (in the sense of  \autoref{sec:pullback} below) to obtain new examples of log structures for which $\alpha_X$ is not injective.

\begin{example}[The log line]\label{ex:log-line}
   The \defn{log line} $\logline$ is the log scheme over $\KK$ whose underlying scheme is the affine line $\AF^1=\Spec{\KK[z]}$
   and whose log structure is the divisorial log structure for the divisor $\{z=0\}$. Concretely, $\cM_{\logline} = z^\NN \cOx{\AF^1} \subset \cO{\AF^1}$ is the sheaf of functions that are expressible as a monomial in $z$ times an invertible function. Equivalently, $\logline = \Spec{z^\NN\hookrightarrow \KK[z]}$.
\end{example}

\begin{example}[The log point]\label{ex:log-point}
    The \defn{log point} $\logpt$ is the log scheme over $\KK$ whose underlying scheme is $\Spec{\KK}$ and whose log structure is the pullback of the log structure of $\logline$ along the closed immersion $i\colon \{0\} \hookrightarrow \AF^1$. Let $t$ denote the germ at $0$ of the coordinate $z$ on $\AF^1$. Then $\cM_{\logpt}$ has global sections given by $\KK^\times  t^\NN$, the product of $\KK^\times$ with the free monoid generated by $t$, and $\alpha$ is given by ``evaluation at $t=0$'':
    \[
    \alpha\colon \KK^{\times} t^\NN \to \KK  \qquad \qquad\lambda t^j \mapsto \lambda 0^j := \begin{cases} \lambda & \mbox{if } j=0 \\ 0 & \mbox{if } j>0.  \end{cases} 
    \]
    Equivalently, $\logpt =\Spec{t^\NN\to \KK[t]/(t)}$.
\end{example}

\subsection{Ordinary and virtual morphisms}
Recall that a morphism of (pre-)log schemes $\phi \colon X \to Y$ is a morphism of schemes $\uphi \colon \uX\to\uY$ together with a morphism of sheaves of monoids $\phi^* \colon \uphi^{-1}\cM_Y \to \cM_X$ such that the following diagram commutes, where the bottom horizontal arrow is the usual pullback of functions. 
\begin{equation}
\begin{tikzcd}
    \uphi^{-1}\cM_{Y} \ar[r,"\phi^*"]\ar[d,"\alpha_Y"'] & \cM_{X} \ar[d,"\alpha_X"] \\
    \uphi^{-1}\cO{Y} \ar[r,"\uphi^*"'] \ar[r] & \cO{X}
\end{tikzcd} \label{eq:ordinary-commutativity}
\end{equation}
 We will refer to such morphisms as \defn{ordinary morphisms}, to distinguish them from the virtual morphisms we will now consider. 
 
 To formulate the notion of virtual morphisms, recall that the \defn{group completion} of a sheaf of (commutative) monoids $\cM$ is the universal sheaf of (abelian) groups $\cMgp$ receiving a homomorphism from $\cM$. The natural map $\cM \to \cMgp$ is injective if and only if $\cM$ is integral, i.e.~satisfies the cancellation law ($fg=fh$ implies $g=h$ for any local sections $f,g,h$ of $\cMgp$).

\begin{definition}\label{def:log-virtual-morph}
 Let $X$ and $Y$ be log schemes.  A \defn{virtual morphism} $X \to Y$ is a pair $(\uphi,\phi^*)$ consisting of a morphism of schemes $\uphi \colon \uX \to \uY$, and a morphism $\phi^* \colon \uphi^{-1}\cMgp_Y \to \cMgp_X$ of sheaves of groups, making the following diagram commute:
 \begin{equation}
 \begin{tikzcd} 
    \uphi^{-1}\cOx{Y} \ar[r,"\uphi^*"] \ar[d] & \cOx{X} \ar[d] \\
     \uphi^{-1}\cMgp_Y \ar[r,"\phi^*"] & \cMgp_X
\end{tikzcd} \label{eq:virtual-commutativity}
\end{equation}
\end{definition}

\begin{remark}
    By the universal property of group completion, the datum of $\phi^*$ is equivalent to the datum of a morphism $\uphi^{-1}\cM_Y \to \cMgp_X$.
\end{remark}

There are two notable differences between the notions of ordinary and virtual morphisms, illustrated by examples below (\autoref{sec:examples-of-morphisms}).
\begin{enumerate}
\item Firstly, the notion of virtual morphism does not require compatibility of pullback with the maps $\alpha$, only with invertible functions. If $Y$ is a log scheme, let us say that a section $f$ of $\cM_Y$ is a \defn{phantom} if $\alpha_Y(f) = 0$. If $\phi \colon X \to Y$ is an ordinary morphism, then the commutativity of \eqref{eq:ordinary-commutativity} implies that every phantom $f \in \uphi^{-1}\cM_Y$ pulls back to a phantom $\phi^*f \in \cM_X$, and more generally if $f \in \uphi^{-1}\cM_Y$ is such that $\uphi^*\alpha_Y(f)= 0$ (i.e.\ $f$ is a phantom ``relative to $\uphi$''), then $\phi^*f\in\cM_X$ is a phantom.  For a virtual morphism, this is no longer true and therefore virtual morphisms are allowed to ``breathe life into phantoms'' by pulling them back to non-phantoms, but the axiom asserts that they cannot change the values assigned to invertible functions.
\item Secondly, the pullback happens at the level of the group completions of the sheaves of monoids and produces formal quotients of monoid elements which are no longer associated to functions on the underlying scheme---this explains the terminology ``virtual'' (as in ``virtual representation of a group''). For instance, for the log point $\logpt$ (\autoref{ex:log-point}), the monoid element $t$ corresponds to the function $\alpha(t)=0$ on the point, but its inverse $t^{-1}$ does not correspond to any function.
\end{enumerate}

\subsection{Categories of log schemes}
Note that virtual morphisms of log schemes may be composed in the same way as ordinary morphisms: one simply composes the underlying morphisms of schemes and sheaves. 

\begin{definition}
We denote by $\WLogSch$ the category of log schemes with virtual morphisms, and by $\LogSch$ the category of log schemes with ordinary morphisms.   
\end{definition}
Thus $\LogSch$ and $\WLogSch$ have the same objects, but the morphisms are different.  Every ordinary morphism $X \to Y$ gives rise to a virtual morphism $X \to Y$ by replacing the pullback map $\phi^* \colon \uphi^{-1}\cM_Y \to \cM_X$ with the induced map $\uphi^{-1}\cMgp_Y \to \cMgp_X$.  Note that the former is completely determined by the latter if $\cM_X$ is integral.  Thus we have a functor
\[
\LogSch \to \WLogSch
\]
which is faithful on the subcategory of integral log schemes. Note that even in the integral setting there are more virtual isomorphisms (isomorphisms in $\WLogSch$) than ordinary isomorphisms (isomorphisms in $\LogSch$); see \autoref{ex:automorphisms-log-pt}, \autoref{ex:inverting-line-bundle} and \autoref{ex:automorphisms-DF-log-structures} below.

We may also consider categories of log schemes relative to a base, as follows.  
 For a log scheme $S$, we denote by $\LogSch_S$ (respectively, $\WLogSch_S$) the slice category $\LogSch/S$ (resp.~$\WLogSch/S$), whose objects are log schemes $X$ equipped with an ordinary (resp.~virtual) morphism of log schemes $X\to S$, and whose morphisms are ordinary morphisms (resp.~virtual morphisms) $X\to Y$ commuting with the morphisms to $S$. Clearly, $\WLogSch_S$ has more objects than $\LogSch_S$ in general---but not if $S$ has the trivial log structure, thanks to the following lemma, which is immediate from the definitions:

\begin{lemma}\label{lem:virtual-mor-to-scheme}
    Let $S$ be a scheme viewed as a log scheme with the trivial log structure. Then any virtual morphism of log schemes $X\to S$ is ordinary, i.e., is simply the datum of a morphism of schemes $\uX\to S$.
\end{lemma}

In particular, if $\KK$ is a ring and $S$ is the ``point'' $*:=\Spec{\KK}$ with the trivial log structure, we get the categories $\LogSch_\KK$ and $\WLogSch_\KK$ whose objects are log schemes over $\KK$ in the usual sense.

\subsection{Examples of virtual morphisms}\label{sec:examples-of-morphisms}
The following examples illustrate the differences between the notions of ordinary and virtual morphisms.

\begin{example}[Virtual morphisms to a log point]\label{ex:virtual-mor-to-log-pt}
Let $X$ be a log scheme over $\KK$ and $\logpt$ be the log point (\autoref{ex:log-point}). A virtual morphism $f\colon X\to \logpt$ of log schemes over $\KK$ is equivalent to the datum of a group homomorphism $t^{\ZZ}\to \cMgp(X):= \sect{X,\cMgp_X}$, i.e., an element $f^*(t)\in \cMgp(X)$. Thus virtual morphisms $X\to \logpt$ of log schemes over $\KK$ are in bijection with $\cMgp(X)$. In contrast, ordinary morphisms $X\to \logpt$ are in bijection with $\cMphan(X)\subset \cM(X)$, the ideal of \emph{phantom} global sections of $\cM_X$, i.e.~those sent to zero by $\alpha_X$.
\end{example}

\begin{example}[Virtual points of a log point]\label{ex:basepoints-of-log-point}
    As a special case of the previous example, let $X=*$ denote the point $\Spec{\KK}$ equipped with the trivial log structure. There is a unique morphism $p \colon \logpt \to \pt$ of log schemes over $\KK$, given by the identity map of $\Spec{\KK}$, with the map of monoids of global sections given by the inclusion $\KK^{\times} \hookrightarrow \KK^{\times}t^\NN$. Meanwhile, a virtual morphism 
    \[
    s\colon  * \to \logpt
    \]
    of log schemes over $\KK$ is determined by an element $s^*(t)=\lambda\in\KK^\times$. On global sections, $s^*\colon\KK^{\times}t^\NN \to \KK^{\times}$ is the morphism of monoids given by evaluation at $t=\lambda$.  Thus virtual morphisms $*\to\logpt$ of log schemes over $\KK$ are in bijection with $\KK^\times$, and are all sections of $p\colon \logpt\to *$. Note, however, that there are no ordinary morphisms from $*$ to $\logpt$. Indeed, since $\alpha(t)=0$, an ordinary morphism would have to send $t$ to a phantom in $\KK^\times$, of which there are none.  
\end{example}

\begin{example}[Automorphisms of the log point]\label{ex:automorphisms-log-pt}
A virtual automorphism $f$ of the log point $\logpt$ over $\KK$ is equivalent to the datum of
$$f^*(t)=\lambda t^j \quad \mbox{ with } \lambda\in\KK^\times \mbox{ and } j\in\{-1,1\}.$$
It is an ordinary automorphism if and only if $j=1$.
\end{example}

\begin{example}[Retraction from a log line to a log point]\label{ex:log-line-retract}
Let $X = \logline$ be the log line (\autoref{ex:log-line}). Letting $j\colon \AF^1\setminus \{0\}\hookrightarrow\AF^1$, the group completion of $\cM_X = z^\NN\cOx{\AF^1}$ is $\cMgp_X = j_*\cOx{\AF^1\setminus \{0\}}$,
, the sheaf of rational functions that are regular and invertible away from zero.  Thus the  global sections are given by $\cMgp(X) = \KK^\times z^\ZZ$, the set of Laurent monomials in $z$.
A virtual morphism 
\[
r\colon \logline \to \logpt
\]
of log schemes over $\KK$ is thus determined by an element $r^*(t)=\mu z^j$ with $\mu\in \KK^\times$ and $j\in\ZZ$. If $\mu=1$ and $j=1$, then $r$ is a retraction of the ordinary morphism $i\colon \logpt \hookrightarrow \logline$ defined by $i^*(z)=t$. Note, however, that there are no ordinary morphisms from $\logline$ to $\logpt$ because $\cM_{\logline}$ has no phantoms.
\end{example}

\begin{example}[Mapping a divisorial log structure to a log point]More generally, let $X = (\uX,\uD)$ be a divisorial log scheme over $\KK$ (\autoref{ex:divisorial-log-structures}). Letting $j\colon \uX\setminus \uD\hookrightarrow\uX$, we have $\cMgp_X = j_*\cOx{\uX\setminus \uD}$, the sheaf of rational functions that are regular and invertible on $\uX\setminus\uD$.  Therefore virtual morphisms $X \to \logpt$ of log schemes over $\KK$ are in bijection with $\cMgp(X) = \cOx{}(\uX\setminus \uD)$, or equivalently maps of schemes $\uX\setminus \uD \to \Gm$ over $\KK$.
\end{example}

\begin{example}[Splittings]\label{ex:achinger-splittings}
    Let $X$ be a log scheme and let $\uX$ be the underlying scheme, viewed as a log scheme with trivial log structure. 
 Then there is a canonical ordinary morphism $\pi\colon X \to \uX$, given by the inclusion $\cOx{X}\hookrightarrow\cM_X$ of sheaves of monoids.  When the log structure of $X$ is nontrivial, this morphism does not admit any ordinary sections, i.e.~ordinary morphisms $s\colon \uX \to X$ such that $\pi \circ s = \id{\uX}$.  However it may admit virtual sections: these are equivalent to monoid homomorphisms $\cM_X \to \cOx{X}$ that split the inclusion.  For integral monoids, such a morphism is equivalent to a section of the projection $\cM_X\to \cM_X/\cOx{X}$; the latter were called ``splittings'' of the log structure in \cite{Achinger2023}.
\end{example}

\subsection{Deligne--Faltings log schemes and virtual morphisms}

    \subsubsection{Definitions}

    Let $\uX$ be a scheme. A \defn{split vector bundle} over $\uX$ is a vector bundle $\uE$ over $\uX$ equipped with a decomposition  
    \begin{equation}\label{eq:splitting-bundle-into-lines}
    \uE = \uL_1 \oplus \cdots \oplus \uL_n
    \end{equation}
    as a direct sum of line subbundles. (The ordering of those line subbundles is not part of the data and simply appears here for convenience of notation.)     Equivalently, it is a vector bundle equipped with a reduction of its structure group to a torus $(\Gm)^n$. The \defn{interior} of $\uE$ is the associated $(\Gm)^n$-bundle
    \[
    \uE^\circ = \uL_1^\times\times_X\cdots \times_X \uL_n^\times.
    \]
    The torus $(\Gm)^n$ acts on both $\uE$ and $\uE^\circ$ by rescaling in each line bundle.
    
    A function $\uE\to \mathbb{A}^1$ (resp. $\uE^\circ\to\Gm$) over some \'{e}tale open of $\uX$ is \defn{monomial} if it is equivariant with respect to a character $(\Gm)^n\to \Gm$. 
    If $t_1,\ldots, t_n$ are local linear coordinates on the fibres of $\uL_1,\ldots,\uL_n$, then a monomial function $\uE\to\mathbb{A}^1$ is expressed as
    \begin{equation}\label{eq:fibrewise-monomial-function}
    f=g(x)\,t_1^{a_1}\cdots t_n^{a_n}
    \end{equation}
    with $a_i\in\NN$ and $g\in\cO{X}$. Similarly, a monomial function $\uE^\circ\to\Gm$ has the form
    \begin{equation*}
    f=h(x)\,t_1^{b_1}\cdots t_n^{b_n}
    \end{equation*}
    with $b_i\in\ZZ$ and $h\in \cOx{X}$.

    \begin{definition}\label{defi:DF-log-structures}
    A \defn{Deligne--Faltings datum} is a triple $(\uX,\uD,\uE)$ where $\uX$ is a scheme, $\uD \subset \uX$ is an effective Cartier divisor and $\uE \to \uX$ is a split vector bundle. The associated \defn{Deligne--Faltings log scheme} is the scheme $\uX$ equipped with the log structure $\cM_{X,D,E}$ whose sections are the monomial functions on $\uE$ that are nonvanishing on $\uE^\circ|_{\uX\setminus \uD}$, with $\alpha_{X,D,E}\colon \cM_{X,D,E} \to \cO{X}$ given by the restriction of functions to the zero section of $\uE$.
   \end{definition}

    A section of $\cM_{X,D,E}$ locally has the form \eqref{eq:fibrewise-monomial-function} with $a_i\in\NN$ and $g\in\cM_{X,D}\subset \cO{X}$ a regular function that is invertible on $\uX\setminus \uD$.  Then we have 
    \[
        \alpha_{X,D,E}(f) = \begin{cases}
            g & \mbox{ if }a_1=\cdots=a_n=0; \\
            0 & \mbox{ otherwise.}
        \end{cases}
    \]
    More invariantly, we have
    \begin{equation}\label{eq:DF-log-structure-described-invariantly}
    \cM_{X,D,E} \cong \cM_{X,D}\cdot (\cL_1^{\vee\times})^\NN\cdots (\cL_n^{\vee\times})^\NN
    \end{equation}
    with $\cL_i^{\vee \times}$ the sheaf of nonvanishing sections of $\uL_i^\vee$, and $\alpha_{X,D,E} \colon \cM_{X,D,E} \to \cO{X}$ given by extending the map $\cM_{X,D}\to\cO{X}$ via $\cL_i^{\vee\times} \mapsto 0$. 

    The canonical projection $(\uX,\cM_{X,D,E})\to \uX$ factors uniquely through a morphism $(\uX,\cM_{X,D,E})\to (\uX,\cM_{X,D})$ to the divisorial log structure of $(\uX,\uD)$.

\subsubsection{Morphisms as monomial maps}

Let $\uE = \bigoplus_{i=1}^n \uL_i$ and $\uE'=\bigoplus_{j=1}^{n'}\uL'_j$ be split vector bundles over $\uX$. A morphism $\uE\to \uE'$ (resp. $\uE^\circ\to \uE^{\prime\circ}$) of schemes over $\uX$ is \defn{monomial} if it is equivariant with respect to a morphism of tori $(\Gm)^n\to (\Gm)^{n'}$. Concretely, a monomial morphism $\uE\to \uE'$ is expressed in local trivializations of the line bundles by a formula
\begin{equation}
(t_1,\ldots,t_n) \mapsto \rbrac{g_1(x)\prod_i t_i^{a_{i,1}},\ \ldots,\  g_{n'}(x)\prod_i t_i^{a_{i,n'}}}
\label{eq:monomial-map}
\end{equation}
with $a_{i,j}\in\NN$ and $g_j\in\mathcal{O}_X$. Similarly, a monomial morphism $\uE^\circ\to \uE^{\prime\circ}$ has the form
\begin{equation}
(t_1,\ldots,t_n) \mapsto \rbrac{h_1(x)\prod_i t_i^{b_{i,1}},\ \ldots ,\ h_{n'}(x)\prod_i t_i^{b_{i,n'}}}
\label{eq:monomial-map-open}
\end{equation}
with $b_{i,j}\in\ZZ$ and $h_j\in\mathcal{O}^\times_X$.

\begin{definition}
Consider two Deligne--Faltings data $(\uX,\uD,\uE)$ and $(\uX,\uD,\uE')$ with the same underlying divisor.
\begin{enumerate}
    \item An \defn{ordinary monomial map} from $(\uX,\uD,\uE)$ to $(\uX,\uD,\uE')$ is a monomial map $\uE\to \uE'$ which sends the zero section of $\uE$ to the zero section of $\uE'$ and whose restriction to $\uX\setminus \uD$ sends $\uE^\circ|_{\uX\setminus \uD}$ to $\uE^{\prime\circ}|_{\uX\setminus \uD}$.
    \item A \defn{virtual monomial map} from $(\uX,\uD,\uE)$ to $(\uX,\uD,\uE')$ is a monomial map $\uE^\circ|_{\uX\setminus \uD} \to \uE^{\prime\circ}|_{\uX\setminus \uD}$.
\end{enumerate}
\end{definition}

Concretely, an ordinary monomial map is expressed locally by a formula \eqref{eq:monomial-map}, with $a_{i,j} \in \NN$ such that $\sum_{i} a_{i,j} > 0$ for all $j$, and with $g_j \in \cM_{X,D}\subset \cO{X}$ functions on $\uX$ which are invertible on $\uX\setminus \uD$. Likewise, a virtual monomial map has the form \eqref{eq:monomial-map-open} with $b_{i,j}\in \ZZ$ and $h_j\in\cOx{X\setminus D}$.

Note that the set of ordinary monomial maps injects in the set of virtual monomial maps because restriction of functions from $\uX$ to $\uX\setminus \uD$ is injective.

\begin{remark}
If $\uX\setminus\uD$ is a reduced scheme, then every map of $(\uX\setminus \uD)$-schemes $\uE^\circ|_{\uX\setminus \uD} \to \uE^{\prime\circ}|_{\uX\setminus \uD}$ is automatically a monomial map, and therefore defines a virtual monomial map from $(\uX,\uD,\uE)$ to $(\uX,\uD,\uE')$. Indeed, if $R$ is a reduced ring, then the monomials $ft_1^{a_1}\cdots t_n^{a_n}$ with $f\in R^\times$ and $a_i\in\ZZ$ are the only invertible elements of the ring $R[t_1^{\pm 1},\ldots,t_n^{\pm 1}]$. This fails if $R$ is not reduced, e.g. $1+2t$ is its own inverse in $(\ZZ/4\ZZ)[t^{\pm 1}]$.
\end{remark}

\begin{proposition}\label{prop:morphisms-as-monomial-maps}
Consider two Deligne--Faltings data $(\uX,\uD,\uE)$ and $(\uX,\uD,\uE')$. There is a natural bijection between ordinary (resp. virtual) monomial maps 
\[
(\uX,\uD,\uE) \to (\uX,\uD,\uE')
\]
and ordinary (resp. virtual) morphisms 
\[
(\uX,\cM_{X,D,E}) \to (\uX,\cM_{X,D,E'})
\]
of log schemes over $(\uX,\cM_{X,D})$, compatible with composition.
\end{proposition}

\begin{proof}
This is an easy consequence of the definitions, so we shall only sketch the main points of the proof.  Consider the following recipes to construct a logarithmic morphism from a monomial map:
\begin{enumerate}
\item Let $\varphi\colon (\uX,\uD,\uE)\to (\uX,\uD,\uE')$ be an ordinary monomial map. The pullback by $\varphi$ of a monomial function on $\uE'$ is a monomial function on $\uE$, and since $\varphi$ sends $\uE^\circ|_{\uX\setminus \uD}$ to $\uE^{\prime\circ}|_{\uX\setminus \uD}$, we get a morphism of sheaves of monoids
\[
\varphi^*\colon \cM_{X,D,E'}\to \cM_{X,D,E}
\]
which is the identity on $\cM_{X,D}$. The fact that $\varphi$ sends the zero section of $\uE$ to the zero section of $\uE'$ gives the compatibility of $\varphi^*$ with $\alpha_{X,D,E}$ and $\alpha_{X,D,E'}$, hence an ordinary morphism $(\uX,\cM_{X,D,E})\to (\uX,\cM_{X,D,E'})$. 
\item Let  $\varphi\colon (\uX,\uD,\uE)\to (\uX,\uD,\uE')$ be a virtual monomial map. Note that $\cMgp_{X,D,E}$ consists of monomial functions $\uE^\circ|_{\uX\setminus \uD}\to \Gm$ and therefore we get a morphism of sheaves of monoids
$$\varphi^*\colon \cMgp_{X,D,E'} \to \cMgp_{X,D,E}$$
which is the identity of $\cM_{X,D}$, hence a virtual morphism $(\uX,\cM_{X,D,E})\to (\uX,\cM_{X,D,E'})$.
\end{enumerate}
These recipes are clearly compatible with composition of morphisms, and one easily checks that they give bijections between ordinary (respectively, virtual) monomial maps and ordinary (resp.\ virtual) morphisms, as claimed.
\end{proof}

From now on, in the Deligne--Faltings setting, we will use the notions of Deligne--Faltings data with monomial maps and Deligne--Faltings log schemes with logarithmic morphisms interchangeably.  For instance, we may simply refer to a ``virtual morphism of Deligne--Faltings log schemes $(\uX,\uD,\uE) \to (\uX,\uD,\uE')$''.

\subsubsection{Examples}
The following are some key examples of morphisms between Deligne--Faltings log schemes.

\begin{example}[Linear maps]
The simplest example of an ordinary (respectively, virtual) monomial map is a morphism of vector bundles $\uE \to \uE'$ (resp.\ $\uE|_{\uX\setminus\uD}\to\uE'|_{\uX\setminus \uD})$ that induces an injection of each line $\uL_{i}|_{\uX\setminus\uD}$ into some line  $\uL_{j(i)}'|_{\uX\setminus \uD}$.
\end{example}

\begin{example}[Pairings]
    Let $\uL_1,\uL_2$ and $\uN$ be line bundles on $\uX$, and let
    \[
    \abrac{-,-} \colon \uL_1 \otimes \uL_2 \to \uN
    \]
    be a bilinear pairing that is nondegenerate over $\uX\setminus \uD$.  Then $\abrac{-,-}$ defines an ordinary monomial map $(\uX,\uD,\uL_1\oplus\uL_2) \to (\uX,\uD,\uN)$.  Similarly, a nondegenerate bilinear pairing defined over $\uX\setminus \uD$ (possibly with poles on $\uD$) gives a virtual monomial map.
\end{example}

\begin{example}[Inverting a line bundle]\label{ex:inverting-line-bundle}
Let  $\uL\to \uX$ be a line bundle over a scheme, and form the log schemes $X_+=(\uX,\cM_{X,\varnothing,L})$ and $X_-=(\uX,\cM_{X,\varnothing,L^\vee})$ associated to $\uL$ and its dual. Inverting non-vanishing sections on the fibres produces a virtual isomorphism $X_+\to X_-$ which is not ordinary. 
\end{example}

\begin{example}[Virtual isomorphisms of Deligne--Faltings log structures]\label{ex:automorphisms-DF-log-structures}
Up to virtual isomorphism, the log structure associated to a Deligne--Faltings triple $(\uX,\uD,\uE)$  depends only on the restriction of the split vector bundle $\uE$ to the open set $\uX\setminus \uD$. For instance, let $\uL$ and $\uL'$ be two line bundles on $\uX$ and  consider the Deligne--Faltings log schemes $X=(\uX,\cM_{X,D,L})$ and $X'=(\uX,\cM_{X,D,L'})$. Then any isomorphism 
\begin{equation}\label{eq:iso-line-bundles-in-example}
\uL|_{\uX\setminus \uD}\cong \uL'|_{\uX\setminus \uD}.
\end{equation}
gives a virtual isomorphism $X \cong X'$.  In contrast, an ordinary isomorphism is equivalent to an isomorphism $\uL\cong\uL'$ defined over all of $\uX$.

As an important special case, if $\uL|_{\uX\setminus \uD}$ is trivial, then any trivialization gives rise to a virtual isomorphism
\[
X\cong (\uX,\cM_{X,D})\times \logpt,
\]
of log schemes over $\uX$.  In contrast, these log schemes are isomorphic in the ordinary sense if and only if $\uL$ is trivial over all of $\uX$.
\end{example}

\subsection{Logification}\label{subsec:logification}

The notion of virtual morphism can be naturally extended to pre-log schemes at the cost of a slightly more complicated definition, as follows:

 \begin{definition}\label{def:virtual-morph}
    If $X$ and $Y$ are pre-log schemes, a \defn{virtual morphism $\phi\colon X \to Y$} is a tuple $\phi = (\uphi,\phi_0^*,\phi^*)$ consisting of a morphism of schemes $\uphi \colon \uX \to \uY$ together with a morphism $\phi_0^* \colon \uphi^{-1}\alpha^{-1}_Y(\cOx{Y}) \to \alpha^{-1}_X(\cOx{X})$ of sheaves of monoids and a morphism $\phi^* \colon \uphi^{-1}\cMgp_{Y} \to \cMgp_{X}$ of sheaves of groups,  making the following diagram commute:
    \begin{equation}
    \begin{tikzcd}
        \uphi^{-1}\cOx{Y} \ar[d,"\uphi^*"']& \uphi^{-1}\alpha_Y^{-1}(\cOx{Y}) \ar[l,"\alpha_Y"']\ar[d,"\phi_0^*"]\ar[r,hook] & \uphi^{-1}\cM_Y \ar[r] & \uphi^{-1}\cMgp_Y \ar[d,"\phi^*"] \\
       \cOx{X} & \alpha_X^{-1}(\cOx{X}) \ar[l,"\alpha_X"']\ar[r,hook] & \cM_X \ar[r] & \cMgp_X
    \end{tikzcd}\label{eq:prelog-virtual}
    \end{equation}
\end{definition}

\begin{remark}
    If $\cM_X$ is integral, then $\phi_0^*$ is uniquely determined by $\phi^*$.
\end{remark}
\begin{remark}
    If the pre-log schemes $X$ and $Y$ are log schemes, i.e.\ $\alpha^{-1}\cOx{} \cong \cOx{}$, then the left square in  \eqref{eq:prelog-virtual} is redundant, so that \autoref{def:virtual-morph} reduces to \autoref{def:log-virtual-morph}.
\end{remark}

For a pre-log scheme $X=(X,\cM_X,\alpha_X)$, we denote by $X^{\log}=(X,\cM_X^{\log},\alpha_X^{\log})$ the log scheme obtained by the logification procedure. The morphism $\cM_X\to \cM_X^{\log}$ induces an ordinary morphism of pre-log schemes $X^{\log}\to X$, whose underlying morphism of schemes is $\mathrm{id}_X$.  To understand its interaction with virtual morphisms, we require the following observation, which is an immediate consequence of the universal properties of pushouts and group completions:
\begin{lemma}\label{lem:logification-and-group-completion}
    We have a canonical isomorphism
    \[
    (\cM_X^{\log})^{\gp} \cong \cOx{X} \underset{\alpha_X^{-1}(\cOx{X})}{\sqcup} \cMgp_X
    \]
    commuting with the natural maps from $\cOx{X}$.
\end{lemma}

Now observe that the commutativity of the diagram \eqref{eq:prelog-virtual} defining virtual morphisms implies that a virtual morphisms of pre-log schemes $X \to Y$ induces a morphism of the pushouts describing the group completions of $\cM_X^{\log}$ and $\cM_Y^{\log}$ in \autoref{lem:logification-and-group-completion}.  Thus, by the universal property of pushouts, we obtain the following compatibility of virtual morphisms and logification.

\begin{lemma}\label{prop:logification}
    If $\phi\colon X \to Y$ is a virtual morphism of pre-log schemes, then there is a unique virtual morphism of log schemes $\phi^{\log}\colon X^{\log} \to Y^{\log}$ such that the following induced diagram of virtual morphisms of pre-log schemes commutes.
    \[
    \begin{tikzcd}
        X^{\log} \ar[r,"\phi^{\log}"]\ar[d] & Y^{\log} \ar[d]\\
        X \ar[r,"\phi"] & Y
    \end{tikzcd}
    \]
\end{lemma}
\begin{corollary}\label{cor:log-to-prelog}
If $X$ is a log scheme and $Y$ is a pre-log scheme, then virtual morphisms of log schemes $X\to Y^{\log}$ are naturally in bijection with virtual morphisms of pre-log schemes $X\to Y$.
\end{corollary}

\subsection{Pullback}\label{sec:pullback}
One particularly important application of logification is the construction of pullback log structures. Let $Y=(\uY,\cM_Y,\alpha_Y)$ be a log scheme, and let $\uphi\colon \uX\to\uY$ be a morphism of schemes. The composition
\[
\begin{tikzcd}
\uphi^{-1}\cM_Y \ar[r,"\uphi^{-1}\alpha_Y"] &[.5cm] \uphi^{-1}\cO{Y} \ar[r,"\uphi^*"] & \cO{X}
\end{tikzcd}
\]
is a pre-log structure on $\uX$. The associated log structure is called the \defn{pullback of $\cM_Y$ along $\uphi$}  and denoted by $\uphi^*\cM_Y$.

The map of schemes $\uphi$ then lifts canonically to an ordinary morphism of log schemes $(\uX,\uphi^*\cM_Y)\to Y$ that has the following universal property, which follows immediately from the definitions and \autoref{prop:logification}.
\begin{lemma}\label{lem:pullback}
If $X$ and $Y$ are log schemes, then every virtual morphism $\phi : X \to Y$ factors uniquely through the pullback $\uphi^*\cM_Y$, i.e.~there is a unique virtual morphism $X \to (\uX,\uphi^*\cM_Y)$ making the following diagram commute:
\[
\begin{tikzcd}
    X \ar[rr,"\phi"] \ar[rd,dashed] && Y \\
    &(\uX,\uphi^*\cM_Y)\ar[ur] 
\end{tikzcd}
\]
\end{lemma}

\begin{example}[Pullbacks to divisors]\label{ex:pullback-log-structure-to-divisor}
    Let  $\uphi \colon \uD \hookrightarrow \uX$ be the inclusion of an effective Cartier divisor.  Then the pullback $(\uD,\uphi^*\cM_{X,D})$ is identified with the Deligne--Faltings log scheme $(\uD,\varnothing,\uN)$ associated to the normal bundle $\uN := \cO{\uX}(\uD)|_{\uD}$.  Thus, if $\uE \to \uD$ and $\uF \to \uX$ are split vector bundles, then lifts of the inclusion $\uphi$ to an ordinary (resp.~virtual) morphism $(\uD,\cM_{D,\varnothing, E}) \to (\uX,\cM_{X,D,F})$ are in bijection with ordinary (resp.~virtual) monomial maps $(\uD,\varnothing,\uE) \to (\uD,\varnothing,\uN\oplus \uphi^*\uF)$. 
\end{example}

\subsection{Reduction}
We will need to understand how nilpotent functions interact with virtual morphisms when discussing differential forms. For this, let $X$ be a log scheme and let $X_\red$ denote the log scheme obtained by equipping the underlying reduced scheme $\uX_\red$ with the pullback log structure via the closed embedding $\uX_\red\hookrightarrow\uX$. 
This construction is functorial for virtual morphisms, as follows.

\begin{lemma}
    Let $\phi\colon X \to Y$ be a virtual morphism of log schemes.  Then there is a unique virtual morphism $\phi_\red \colon X_\red \to Y_\red$ making the following diagram commute:
    \[
    \begin{tikzcd}
        X_\red \ar[r,"\phi_\red"]\ar[d,hook] & Y_\red \ar[d,hook]\\
        X \ar[r,"\phi"] & Y
    \end{tikzcd}
    \]
\end{lemma}

\begin{proof}
    Recall that an element $f \in \cO{}$ is invertible if and only if its image in the reduced ring $\cO{\red}$ is invertible. It follows that for the induced pre-log structure $\alpha_\red\colon \cM \to \cO{\red}$ on the reduction, we have $\alpha_\red^{-1}(\cOx{\red}) = \alpha^{-1}(\cOx{})$. Thus a virtual morphism $\phi\colon X \to Y$ induces a virtual morphism of the reduced pre-log structures, 
    giving the desired morphism $X_\red \to Y_\red$ by logification via \autoref{prop:logification}.
\end{proof}

\subsection{Products}

Let us recall the classical construction of fibre products in the category of log schemes and ordinary morphisms. Let $a_X\colon X\to S$ and $a_Y\colon Y\to S$ be ordinary morphisms of log schemes. We form the fibre product of the underlying morphisms of schemes, as follows:
\[
\begin{tikzcd}
\uX\times_{\uS}\uY \ar[r,"\underline{p}_X"]\ar[d,"\underline{p}_Y"'] \ar[rd,"\underline{p}"]&[.6cm] \uX \ar[d,"\underline{a}_X"] \\[.4cm]
\uY \ar[r,"\underline{a}_Y"'] & S
\end{tikzcd}
\]
On $\uX\times_{\underline{S}}\uY$, we have the pullback log structures from $X$, $Y$, and $S$, and we can consider the following pushout of sheaves of monoids on $\uX\times_{\underline{S}}\uY$.
\[
\begin{tikzcd}
    \underline{p}^*\cM_S \ar[r,"a_X^*"] \ar[d,"a_Y^*"'] &[.6cm] \underline{p}_X^*\cM_X \ar[d,dashed] \\[.6cm]
    \underline{p}_Y^*\cM_Y \ar[r,dashed] & \cM_{X\times_S Y}
\end{tikzcd}
\]
By \cite[III, Proposition 1.1.3, Proposition 2.1.2]{Ogus}, this produces a log structure on $X\times_S Y$, and (ordinary) morphisms of log schemes $p_X\colon X\times_S Y\to X$ and $p_Y\colon X\times_S Y\to Y$, making $X\times_S Y$  the fibre product of $X$ and $Y$ in $\LogSch$. It is also the fibre product in the category of virtual morphisms:

\begin{lemma}\label{lem:products}
Let $X\to S$ and $Y\to S$ be ordinary morphisms of log schemes. Then $X\times_S Y$ is the fibre product of $X$ and $Y$ over $S$ in the category $\WLogSch$.
\end{lemma}

If $S$ has trivial log structure, then the structure maps $X \to S$ of all objects of $\WLogSch_S$ are ordinary by \autoref{lem:virtual-mor-to-scheme}, and hence we have the following.
\begin{corollary}
    If $S$ is an ordinary scheme, then the category $\WLogSch_S$ has all finite products.
\end{corollary}

\begin{proof}[Proof of \autoref{lem:products}]
    Let $\phi \colon Z \to X$ and $\psi \colon Z \to Y$ be virtual morphisms of log schemes such that $a_X\circ \phi=a_Y\circ\psi$. We consider the following diagram of virtual morphisms of log schemes.  We must construct a virtual morphism $\xi$, making the  diagram commute:
    \[
    \begin{tikzcd}
    Z \ar[dr,dashed, "\xi"] \ar[drr,bend left=30,"\phi"] \ar[ddr,bend right=30,"\psi"'] && \\
    & X\times_S Y \ar[r,"p_X"] \ar[d,"p_Y"'] &[.5cm] X \ar[d,"a_X"] \\[.4cm]
    & Y \ar[r,"a_Y"']& S
    \end{tikzcd}
    \]
    To this end, let $\underline{\xi}=(\uphi,\upsi)\colon \uZ\to \uX\times_{\underline{S}}\uY$ be the morphism of schemes induced by $\uphi$ and $\upsi$.      The morphisms 
    \[
\underline{\xi}^*\underline{p}_X^*\cM_X=\uphi^*\cM_X\stackrel{\phi^*}{\longrightarrow} \cMgp_Z \quad \mbox{ and } \quad \underline{\xi}^*\underline{p}_Y^*\cM_Y=\upsi^*\cM_Y\stackrel{\psi^*}{\longrightarrow} \cMgp_Z
    \]
    induce a morphism $\xi^*\colon\underline{\xi}^*\cM_{X\times_S Y}\to \cMgp_Z$ because for every section $m$ of $\underline{p}^*\cM_S$ we have $\phi^*a_X^*(m)=\psi^*a_Y^*(m)$ by assumption. The fact that $\xi=(\underline{\xi},\xi^*)$ is a virtual morphism of log schemes follows from the fact that $\phi$ and $\psi$ are.
    One easily checks that it is the unique such virtual morphism, and the claim follows.
\end{proof}

We do not know whether all fibre products exist in $\WLogSch$. However, even when they do, the underlying scheme of the fibre product is not always the fibre product of the underlying schemes, as the following elementary example shows. This is a clear departure from the case of ordinary morphisms treated above.

\begin{example}
Recall from \autoref{ex:basepoints-of-log-point} that virtual morphisms $\pt \to \logpt$ over $\KK$ are in bijection with invertible scalars $\KK^\times$.  Let $X,Y=\pt$, equipped with the virtual morphisms to $S=\logpt$ given by scalars $\lambda,\mu \in \KK^\times$, with $\lambda\neq \mu$. Given a commutative diagram of virtual morphisms log schemes
\[
\begin{tikzcd}
Z \ar[r] \ar[d] & * \ar[d] \\[0.5cm]
* \ar[r] & \logpt
\end{tikzcd}
\]
one must have the equality $\lambda=\mu$ in $\sect{Z,\mathcal{O}_Z^\times}$, and hence $Z=\varnothing$. The fibre product of $X$ and $Y$ over $S$ therefore exists, but is empty:
\[
X\times_S Y = \varnothing.
\]
Meanwhile,  the fibre product of the underlying schemes is given by
\[
\uX \times_{\uS}\uY = \pt \times_\pt \pt \cong \pt,
\]
which is non-empty.
\end{example}

\subsection{Continuity principle}
When dealing with virtual morphisms, it is sometimes convenient to make arguments in which one uses the axiom of a virtual morphism to check a property on the locus where some function is invertible, and then deduces that the property must actually hold everywhere by continuity.  This typically works best when the scheme is reduced and irreducible, i.e.~integral, to eliminate the possibility of functions  such as $\epsilon \in \KK[\epsilon]/(\epsilon^2)$ or $x \in \KK[x,y]/(xy)$ that are neither ``generically zero'' nor ``generically invertible''.  We encapsulate this in the following.

\begin{lemma}[Continuity principle]\label{lem:continuity-principle}
    Let $\phi \colon X \to Y$ be a virtual morphism of log schemes, where $\uX$ is integral, and let $\eta \in \uX$ be the generic point.  Then for every $f \in \sect{Y,\cM_Y}$, either
    \[
    \uphi^*\alpha_Y(f) = 0 
    \]
    or $f \in \cOx{Y,\uphi(\eta)}$, and 
    \[
    \alpha_X(\phi^*f) = \uphi^*\alpha_Y(f)
    \]
    as elements of $\cO{}(X)$.
\end{lemma}
\begin{proof}
The local ring $\cO{X,\eta}$ is a field, so if $\uphi^*\alpha_Y(f)$ is nonzero, it is a unit in $\cO{X,\eta}$.  Therefore $\alpha_Y(f)$ must be a unit in $\cO{Y,\uphi(\eta)}$. 
\end{proof}

The following examples shows that the hypothesis of integrality is essential:
\begin{example}\label{ex:fat-log-point}
    Let $\KK$ be a field and let $X := \Spec{t^\NN \to k[\epsilon]/(\epsilon^2);t\mapsto \epsilon}$,  be the log scheme over $\KK$ obtained by pulling back the standard log structure on $\logline$ to the first-order neighbourhood of $0$.  Note that  $\epsilon \notin \cOx{X}$.  Hence, for any $a,b \in \KKx$, there is a unique virtual morphism $\phi_{a,b} \colon X \to X$ such that
  \begin{align*}
      \uphi_{a,b}^*(\epsilon) &= a\epsilon &\phi_{a,b}^*(t) = bt.
  \end{align*}
   Note that $\phi_{a,b}$ is an ordinary morphism if and only if $a=b$.   
   On the other hand, if $a \neq b$, the conclusion of the continuity principle (\autoref{lem:continuity-principle}) fails, since $\uphi_{a,b}^*\alpha(t) = \uphi_{a,b}^*\epsilon = a\epsilon$ is nonzero, but is also not equal to $\alpha(\phi_{a,b}^*(t)) = \alpha(bt) = b\epsilon$.
\end{example}

\begin{example}
    Let $Y = \Spec{t^\NN \hookrightarrow k[z];t\mapsto z}$ be the log line and let $X = \Spec{w^\NN \hookrightarrow \KK[x,y]/(xy);w\mapsto x-y}$. Geometrically, $X$ is the log scheme given by equipping the coordinate axes in $\AF^2$ with the divisorial log structure associated to the pullback of the diagonal $\{x=y\} \subset \AF^2$.  
    
    Consider the maps of rings and monoids defined by
    \begin{align*}
    \uphi^* z &= x \in \cO{X} & \phi^* t = w \in \cM_X
    \end{align*}
    Note that we have
    \begin{align*}
    \uphi^*\alpha_Y(t) &= \uphi^*z = x & \alpha_X\phi^*(t) = \alpha_X(w) = x-y.
    \end{align*}
    We claim that $\phi$ is a virtual morphism.  Indeed, for this we need only check that $\alpha_X\phi^*(t)=\uphi^*\alpha_Y(t)$ over the locus where $t \in \cOx{Y}$.  But this locus is precisely the open set $U:= \{z \neq 0\} \subset \AF^1$, whose preimage $V := \phi^{-1}(U) \subset X$ is the $x$-axis punctured at the origin.  Hence $y$ vanishes on $V$ so that $\uphi^*\alpha_Y(t)$ and $\alpha_X\phi^*(t)$ agree there.

    Now observe that  $X$ has two generic points $\eta_1$ and $\eta_2$, corresponding to the $y$-axis (where $x=0$) and the $x$-axis (where $y=0$), respectively.  At $\eta_1$ we have $\uphi^*\alpha_Y(t) = 0 \neq \alpha_X\phi^*(t)$, while at  $\eta_2$ we have $\alpha_X\phi^*t = w = x = \uphi^*\alpha_Y z$, so that $\eta_1$ and $\eta_2$ exhibit the two distinct possibilities in the continuity principle.  
\end{example}

\section{Tangential basepoints as virtual points}
\label{sec:tangential-basepoints}

In this section, we connect the notion of virtual morphism with Deligne's notion of tangential basepoint from \cite[\S15]{Deligne1989}. A similar discussion in the setting of manifolds with corners can be found in \cite{DPP:Integrals}. Throughout this section, we consider a pair $(\uX,\uD)$ where $\uX$ is a regular scheme over a field $\KK$ and $\uD\subset \uX$ is a normal crossing divisor, which for simplicity we assume to be strict. We equip $\uX$ with the corresponding divisorial log structure and denote the resulting log scheme by $X$. It is such that $\alpha_X\colon \cM_X\to \cO{X}$ is the inclusion of the subsheaf of regular functions that are invertible on $\uX\setminus \uD$.  Recall that if $j\colon \uX\setminus \uD\hookrightarrow \uX$ denotes the open immersion, then $\cMgp_X = j_*\cOx{X\setminus D}$ is the sheaf of rational functions that are regular and invertible on $\uX \setminus \uD$.

\subsection{Tangential basepoints as virtual morphisms}
The divisor $\uD$ gives a stratification
\[
\uX = \uX_0 \sqcup \uX_1 \sqcup \uX_2 \sqcup \cdots 
\]
by regular locally closed subschemes, where $\uX_j$ is the locus where $j$ irreducible components of $\uD$ meet. 
If $x \in \uX_j$, then $\uD$ is defined locally by an equation of the form $z_1\cdots z_j = 0$ where $z_1,\ldots,z_j \in \cO{X,x}$ are functions whose differentials are linearly independent at $x$, and $\uX_j$ is locally identified with the locus where $z_1=\cdots =z_j=0$.
 
Associated to any $\KK$-point $x \in \uX(\KK)$ is its \defn{normal space} $N_x(\uX,\uD)$, defined as the Zariski normal space of the unique stratum through $x$, i.e.~
\[
N_x(\uX,\uD) := \tb[x]\uX / \tb[x]\uX_j \quad \mbox{ where } x \in \uX_j(\KK). 
\]
 Let $\uD_1,\ldots,\uD_j$ denote the local branches of $\uD$ passing through $x$, and let $N_{x}(\uX,\uD_i) = \tb[x]{\uX}/\tb[x]{\uD_i}$ denote their normal spaces, each of which is a one-dimensional vector space over $\KK$.  Then the natural projection
\begin{align}
    N_x(\uX,\uD) \to \bigoplus_{i=1}^j N_{x}(\uX,\uD_i) \label{eq:normal-splitting}
\end{align}
is an isomorphism.

\begin{definition}
    A \defn{normal vector of $(\uX,\uD)$ at $x$} is an element of the normal space $N_x(\uX,\uD)$.  It is called \defn{inward pointing} if its projection to $N_{x}(\uX,\uD_i)$ is nonvanishing for every local branch $\uD_i$ of $\uD$ at $x$.  We denote the set of inward-pointing normal vectors by $N^\circ_x(\uX,\uD)$.
\end{definition}

In local coordinates as above, $\uD_i$ is given by the equation $\{z_i=0\}$ and a normal vector of $(\uX,\uD)$ at $x$ has the form
\begin{equation}\label{eq:normal-vector}
v=v_1\partial_{z_1}|_x+\cdots +v_j\partial_{z_j}|_x
\end{equation}
with each $v_i\in\KK$. It is inward-pointing if and only if $v_i\in\KK^\times$ for all $i$. More intrinsically, the isomorphism \eqref{eq:normal-splitting} implies that the set of inward-pointing normal vectors of $(\uX,\uD)$ at $x$ is a torsor for the group $(\KK^\times)^j$.

 \begin{definition}\label{def:tangential-basepoints}
     A \defn{tangential basepoint} of $(\uX,\uD)$ is a pair $(x,v)$ where $x \in \uX(\KK)$ is a $\KK$-point and $v \in N^\circ_x(\uX,\uD)$ is an inward-pointing normal vector based at $x$.
 \end{definition}

\begin{remark}
For $x\in \uX_0(\KK)=\uX(\KK)\setminus \uD(\KK)$, we are in the $j=0$ case and there is a unique normal vector of $(\uX,\uD)$ at $x$, which is inward-pointing. Therefore a $\KK$-point of $\uX\setminus \uD$ is a special case of a tangential basepoint of $(\uX,\uD)$.
\end{remark}

\begin{example}
Let $\uX=\AF^1$ with coordinate $z$, and $\uD=\{0\}$. A tangential basepoint of $(\AF^1,\{0\})$ is either a $\KK$-point of $\AF^1\setminus\{0\}$ or a non-zero tangent vector at $0$, denoted by $c\,\partial_z|_0$ with $c\in\KK^\times$.
\end{example}

Recall that $X$ denotes the log scheme over $\KK$ defined by the pair $(\uX,\uD)$. Virtual morphisms are a natural geometric language to talk about tangential basepoints, in the following sense.

\begin{theorem}\label{thm:tangential-basepoints}
    Virtual morphisms $\pt \to X$ of log schemes over $\KK$ are in canonical bijection with tangential basepoints of $(\uX,\uD)$.
\end{theorem}

\begin{proof}
According to \autoref{lem:pullback}, a virtual morphism $*\to X$ is the same thing as the data of a rational point $x\in \uX(\KK)$, viewed as a morphism of $\KK$-schemes $x\colon \Spec{\KK}\to \uX$, along with a virtual morphism from $*$ to the pullback of the log structure,
\begin{equation}\label{eq:virtual-morphism-from-proof-tangential-basepoint}
(\Spec{\KK},\cOx{\Spec{\KK}}\to \cO{\Spec{\KK}}) \longrightarrow (\Spec{\KK},\cM_X|_x\to \cO{\Spec{\KK}}),
\end{equation}
whose underlying morphism of schemes is the identity of $\Spec{\KK}$. Since the normal crossing divisor $\uD$ is strict, one sees that the log structure $\cM_X|_x$ is constant, i.e. induced by a constant pre-log structure. Indeed, if $z_1,\ldots,z_j$ are local defining equations for the components of $\uD$ near $x$ as above, then
\[
\sect{\Spec{\KK},\cM_{X}|_x} \cong \KK^\times z_1^\NN z_2^\NN\cdots z_j^\NN
\]
and $\cM_X|_x$ is induced by the constant pre-log structure
$$\NN^j \to \sect{\Spec{\KK},\cM_X|_x} \; , \; (k_1,\ldots,k_j)\to z_1^{k_1}\cdots z_j^{k_j}.$$
It is therefore (non-canonically) isomorphic to $(\logpt)^j$. Thus, following \autoref{ex:virtual-mor-to-log-pt}, a virtual morphism \eqref{eq:virtual-morphism-from-proof-tangential-basepoint} is equivalent to the datum of a group homomorphism 
\begin{equation}\label{eq:virtual-morphism-from-proof-tangential-basepoint-global}
\sect{\Spec{\KK},\cMgp_X|_x} \to \KK^\times
\end{equation}
which acts as the identity on $\KK^\times$. This is given by a collection of $j$ non-zero scalars $(v_1,\ldots,v_j)\in (\KK^\times)^j$, with $v_i$ the image of $z_i$.  Associating to this tuple the tangential basepoint $v_1\partial_{z_1}|_x+\cdots +v_j\partial_{z_j}|_x$ gives a bijection between virtual morphisms $*\to X$ and tangential basepoints, as claimed.

To check that the bijection is independent of choices of local coordinates for $\uD$, we recast it in a more coordinate-free way as follows. For a point $x\in \uX(\KK)$, write $N_x:=N_x(\uX,\uD)$ and $N_{x,i}:=N_x(\uX,\uD_i)$ for each local branch $\uD_i$ of $\uD$ at $x$. The decomposition \eqref{eq:normal-splitting} gives a dual decomposition $N^\vee_x \cong \bigoplus_i N^\vee_{x,i}$ of the conormal space, and we denote by 
\[
M_x \subset \mathrm{Sym}_\KK(N_x^\vee) = \cO{}(N_x)
\]
the monoid of regular functions on $N_x$ that are monomial with respect to this decomposition, i.e.\ which can be written as products of element of $\KK^\times$ and invertible elements of the conormal lines $N^\vee_{x,1},\ldots,N^\vee_{x,j}$ of the divisor components.  In coordinates,
\[
M_x = \KK^\times (\dd z_1|_x)^{\NN} \cdots (\dd z_j|_x)^{\NN}.
\]
The morphism of monoids $\alpha\colon M_x \to \KK$ given by the evaluation at the origin in $N_x$, i.e.\ by $\alpha(\dd z_i|_x)=0$, induces a (constant) log structure on the point $\Spec{\KK}$, non-canonically isomorphic to $(\logpt)^j$. If $f$ is a section of $\cM_X$ in a neighbourhood of $x$, then its leading Taylor monomial in the normal direction is naturally an element of $M_x$, and this gives a canonical isomorphism
\[
\sect{\Spec{\KK},\cM_{X}|_x} \cong M_x
\]
of log structures on $\Spec{\KK}$; it is given concretely by
\[
\lambda \, z_1^{k_1} \cdots z_j^{k_j} \mapsto \lambda (\dd z_1|_x)^{k_1} \cdots (\dd z_j|_x)^{k_j} 
\]
for $\lambda\in\KK^\times$ and $k_1,\ldots,k_j \ge 0$. Therefore, a morphism \eqref{eq:virtual-morphism-from-proof-tangential-basepoint-global} is equivalent to a morphism of monoids $M_x\to \KK^\times$ which acts as the identity on $\KK^\times$. Clearly, such a morphism of monoids is equivalent to a linear map $N_x^\vee\to \KK$ whose restriction to each conormal line is non-zero. This is the same thing as a tangential basepoint of $(\uX,\uD)$ at $x$.
\end{proof}

\subsection{More general tangential basepoints as virtual points}\label{sec:general-basepoints}
The term ``tangential basepoint'' is typically used in the context of a (strict) normal crossing divisor in a smooth variety over a field.  \autoref{thm:tangential-basepoints} suggests that ``virtual morphisms from a point'' give a natural notion of tangential basepoints in more general situations.  For instance, this notion makes sense for a more general base, for normal crossing divisors that are not necessarily strict, and for varieties that are singular. We therefore make the following general definition.

\begin{definition}[Virtual point]\label{def:virtual-point}
Let $\pi\colon X\to S$ be a virtual morphism of log schemes. A \defn{virtual $S$-point} of $X$ is a virtual morphism $x\colon S\to X$ such that $\pi\circ x=\id{S}$. We denote by $X(S)$ the set of virtual $S$-points of $X$, or equivalently:
$$X(S) = \mathrm{Hom}_{\WLogSch_S}(S,X).$$
\end{definition}

We note that virtual points behave as expected with respect to fibre products (of ordinary morphisms) thanks to \autoref{lem:products}.

\begin{corollary}\label{cor:monoidal-functor-of-points}
For ordinary morphisms $X,Y \to S$, we have bijections
\[
(X\times_S Y)(S)\cong X(S)\times Y(S).
\]
that are natural in $X,Y$ and $S$.
\end{corollary}

We now illustrate the notion of virtual point with some examples.

\subsubsection{Examples of tangential basepoints over $\ZZ$}\label{sec:basepoints-on-P1}

    Let $X$ be the divisorial log scheme over $\ZZ$ associated to the pair $(\uX,\uD) := (\PP^1_\ZZ,\{0,1,\infty\})$, which is the setting in which tangential basepoints were first defined \cite{Deligne1989}. 
    Note that $\uX\setminus\uD=\PP^1\setminus\{0,1,\infty\}$ does not have any $\ZZ$-points.  Let $z$ be the standard coordinate on $\AF^1 = \PP^1 \setminus \{\infty\}$, and let $w = z^{-1}$ be the standard coordinate at infinity.  Since $\ZZ^\times = \{\pm 1\}$, one sees by adapting the proof of \autoref{thm:tangential-basepoints} that $X$ has exactly two virtual $\ZZ$-points at each of the three points $0,1,\infty \in \PP^1(\ZZ)$, namely the ``unit'' tangential basepoints
    \begin{equation}\label{eq:six-tangential-basepoints}
    \pm \cvf{z}|_{0} \in \tb[0]{\PP^1_\ZZ} \qquad \pm \cvf{z}|_{1} = \mp \cvf{w}|_1 \in \tb[1]{\PP^1_\ZZ} \qquad \pm \cvf{w}|_\infty \in \tb[\infty]{\PP^1_\ZZ}.
    \end{equation}
    These form a torsor for the automorphism group $\mathrm{Aut}(X)\cong \mathfrak{S}_3$. The following lemma says that there are no other tangential basepoints for $X$.

    \begin{lemma}\label{lem:virtual-Z-points-of-projective-line-three-points}
    The $6$ tangential basepoints \eqref{eq:six-tangential-basepoints} are the only virtual $\ZZ$-points of the divisorial log scheme $(\PP^1_\ZZ,\{0,1,\infty\})$.
    \end{lemma}

    \begin{proof}
    Assume that there exists a virtual morphism $s:\Spec{\ZZ}\to (\PP^1_\ZZ,\{0,1,\infty\})$ whose underlying morphism $\underline{s}:\Spec{\ZZ}\to \PP^1_\ZZ$ is different from $0$, $1$, $\infty$, and denote this point by $[a:b]$ with $a,b\in \ZZ$ coprime and $a,b,a-b\neq 0$. Let $S$ be the finite set of prime numbers $p$ such that $p$ divides one of $a, b, a-b$, those three cases being mutually exclusive. One easily sees that $S$ is non-empty (this is the same thing as saying that $\PP^1_\ZZ\setminus \{0,1,\infty\}$ does not have any $\ZZ$-points). Geometrically, $S$ is the set of closed points of $\Spec{\ZZ}$ over which $\underline{s}$ intersects one of the three sections $0,1,\infty$ of the structure morphism $\PP^1_\ZZ\to \Spec{\ZZ}$. Via $\underline{s}$, the divisorial log structure $(\PP^1_\ZZ,\{0,1,\infty\})$ pulls back to the divisorial log structure $(\Spec{\ZZ},S)$, and $s$ factors through a virtual morphism of log schemes $f:\Spec{\ZZ}\to (\Spec{\ZZ},S)$ whose underlying morphism of schemes is the identity. Consider a prime $p\in S$, viewed as a global section of the structure sheaf of monoids for $(\Spec{\ZZ}, S)$. The element $f^*(p)$ is necessarily $\pm 1$, but because $p$ is invertible outside of the closed point $p$, we necessarily have $f^*(p)=p$ in the localization $\ZZ_{(p)}$. This is a contradiction, and the claim follows. 
    \end{proof}

\subsubsection{Examples of tangential basepoints for a non-strict normal crossing divisor}

Let $X$ be the log scheme over $\RR$ associated to the pair $(\uX,\uD)$ where $\uX=\mathbb{A}^2_\RR$ with coordinates $(u,v)$ and $\uD=\{u^2+v^2=0\}$ in $\uX$. It is a normal crossing divisor that only becomes strict after base change to $\CC$ since $\uD_\CC=\{(u+v\iu)(u-v\iu)=0\}$, where $\iu$ is a fixed square root of $-1$. 
Let us classify virtual $\RR$-points of $X$ whose underlying point of $\uX$ is $x=(0,0)$. The log structure $\cM_X|_x$ has sections on $\Spec{\RR}$ and $\Spec{\CC}$ given by
\[
\RR^\times (u^2+v^2)^\NN \quad \mbox{ and } \quad \CC^\times (u+v\iu)^\NN(u-v\iu)^\NN
\]
respectively. It is therefore not induced by a constant pre-log structure as in the proof of \autoref{thm:tangential-basepoints}. A tangential basepoint for $(\uX,\uD)$ at $x$ is equivalent to a $\operatorname{Gal}(\CC/\RR)$-equivariant morphism of monoids $\CC^\times (u+v\iu)^\NN(u-v\iu)^\NN\to \CC^\times$ acting as the identity on $\CC^\times$, i.e., is a complex tangential basepoint
\[
\lambda \, \partial_{u+v\iu}|_{(0,0)} + \mu \, \partial_{u-v\iu}|_{(0,0)} \quad \mbox{ with } \lambda=\overline{\mu}\in\CC^\times.
\]
Writing $\mu=\frac{1}{2}(a+b\iu)$, one can rewrite the latter as the real tangent vector
\[
a\,\partial_u|_{(0,0)} + b\,\partial_v|_{(0,0)} \quad \mbox{ with } (a,b)\in\RR^2\setminus\{(0,0)\}.
\]
Therefore, we see that a virtual $\RR$-point of $X$ at $x$ is the same thing as a non-zero vector in $T_xX=\RR^2$, as one could have expected.

\subsubsection{Examples of tangential basepoints for singular varieties}

Let $\uX$ be a variety (possibly singular) over a field $\KK$ and $\uD \subset \uX$ be a reduced effective Cartier divisor, and let $X$ denote the log scheme over $\KK$ associated to the pair $(\uX,\uD)$.
Note that for a point $x \in \uX(\KK)$, the pullback $\cM_X|_x$ is generated over $\KKx$ by \'etale local equations for the branches of $\uD$, which may also be viewed as generators of the conormal space of the stratum of $\uD$ through $x$ (i.e.~the intersection of all \'etale-local components of $\uD$ passing through $x$).  Hence a virtual
point of $X$ at $x$ is equivalent to a normal vector to each branch of $\uD$ through $x$, potentially satisfying some constraints.

\begin{example}
Let $\uX = \{uv=w^2\} \subset \AF^3$ and let $\uD\subset \uX$ be the divisor given by the vanishing of $uv$, or equivalently $w^2$, and $x = (0,0,0)$.  Then $\cM_{X}|_x$ is generated by the restrictions $u_x,v_x,w_x$  of $u,v,w$, subject to the relation $u_xv_x=w_x^2$.  
A virtual point of $X$ at $x$ is thus equivalent to a tangent vector
\[
a \,\cvf{u}|_x + b\,\cvf{v}|_x + c\,\cvf{w}|_x\in \tb[x]{\uX},
\]
where $a,b,c \in \KKx$ are nonzero constants such that $ab=c^2$.  In other words, a tangential basepoint at the origin in $(\uX,\uD)$ is a point in the tangent cone of $x \in \uX$ that does not lie in the tangent cone of $x \in \uD$.
\end{example}

\begin{example}
Let $L_1,L_2,L_3 \subset \AF^2$ be distinct lines through the origin, cut out by the vanishing of linear forms $\ell_1,\ell_2$ and $\ell_3$, respectively.  Let $\uD = L_1+L_2+L_3$ be the divisor given by their union, $X$ the log scheme over $\KK$ associated to the pair $(\AF^2,\uD)$, and $x$ the point $(0,0)$. Then $\cM_X|_x = \KKx \ell_1^\NN \ell_2^\NN\ell_3^\NN$ is the monoid freely generated over $\KKx$ by the germs of $\ell_1,\ell_2$ and $\ell_3$.   A virtual point of $X$ at $x$ is thus equivalent to a tangential basepoint relative to each irreducible component of $\uD$. The set of tangential basepoints at the origin is therefore given by $\prod_{j=1}^3 N^\circ_x(\AF^2,L_j)$, which is non-canonically isomorphic to  $(\KKx)^3$; in particular, it is three-dimensional.

Note, in contrast, that for a normal crossing divisor in a smooth surface, the maximal dimension of a space of tangential basepoints through any point is two; hence not every 
virtual point of $X$ can be lifted to a normal crossing resolution of $(\uX,\uD)$.  
\end{example}

\subsection{Inclusions of strata}\label{sec:strata}

Higher-dimensional analogues of tangential basepoints for a strict normal crossing divisor $(\uX,\uD)$ are obtained by considering higher-dimensional strata in $\uX$ rather than individual points, as follows.

Let $i\colon \uY \to \uX$ be a locally closed immersion, given locally by the inclusion of an intersection of components of $\uD$.  Note that the other components of $\uD$ then induce a normal crossing divisor $\uD_Y \subset \uY$.  We thus have two natural log structures on the $\KK$-scheme $\uY$:
\begin{definition}
    We denote by $Y = (\uY,\cM_Y,\alpha_Y)$ the divisorial log scheme associated to the normal crossing divisor $\uD_Y \subset \uY$, and  by $\widehat{X}^\circ_Y$ the log scheme defined by the pullback log structure $\cM_{\widehat{X}^\circ_Y} := i^*\cM_X$ on $\uY$.
\end{definition}
As the notation is meant to suggest, the log scheme $\widehat{X}^\circ_Y$ plays the role of a punctured tubular neighbourhood of $\uY$ in $\uX$.  More precisely, the normal bundle $N_{\uY} \uX = i^*\tb[\uX]/\tb[\uY]$ comes equipped with a canonical normal crossing divisor $N_{\uY}\uD \subset N_{\uY}\uX$, namely the normal cone of $\uD$ along $\uY$.  This divisor defines a smooth log scheme $N^\circ_Y X$, and the restriction of its log structure to the zero section is canonically identified with $\cM_{\widehat{X}^\circ_Y}$, by extracting the leading Taylor monomials as in the previous subsection.  Put differently, $N_{\uY}\uX$ carries a canonical splitting as a sum of the normal bundles of the branches of $\uD$ containing $\uY$, and $\widehat{X}^\circ_Y$ is the Deligne--Faltings log structure associated to the triple $(\uY,\uD_{\uY},N_{\uY}\uX)$.

These relationships are summarized by the following diagram of ordinary morphisms of log schemes in $\LogSch_\KK$:
\begin{equation}
\begin{tikzcd}
    N_{\uY}\uX \setminus N_{\uY}\uD \ar[r,hook]\ar[d,twoheadrightarrow] & N^\circ_Y X \ar[d,twoheadrightarrow]& \widehat X^\circ_Y \ar[l,hook']\ar[r,hook] & X & \uX \setminus \uD \ar[l,hook'] \ar[d,hook']\\
    \uY \setminus \uD_{Y}\ar[r,hook] & Y\ar[rrr,hook] &&& \uX
\end{tikzcd}\label{eq:punctured-neighbourhood}
\end{equation}
Using that virtual morphisms factor uniquely through pullback log structures, and maps of divisorial log schemes are given by maps of the underlying schemes that respect the interiors, we deduce the following.

\begin{proposition}\label{prop:virtual-inclusions-of-strata}
Lifts of $i\colon \uY \to \uX$ to a virtual morphism $Y \to X$ are in bijection with inward pointing sections of the normal bundle of $\uY\setminus\uD_{\uY}$.
\end{proposition}
In other words, such a morphism is the assignment of a tangential basepoint to every point of $\uY\setminus \uD_Y$, in a regularly varying fashion.

\section{Kato-Nakayama spaces and Betti cohomology}
\label{sec:Betti}

In this section, we discuss the interaction of virtual morphisms and tangential basepoints with Kato--Nakayama's logarithmic analogue of the space of complex points equipped with the classical analytic topology \cite{KatoNakayama}.

\subsection{The associated topological space}

 We will be interested in log schemes $X$ that have the following  further property.  Recall that a monoid is $M$ is \defn{fs} (for ``\defn{fine and saturated}'') if it is finitely generated, integral ($fg=fh\implies g=h$ holds for all $f,g,h\in M$, i.e.~$M$ embeds inside its group completion $M^{\mathrm{gp}}$), and saturated (for $f\in M^{\mathrm{gp}}$ and $n\ge 1$, $f^n\in M \implies f\in M$). Recall that a log scheme $X$ is \defn{fs} if (\'etale) locally the log structure is the logification of a constant pre-log structure $M_X\to\cO{X}$ for an fs monoid $M$. The log point and the log line (\autoref{ex:log-line} and \autoref{ex:log-point}) have this property, as do divisorial and Deligne--Faltings log structures, (\autoref{ex:divisorial-log-structures} and \autoref{defi:DF-log-structures}). We denote by
\[
\fsWLogSch_\KK \subset \WLogSch_\KK
\]
the full subcategory of fs log schemes of finite type over $\KK$ and virtual morphisms.

In \cite[Section 1]{KatoNakayama} Kato--Nakayama associate, to any $X\in \fsWLogSch_\CC$, a topological space that we denote\footnote{Kato--Nakayama use the notation $X^{\log}$, which in this article denotes the logification of a pre-log scheme; see \autoref{subsec:logification}.} by $\KN{X}$; see also \cite[V.1]{Ogus}. We recall the relevant aspects of the construction.

\begin{definition}
    A \defn{KN-point of $X$} is a pair $(x,\lambda)$ where $x \in \uX(\CC)$ is a complex point of $X$ and  $\lambda\colon \cMgp_{X,x} \to \unitcirc$ is a  homomorphism to the unit circle $\unitcirc \subset \CC$, such that
    \begin{align}
    \lambda(f) = \frac{f(x)}{|f(x)|} \label{eqn:KN-point}
    \end{align}
    for all $f \in \cOx{X,x}$. 
  The set of KN-points of $X$ is called the \defn{Kato-Nakayama space of $X$} and denoted by $\KN{X}$.
\end{definition}

There is a natural map of sets
    \begin{align}
    \begin{array}{rccc}
    \tau = \tau_X \colon & \KN{ X} &\to &\uX(\CC) \\
     & (x,\lambda) & \mapsto& x.
    \end{array}
    \label{eq:def-tau}
    \end{align}
 If $f \in \cM_X$ is a local section with domain $U \subset \uX(\CC)$, then we may define functions $|f|\colon \tau^{-1}(U) \to \halfspc{}$ and $\arg f\colon \tau^{-1}(U) \to \unitcirc$ by the formulae
    \[
    |f|(x,\lambda) := |\alpha(f)(x)| \qquad (\arg f)(x,\lambda) := \lambda(f).
    \]
    Then we equip $\KN{X}$ with the coarsest topology for which the map $\tau$  and the locally defined functions  $\arg f$ are continuous for all $f \in \cM_X$.  Note that the functions $|f|$ are then automatically continuous, since they are given by the absolute value of the continuous $\CC$-valued function $\alpha(f)$ on $\uX(\CC)$

    Note that since the condition \eqref{eqn:KN-point} only makes reference to invertible functions, it is naturally compatible with the notion of virtual morphism. 
 Namely, if $X,Y \in \fsWLogSch_\CC$, and $\phi \colon X \to Y$ is a virtual morphism, we obtain a map of sets
    \[
    \mapdef{\KN{\phi}}{\KN{X}}{\KN{Y}}
    {(x,\lambda)}{(\uphi(x),\lambda\circ\phi^*)}
    \]
    such that $\KN{\phi}^*\arg(f) = \arg(\phi^*f)$ for all $f \in \cM_Y$, and the following diagram commutes:
    \[
    \begin{tikzcd}
        \KN{X} \ar[d]\ar[r,"\KN{\phi}"] &\KN{Y}\ar[d] \\
 \uX(\CC)\ar[r,"\uphi"]&\uY(\CC).
    \end{tikzcd}
    \]
    It follows immediately that $\KN{\phi}$ is a continuous map.  Furthermore, it is evident that $\KN{\id{X}} = \id{\KN{X}}$ and $\KN{\phi \circ \psi} = \KN{\phi}\circ\KN{\psi}$.
\begin{proposition}
    There is a canonical functor from $\fsWLogSch_\CC$ to the category of topological spaces,  sending a virtual morphism $\phi\colon X \to Y$ of log schemes to the induced continuous map $\KN{\phi} \colon \KN{X} \to \KN{Y}$.  This functor preserves products, i.e.~the natural map $\KN{X\times_\CC Y}\to\KN{X}\times\KN{Y}$ is an isomorphism.
\end{proposition}

\begin{definition}
    The \defn{Betti cohomology} of an fs log scheme $X$ of finite type over $\CC$ is the singular cohomology of its Kato--Nakayama space:
    \[
    \HB{X} := \coH[\bullet]{\KN{X};\ZZ}.
    \]
\end{definition}

Combining the symmetric monoidal functoriality of $\KN{-}$ with that of singular cohomology, we deduce the following.
\begin{corollary}
    Betti cohomology $\HB{-}$ (resp.~$\HB{-}\otimes \QQ$) is a lax (resp.~strong) symmetric monoidal functor  from $\fsWLogSch_\CC$ to the category of graded abelian groups (resp.~$\QQ$-vector spaces).
\end{corollary}

\subsection{Positive log structures}
As explained by Gillam--Molcho in \cite[Theorem 6.8.1]{GillamMolcho}, the Kato--Nakayama space has the richer structure of a ``positive log differentiable space''.  By the latter, we mean a locally ringed space  $(\Sig,\Cinf{\Sig})$ that is locally isomorphic to the vanishing set of a collection of smooth functions on $\RR^n$, equipped with a sheaf of monoids $\cM_\Sig$ and a morphism of sheaves of monoids $\alpha_\Sig \colon \cM_{\Sig} \to \Cinfge{\Sig}$, for which the induced map $\alpha_{\Sig}^{-1}(\Cinfpos{\Sig}) \to \Cinfpos{\Sig}$ is an isomorphism; here
\[
\Cinfpos{\Sig}\subset\Cinfge{\Sig}\subset\Cinf{\Sig}
\]
denote the submonoids of positive- and non-negative-valued functions, respectively.  Ordinary and virtual morphisms between positive log differentiable spaces are defined in exactly the same way as for log schemes, substituting $\Cinfge{\Sig}$ for $\cO{X}$ and $\Cinfpos{\Sig}$ for $\cOx{X}$.
    
If $X$ is an fs log scheme of finite type over $\CC$, we endow $\KN{X}$ with the structure of a positive log differentiable space in the minimal way so that that the map $\tau : \KN{X} \to \uX(\CC)$ and the locally defined functions $\set{\arg(f)}{f \in \cMgp_X}$ are smooth. In particular, this means that we have a canonical map
    \begin{align}
    \mapdef{|\alpha_X|}{\tau^{-1}\cM_X}{ \Cinfge{\KN{X}}}
    {f }{ \tau^*|\alpha_X(f)|} \label{eq:KN-prelog}
    \end{align}
    of sheaves of monoids, defining a positive pre-log structure on $\KN{X}$. Taking its associated positive log structure, we obtain the desired positive log differentiable space $(\KN{X},\cM_{\KN{X}},\alpha_{\KN{X}})$.  Note that if $f \in \tau^{-1}{\cM_X}$, then $\tau^*|\alpha_X(f)| > 0$ if and only if $f \in \tau^{-1}\cOx{X}$.   In other words, $|\alpha_X|^{-1}(\Cinfpos{\KN{X}}) = \tau^{-1}\cOx{X}$. Hence we have concretely that 
    \[
    \cM_{\KN{X}} = \Cinfpos{\KN{X}} \underset{\tau^{-1}\cOx{X}}{\sqcup} \tau^{-1}\cM_X
    \]
    with group completion
    \[
    \cMgp_{\KN{X}} =  \Cinfpos{\KN{X}} \underset{\tau^{-1}\cOx{X}}{\sqcup} \tau^{-1}\cMgp_X.
    \]
    From these formulae, it is immediate that a virtual morphism $\phi : X \to Y$ of fs log schemes over $\CC$ has a unique lift to a virtual morphism of the positive log structures on the Kato--Nakayama spaces respecting the maps $\tau$ and the maps $\tau^{-1}\cMgp_{-} \to \cMgp_{\KN{-}}$, giving the following result.
    \begin{proposition}
        The assignment $(X,\cM_X,\alpha) \mapsto (\KN{X},\cM_{\KN{X}},\alpha_{\KN{X}})$ defines a functor from $\fsLogSch_\CC$ (resp.\ $\fsWLogSch_\CC$) to the category of positive log differentiable spaces with ordinary (resp.\ virtual) morphisms.  This functor respects products.
    \end{proposition}

\subsection{KN-points vs.~virtual points}

Note that $\KN{\Spec{\CC}}$ is a point, viewed as a positive log differentiable space with the trivial positive log structure $\RR_{>0} \subset \RR$.  Consequently the functor $\KN{-}$ sends virtual points of $X$ (in the sense \autoref{def:virtual-point}) to virtual points of $\KN{X}$, i.e.~virtual morphisms from a point with trivial positive log structure to the positive log differentiable space $\KN{X}$.

Concretely, suppose that $p \in X(\CC)$ is a virtual point given by a virtual morphism $\phi : \Spec{\CC} \to X$.  It consists of a point $x \in \uX(\CC)$ and a group homomorphism $\phi^* : \cMgp_{X,x} \to \CC^\times$ whose restriction to $\cOx{\uX,x}$ is the evaluation at $x$.  Let
\begin{align*}
   \lambda := \frac{\phi^*}{|\phi^*|} : \cMgp_{X,x} \to \unitcirc \qquad \rho := |\phi^*| : \cMgp_{X,x} \to \RR_{>0} 
\end{align*}
by the polar coordinates of $\phi^*$.  Then $(x,\lambda)$ is a KN-point of $X$, and $\rho$ is equivalent to a lift of the inclusion $\{(x,\lambda)\} \hookrightarrow \KN{X}$ to a morphism of positive log differentiable spaces. 

From this description, it is immediate that the data $p=(x,\phi^*)$ and $\KN{p}=(x,\lambda,\rho)$ are equivalent, i.e.~we have the following:
\begin{proposition}\label{prop:virtual-KN-points}
    The map $p \mapsto \KN{p}$ gives a bijection between virtual $\CC$-points of $X$ and virtual points of the positive log differentiable space $\KN{X}$.
\end{proposition}
  Meanwhile, $(x,\lambda)$ is equivalent to the datum of the virtual point $(x,\phi^*)$, up to rescaling  $\phi^*$ by a homomorphism $\cMgp_{X,x}\to\RR_{>0}$ that is trivial on the subgroup $\cOx{X,x}\subset \cMgp_{X,x}$. Thus we have the following
  \begin{lemma}\label{lem:forget-magnitude}\label{lem:rescaling-virtual-points}
  The fibre of the canonical map of sets $X(\CC) \to \KN{X}$ over a KN-point $(x,\lambda)$ is a torsor for the group $\Hom{\cMgp_{X,X}/\cOx{X,x},\RR_{>0}}$.
  \end{lemma}
  In particular, since $\cMgp_{X,x}/\cOx{X,x} \cong \ZZ^j$ for some $j$, the fibre at $(x,\lambda)$ is noncanonically isomorphic to $(\RR_{>0})^j$.

\begin{example}
Let $X$ be the log scheme associated to a normal crossing divisor $\uD \subset \uX$ over $\CC$ as in \autoref{sec:tangential-basepoints}. Assume for simplicity that $\uD$ is strict.  Then $\KN{X}$ is the real oriented blowup of $\uX(\CC)$ along the irreducible components of $\uD(\CC)$; it is a manifold with corners whose boundary hypersurfaces are the $\unitcirc$-bundles associated to the normal bundles of the irreducible components of $\uD(\CC)$.  It is equipped with the positive log structure induced by the boundary defining functions, making it into a ``manifold with log corners'' in the sense of our paper \cite{DPP:Integrals}.

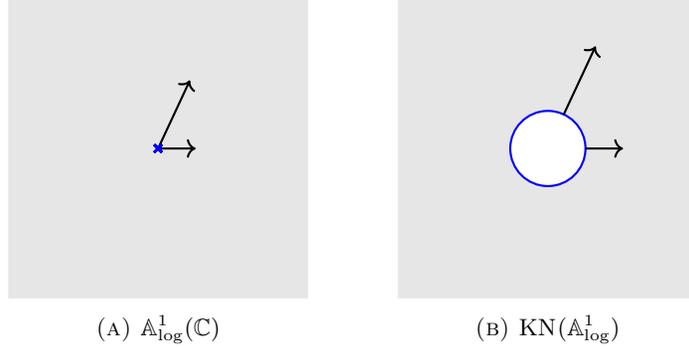
\begin{figure}
    \begin{subfigure}{0.4\textwidth}
    \centering\begin{tikzpicture}
    \draw[white,fill=white!90!black] (-2,-2) rectangle (2,2);
    \draw[->,thick] (0,0)--(0.5,0);
    \draw[->,thick] (0,0) -- (65:1);
    \draw[blue,thick] plot[mark=x] (0,0);
    \end{tikzpicture}
    \caption{$\logline(\CC)$}
    \end{subfigure}
    \begin{subfigure}{0.4\textwidth}
    \centering\begin{tikzpicture}
    \draw[white,fill=white!90!black] (-2,-2) rectangle (2,2);
    \draw[thick,blue,fill=white] (0,0) circle (0.5);
    \draw[->,thick] (0.5,0)--(1,0);
    \draw[->,thick] (65:0.5) -- (65:1.5);
    \end{tikzpicture}
    \caption{$\KN{\logline}$}
    \end{subfigure}
    \caption{Algebraic vs.~$C^\infty$ tangential basepoints at the origin in $\AF^1$.}
    \label{fig:baspoints-alg-vs-diff}
\end{figure} 

  As we explain in \cite[\S8.4]{DPP:Integrals} and illustrate in \autoref{fig:baspoints-alg-vs-diff} (copied from \emph{op.~cit.}), tangential basepoints on $X$ in the sense of \autoref{def:tangential-basepoints} correspond to tangential basepoints in the differentiable sense on $\KN{X}$.  This is a special case of \autoref{prop:virtual-KN-points}.  Namely, if $\phi \colon \Spec{\CC} \to X$ is a virtual morphism, corresponding to a tangential basepoint $(x,v)$, then $\KN{\phi} \colon \KN{\Spec{\CC}} \to \KN{X}$ is the inclusion of the point in $\KN{X}$ corresponding to the ray spanned by $v$ in the normal bundle of each irreducible component of $\uD$, equipped with a map of log structures that encodes the magnitude of the vector $v$ in each normal direction.  Thus, over a point $x \in \uX(\CC)$, the natural map $X(\CC) \to \KN{X}$ from \autoref{lem:forget-magnitude} is the quotient $N^\circ_x \to N^\circ_x / (\RR_{>0})^j$.
  \end{example}

\begin{example}\label{ex:DF-KN}
    More generally, for a Deligne--Faltings log scheme $Y = (\uX,\uD,\uE)$ associated to a split bundle $\uE = \uL_1 \oplus \cdots \oplus \uL_n$ over a normal crossing divisor $X = (\uX,\uD)$ over $\CC$, the virtual points $Y(\CC)$ are given by triples $(x,v,e)$ of a point $x \in \uX(\CC)$, a tangential basepoint $v$ for $(\uX,\uD)$ at $x$, and an element $e \in \uE^\circ|_x$.   Then $\KN{Y} \to \KN{X}$ is the $(\unitcirc)^n$ bundle associated to the pullback of the $(\CC^\times)^n$-bundle $\uE^\circ(\CC) \to \uX(\CC)$ along the blowdown map $\KN{X}\to\uX(\CC)$.  The map $Y(\CC) \to \KN{Y}$ is the quotient that takes a tangential basepoint $v$ to the corresponding normal rays as in the previous example, and takes an element $e \in \uE^\circ(\CC)$ to the element in $\uE^\circ(\CC)/(\RR_{>0})^n \cong \prod_j \uL_j(\CC)/\RR_{>0}$ giving the ray spanned by $e$ in each line of the splitting.
\end{example}

\subsection{KN-points over subrings}\label{sec:KN-points-over-subrings}
If $B \subset \CC$ is a subring and $X \in \fsLogSch_B$, we define the Kato--Nakayama space $\KN{X}$ by base change to $\CC$.  We then have canonical maps of sets $X(B) \to X(\CC) \to \KN{X}$.  The image of the composition gives a subset
\[
\KN[B]{X} \subset \KN{X},
\]
which provides a natural notion of KN-points of $X$ defined over $B$.  

\begin{example}
If $B = \RR$, the real KN-points $\KN[\RR]{X}\subset \KN{X}$ are the fixed points of the involution of $\KN{X}$ induced by complex conjugation.  This is a closed subset and has an induced positive log structure making it the fixed locus in the categorical sense; see \cite[\S8.3]{DPP:Integrals} in the case of normal crossing divisors.  
\end{example}

In light of \autoref{lem:rescaling-virtual-points}, the canonical map
\[
X(B) \to \KN[B]{X}
\]
is the quotient that identifies two virtual $B$-points whenever they differ by rescaling by a homomorphism $\cMgp \to B^\times \cap \RR_{>0}$.  In particular, in the extreme case $B = \ZZ$, we have $\ZZ^\times  \cap \RR_{>0} = \{\pm 1\} \cap \RR_{>0} = \{1\}$, so there is no rescaling freedom whatsoever, and hence the map $X(\ZZ) \to \KN[\ZZ]{X}$ is a bijection.
\begin{example}
Consider the divisorial log scheme $X = (\PP^1_\ZZ,\{0,1,\infty\})$ from \autoref{sec:basepoints-on-P1}, which has six virtual $\ZZ$-points corresponding to the ``unit'' tangential basepoints at $0,1,\infty$.  Its Kato--Nakayama space is the real oriented blowup $\widetilde{\PP^1(\CC)}$ of the Riemann sphere at the points $0,1,\infty$, and the six virtual $\ZZ$-points correspond to the six points in the boundary of $\widetilde{\PP^1(\CC)}$ indicating the positive and negative real tangent directions at 0, 1 and $\infty$. These \emph{directions} uniquely determine the corresponding \emph{vectors} because the latter have unit length in the standard coordinate on the $\PP^1$, whose scale is fixed by the $\ZZ$-structure.
\end{example}

\section{Differentials and the de Rham complex}   

We now turn to the interaction between virtual morphisms and differential forms.  Since the latter are a relative notion, we work over a fixed base log scheme $S$, i.e.~in the category $\WLogSch_S$, and eventually specialize further to the case in which $S = \Spec{\KK}$ is a field of characteristic zero.

\subsection{K\"ahler differentials}

Kato's notion \cite{Kato1989a} of logarithmic K\"{a}hler differential continues to make sense when the structure morphism is virtual, as follows.

\begin{definition}
Let $\sigma\colon X\to S$ be a virtual morphism of log schemes.
The sheaf of \defn{log K\"ahler differentials} $\forms[1]{X/S}$ is the sheaf of $\cO{X}$-modules generated by the  usual relative K\"ahler differentials $\forms[1]{\uX/\uS}$ on the underlying scheme $\uX/\uS$, and formal symbols $\mathrm{dlog}(f)$ for $f \in \cMgp_X$, modulo the following relations:
\begin{enumerate}\item The logarithmic Leibniz rule
\begin{align}
\ddlog{fg} &= \ddlog{f} + \ddlog{g} \qquad \textrm{for all }  f,g\in\cMgp_X, \label{eq:dlog-leibniz}
\end{align}
\item The consistency condition 
\begin{align}
\dd \alpha_X(f) = \alpha_X(f)\, \ddlog{f}  \qquad \textrm{for all }  f \in\cM_X, \label{eq:consistency}
\end{align}
\item The horizontality condition
\begin{align}
\ddlog{ \sigma^* s} = 0 \qquad \textrm{for all }  s \in\cMgp_S. \label{eq:horizontality}
\end{align}
\end{enumerate}
\end{definition}

Note that the consistency condition \eqref{eq:consistency} ensures that
\begin{align}
\ddlog{f} = f^{-1}\dd f \qquad \textrm{for all } f \in \cOx{X}\label{eq:virtual-consistency}
\end{align}
but it also applies to elements in $\cM_X$ that are not in $\cOx{X}$.  Since virtual morphisms have no \emph{a priori} compatibility with $\alpha_X$, this prevents log K\"ahler differentials from being functorial with respect to arbitrary virtual morphisms; see \autoref{ex:non-differentiable} below.  However, functoriality can be restored by weakening the consistency condition to only apply to elements $f \in \cOx{X}$, leading to the following alternative notion.
  \begin{definition}
      The sheaf of \defn{virtual log K\"ahler differentials} is the sheaf $\wforms[1]{X/S}$ of $\cO{\uX}$-modules generated by $\forms[1]{\uX/\uS}$ and formal symbols $\wddlog{f}$ satisfying the relations \eqref{eq:dlog-leibniz}, \eqref{eq:horizontality}, and \eqref{eq:virtual-consistency}.
  \end{definition}

Note that since the differentials of invertible functions generate $\forms[1]{\uX/\uS}$, the sheaf $\wforms[1]{X/S}$ is generated by the elements $\wddlog{f}$ for $f \in \cMgp_X$. Consequently,  we have the following.
  \begin{lemma}
      If $\phi\colon X \to Y$ is a morphism in $\WLogSch_S$, then there is a unique $\phi^{-1}\cO{Y}$-linear map
      \[
      \phi^*\colon \phi^{-1}\rbrac{\wforms[1]{Y/S}} \to \wforms[1]{X/S}
      \]
      such that
      \[
\phi^*\rbrac{\wddlog{f}} = \wddlog{\phi^*f}
      \] for all $f \in \cMgp_X$.
  \end{lemma}

There is a natural quotient map
\[
\wforms[1]{X/S} \to \forms[1]{X/S}
\]
sending $\wddlog{f} \mapsto \ddlog{f}$, so that we may make the following definition.
  \begin{definition}
   Let $\phi\colon X \to Y$ be a morphism in $\WLogSch_S$.  We say that $\phi$ is \defn{\diffable{}} if $\phi^*$ descends to a morphism
   \[
   \phi^*\colon \phi^{-1}\forms[1]{Y/S} \to \forms[1]{X/S}.
   \]
  \end{definition}

 Thus $\phi$ is \diffable{} if and only if the following condition holds in $\forms[1]{X/S}$:
 \begin{align}
 \uphi^*(\alpha_Y(f))\ddlog{\phi^*f} = \dd \uphi^*\alpha_Y(f) \qquad \textrm{for all }f \in \phi^{-1}\cM_Y \label{eq:differentiability}
 \end{align}
In particular, every ordinary morphism is \diffable{}. 
Note that both the left- and right-hand side of \eqref{eq:differentiability} behave like derivations: viewing either side of the equation as defining a map $v \colon \phi^{-1}\cM_Y \to \forms[1]{X/S}$, we have
\[
v(fg) = \uphi^*\alpha_Y(f)v(g) + \uphi^*\alpha_Y(g)v(f).
\]
Hence if \eqref{eq:differentiability} holds for $f,g \in \phi^{-1}\cM_Y$, it also holds for the product $fg$.  It therefore suffices to check \eqref{eq:differentiability} for a collection of elements $f$ that generate $\phi^{-1}\cM_Y$ over $\phi^{-1}\cOx{Y}$.

\begin{example}
    For $X \in \WLogSch_S$, let $\phi\colon S \to X$ be a section of the structure map $X \to S$.  Since $\Omega^1_{S/S} = 0$, every $\phi$ induces the zero map $\Omega^1_{X/S} \to 0$ and is therefore automatically differentiable. 
\end{example}

\begin{example}
    As a special case of the previous example, we see that the virtual morphisms corresponding to tangential basepoints  as in \autoref{sec:tangential-basepoints} are differentiable, giving the zero map on forms.
\end{example}

\begin{example}
    For $\KK$ a field, consider the standard log point $\logpt$ with phantom coordinate $t$, so that $\cM_Y = \KK^\times t^\NN$.  As explained in \autoref{ex:virtual-mor-to-log-pt}, a morphism $\phi \colon X \to \logpt$ in $\WLogSch_\KK$ is equivalent to the datum of the element $\phi^*(t) \in \cMgp(X)$.  We claim that such a morphism is automatically differentiable.  Indeed, since $\cM_Y$ is generated over $\KK^\times$ by $t$, it suffices to verify \eqref{eq:differentiability} for $f = t$.  But in this case, $\alpha_{\logpt}(f) = 0$ so that both the left- and right-hand sides of \eqref{eq:differentiability} are identically zero.  The same argument shows, more generally, that if $Y$ is a log scheme whose underlying scheme is a point $\uY \cong \Spec{\KK}$, then every virtual morphism $X \to Y$ in $\WLogSch_\KK$ is differentiable.
\end{example}

The following example shows that a virtual morphism need not be \diffable{} in general, even if the base $S$ is a field of characteristic zero.

\begin{example} \label{ex:non-differentiable}
Let $X := \Spec{t^\NN \to \KK[\epsilon]/(\epsilon^2) ; t\mapsto \varepsilon}$ be the log fat point from \autoref{ex:fat-log-point}.  The module $\forms[1]{X/\KK}$ is free of rank one, generated by $\ddlog{t}$.  Consider the virtual morphism $\phi : X \to X$ over $\KK$ defined by
\begin{align*}
\uphi^*\epsilon &= \epsilon & \phi^* t &= t^j
\end{align*}
for some $j \in \ZZ$.  We claim that $\phi$ is differentiable if and only if $j=1$.  Indeed, consider the equation \eqref{eq:differentiability} with $f=t$.  The right-hand side is given by
 \[
 \dd \uphi^* \alpha(t) = \dd \uphi^*\epsilon = \dd \epsilon =  \epsilon\, \ddlog{t}
 \]
 where the last equality follows from the consistency condition \eqref{eq:consistency}.  Meanwhile, the left-hand side of \eqref{eq:differentiability} reads
 \[
\uphi^*\alpha(t)\,\ddlog{\phi^*t} =  \uphi^*\epsilon\, \ddlog{t^j} = 
j \epsilon \,\ddlog{t}.
 \]
 Hence \eqref{eq:differentiability} holds if and only if $j=1$, as claimed.
\end{example}

In this example, the issue is caused by the presence of the nilpotent function $\epsilon = \alpha(t)$, which is nonzero, but also nowhere invertible, so that the action of a virtual morphism on $t$ is unconstrained.  This suggests that the lack of reducedness is playing a role in the failure of differentiability, and this is indeed the case, as is made precise by the following lemma. 

\begin{lemma}\label{lem:differentiability}
   Suppose that $X \in \WLogSch_S$ has the following properties:
   \begin{enumerate}
       \item The underlying scheme $\uX$ is reduced and locally Noetherian.
       \item The sheaf $\forms[1]{X/S}$ of logarithmic K\"ahler differentials is locally free.
   \end{enumerate}
     Then every  morphism $\phi \colon X \to Y$ in $\WLogSch_S$ is differentiable.
\end{lemma}

\begin{proof}
We must verify \eqref{eq:differentiability} for every section $f \in \uphi^{-1}\cM_Y$. Since the property of being reduced and locally Noetherian is stable under passing to  \'etale opens, and in particular to Zariski covers, we can assume without loss of generality that $f \in \sect{\uY,\cM_Y}$ is a global section and $\uX$ is Noetherian.
Then since $\forms[1]{X/S}$ is locally free and $\uX$ is reduced, it suffices to check \eqref{eq:differentiability} at the generic point of each irreducible component of  $\uX$.  In particular,  we may assume that $\uX$ is integral. 

Let $\eta \in \uX$ be the generic point. By the continuity principle (\autoref{lem:continuity-principle}) there are two possibilities.  The first possibility is that $\uphi^*\alpha_Y(f)=0$ is identically zero, in which case both sides of \eqref{eq:differentiability} are also zero, so in particular the equation holds.  The other possibility is that $f \in \cOx{Y,\uphi(\eta)}$, in which case \eqref{eq:differentiability} reduces to the defining relation \eqref{eq:virtual-consistency}.
\end{proof}

\begin{example}[Splittings on forms]
Let $\KK$ be a field.  In \cite[Construction 3.3]{Achinger2023} Achinger constructed, for every sufficiently nice log scheme $X/\KK$  equipped with a splitting of monoids $\epsilon \colon \cM_X/\cOx{X} \to \cM_X$, a map on forms $\epsilon^\circledast \colon \forms[1]{X/\KK} \to \forms[1]{\uX/\KK}$ (and a corresponding functor on flat connections and local systems) that ``is not induced
by a map of log schemes $\uX \to X$, though it behaves as if it was'', in that it splits the canonical map $\forms[1]{\uX/\KK}\to\forms[1]{X/\KK}$.   It does, however, come from a \emph{virtual} morphism.  Namely, under the hypotheses of \emph{op.~cit.}, the splitting $\epsilon$ corresponds via \autoref{ex:achinger-splittings} and \autoref{lem:differentiability} to a differentiable virtual morphism $\phi \colon \uX \to X$, giving a section of the canonical projection $X \to \uX$.  Then $\epsilon^\circledast = \phi^*$ is exactly the pullback on forms induced by $\phi$, and indeed the construction  in \emph{op.~cit.}  uses the reducedness (via a special case of the continuity principle) to check the consistency axiom, exactly as in the proof of \autoref{lem:differentiability}.
\end{example}

\subsection{de Rham cohomology}\label{sec:de-Rham}
Let $\KK$ be a field of characteristic zero.  To speak of the de Rham cohomology of a log scheme $X/\KK$, we require a certain smoothness assumption from \cite{KatoNakayama,OgusRHC}, which we now recall.  Note that this notion is more general than the notion of a ``log smooth'' variety:
\begin{definition}
    An fs log scheme $X \in \fsLogSch_\KK$ over $\KK$ is \defn{ideally smooth} if, locally in the \'etale topology, it is isomorphic to one of the form $\Spec{M \to \KK[M]/(I)}$ where $M$ is an fs monoid and $I < M$ is a monoid ideal.  
We denote by
\[
\smWLogSch_\KK \subset \fsWLogSch_\KK
\]
the full subcategory of ideally smooth log schemes.
\end{definition}

In geometric terms,  $X$ ideally smooth if it is locally isomorphic to the vanishing locus of a monomial ideal in a toric variety (possibly non-reduced).  These conditions imply that the sheaf of K\"ahler differentials $\forms[1]{X/\KK}$ is locally free: in a local toric model, a basis is given by the logarithmic differentials of a system of toric coordinates  on the ambient toric variety.

If $X$ is ideally smooth, its \defn{logarithmic de Rham complex} is the sheaf of differential graded $\KK$-algebras
\[
(\forms{X/\KK},\dd) := \left(\begin{tikzcd}
\cO{\uX} \ar[r] & \forms[1]{X/\KK} \ar[r] &\forms[2]{X/\KK} \ar[r]& \cdots 
\end{tikzcd}\right)
\]
where $\forms[p]{X/\KK} := \wedge^p\forms[1]{X/\KK}$ is the $p$th exterior power as an $\cO{\uX}$-module.  The differential is uniquely determined by the graded Leibniz rule, together with the requirement that $\dd \colon \cO{X} \to \forms[1]{X/\KK}$ is the composition of the usual differential $\cO{\uX}\to\forms[1]{\uX/\KK}$ with the natural map $\forms[1]{\uX/\KK}\to\forms[1]{X/\KK}$, and the identity
\[
\dd(\ddlog{f}) = 0
\]
holds for all $f \in \cM_X$.

\begin{definition}
    The \defn{de Rham cohomology} of an ideally smooth log scheme $X/\KK$ is the hypercohomology
    \[
    \HdR{X} := \mathbb{H}^\bullet(\forms{X/\KK},\dd).
    \]
\end{definition}

The logarithmic de Rham complex is functorial with respect to ordinary morphisms $\phi\colon X \to Y$ of ideally smooth log schemes over $\KK$; if, in addition, $\uX$ is reduced, this complex is functorial with respect to all virtual morphisms thanks to \autoref{lem:differentiability}.  To treat the non-reduced case, we shall use the following ``nil-invariance'' property.
\begin{lemma}\label{lem:nil-invariance}
    If $X$ is ideally smooth, then the inclusion $X_\red \to X$ induces a quasi-isomorphism $(\forms{X/\KK},\dd) \to (\forms{X_\red/\KK},\dd)$ of sheaves of differential graded algebras.
\end{lemma}
\begin{proof}
    The statement is compatible with field extensions, so by the Lefschetz principle, it suffices to prove it for $\KK=\CC$, in which case it follows from Kato--Nakayama's comparison isomorphism $\HdR{X}\cong \HB{X}\otimes \CC$ (see \cite[Theorem (0.2)(2)]{KatoNakayama}) and the fact that the topological spaces $\KN{X} = \KN{X_\red}$ are the same.
\end{proof}

Using the lemma we may express the functoriality of de Rham cohomology for virtual morphisms (even at cochain level), as follows. 
Let $\rsect{-}$ denote a symmetric monoidal functor of derived global sections of complexes of sheaves, e.g.~the Thom--Whitney normalization of the Godement resolution as in \cite[\S1-5]{NavarroAznar1987} will work.  We denote by 
\[
\rsect{X,\forms{X}} = \rsect{X,(\forms{X/\KK},\dd)}
\]
the resulting global de Rham complex, so that we have a canonical isomorphism of graded $\KK$-algebras
\[
\HdR{X} \cong \coH{\rsect{X,\forms{X}}}
\]
If $\phi : X \to Y$ is an arbitrary  virtual morphism over $\KK$, we have a zig-zag of morphisms of dg $\KK$-algebras
\[
\begin{tikzcd}[column sep=3.5em]
    \rsect{Y,\forms{Y}}\ar[r,"\sim"] &  \rsect{Y_{\red},\forms{Y_\red}} \ar[r,"\rsect{\phi_\red^*}"] & \rsect{X_{\red},\forms{X_\red}} & \ar[l,"\sim"'] \rsect{X,\forms{X},\dd}
\end{tikzcd}
\]
where the arrows labelled $\sim$ are quasi-isomorphisms by \autoref{lem:nil-invariance}.  Following the arrows from left to right, we obtain a canonical morphism
\[
\phi^* : \rsect{Y,\forms{Y}} \to \rsect{X,\forms{X}}
\]
in the homotopy category of dg $\KK$-algebras.    Moreover, we have a canonical isomorphism of sheaves $\forms{X}\boxtimes \forms{Y} \cong \forms{X\times Y}$, which gives a quasi-isomorphism
\[
\rsect{X,\forms{X}} \otimes_\KK \rsect{Y,\forms{Y}} \to \rsect{X \times Y,\forms{X \times Y}},
\]
Thus the global de Rham complex  is a lax symmetric monoidal functor up to homotopy on $\smWLogSch_\CC$, which becomes a strong symmetric monoidal functor after passing to cohomology.

\begin{corollary}
  Algebraic de Rham cohomology $\HdR{-}$ is a strong symmetric monoidal functor from $\smWLogSch_\KK$ to the category of graded $\KK$-vector spaces.
\end{corollary}

\section{Logarithmic functions and the comparison isomorphism}
\label{sec:comparison}
 
 In \cite{KatoNakayama}, Kato--Nakayama construct a canonical \defn{Betti--de Rham comparison isomorphism}
\begin{equation}\label{eq:Betti-de-Rham-comparison}
\HdR{-} \cong \HB{-}\otimes_\ZZ \CC,
\end{equation}
which is natural with respect to ordinary morphisms of log schemes.  The key step in the proof is a version of the Poincar\'e lemma for a certain sheaf of ``forms with logarithmic coefficients'' on the Kato--Nakayama space.  Our aim now is to prove that this comparison is natural with respect to arbitrary virtual morphisms (and natural up to homotopy at cochain level).  Note that since Betti and de Rham cohomology are invariant under nilpotent thickenings by \autoref{lem:nil-invariance}, it suffices to prove the result for reduced log schemes. 

\begin{remark}
A comment is in order here, since on the one hand, we are using \autoref{lem:nil-invariance} to reduce the problem to reduced schemes, but the proof of that lemma  appeals to the Kato--Nakayama comparison isomorphism.  However, said proof only uses naturality for the reduction $X_\red \to X$, which is an ordinary morphism; hence there is no circular reasoning. It should also be possible to extract a direct proof of the nil-invariance for any $\KK$ from their arguments.
\end{remark}

\begin{remark}
    While we treat only the case $\KK=\CC$, it immediately implies a corresponding result for  any subfield $\KK \subset \CC$ by base change.  Namely, in this setting, the de Rham cohomology is a $\KK$-vector space, and we define the Betti cohomology by base change to $\CC$.  We then have a natural isomorphism
    \[
    \HdR{-}\otimes_\KK \CC \cong \HB{-}\otimes_\ZZ\CC
    \]
    of monoidal functors on $\smWLogSch_\KK$.
\end{remark}

\subsection{Logarithmic functions}

Let $X$ be an ideally smooth log scheme over $\CC$.  We denote by $\an{X}$ the associated log analytic space, given by analytification of the underlying scheme $\uX$ and pullback of the log structure under the natural map $\an{\uX} \to \uX$ of ringed spaces.

In \cite{KatoNakayama}, Kato--Nakayama introduce a sheaf $\cOlog{X}$ of ``formal logarithmic functions'' on  the space $\KN{X}$, as follows.

\begin{definition}
A \defn{formal logarithm} on $\KN{X}$ is a pair $(f,\theta)$ consisting of a section $f \in \tau^{-1}\cMgp_{\an{X}}$ and a continuous function $\theta \in \Cont{\KN{X}}$ such that $\arg(f) = \exp(\iu \theta)$ as $\unitcirc$-valued functions.  We denote this pair by
\[
\log_\theta(f) := (f,\theta).
\]
We denote the sheaf of formal logarithms by
\[
\cL_X \cong \Cont{\KN{X}} \fibprod_{\Contcirc{\KN{X}}} \tau^{-1}\cMgp_{\an{X}}
\]
where $\Contcirc{\KN{X}}$ is the sheaf of $\unitcirc$-valued continuous functions.
\end{definition}
A formal logarithm $\log_\theta(f)$ can be thought of as a choice of ``branch of logarithm'' for the section $f \in \cM_{\an{X}}$, and the projection $\cL_X\to \tau^{-1}\cMgp_{\an{X}}$ as the exponential function. 
There is a canonical map
\[
\begin{tikzcd}
\tau^{-1}\cO{\an{X}} \ar[r]&{\cL_X}
\end{tikzcd}
\]
sending a function $g \in \cO{\an{X}}$ to the section $\log_\theta(f)$ where $f = e^g$ is the exponential, and $\theta = \Im(g) \in \Cont{\KN{X}}$ is the imaginary part.  
\begin{definition}
    We denote by $\cOlog{\an{X}}$ the sheaf of $\tau^{-1}\cO{\an{X}}$ algebras generated by the formal logarithms $\log_\theta(f) \in \cL_X$ modulo the relation
    \begin{align}
    g = \log_{\Im{g}}(e^g) \label{eq:logarithm-consistency}
    \end{align}
    for all $g \in \tau^{-1}\cO{\an{X}}$
\end{definition}

The sheaf of formal logarithms is functorial with respect to virtual morphisms: 
\begin{lemma}\label{lem:functoriality-of-log-functions}
For a morphism $X \to Y$ in $\smWLogSch_\CC$, there is a unique map 
\[
\phi^*\colon \phi^{-1}\cOlog{\an{Y}} \to \cOlog{\an{X}}
\]
of $\tau^{-1}\cOlog{\an{Y}}$-algebras such that 
\begin{align}
\phi^*\log_\theta(f) = \log_{\KN{\phi^*}\theta}(\phi^*f) \label{eq:log-pullback}
\end{align}
for every formal logarithm $\log_\theta(f) \in \cL_Y$.
\end{lemma}

\begin{proof}
That the right-hand-side of \eqref{eq:log-pullback} defines a formal logarithm on $\KN{X}$ follows immediately from the definition of the induced map of Kato--Nakayama spaces, which as noted above, implies that $\arg \phi^* f = \KN{\phi}^*\arg(f)$ for all $f \in\cM_Y$.  This gives the map on generators.  It remains simply to observe that the relation \eqref{eq:logarithm-consistency} only involves the elements of $\cM_{\an{Y}}$ of the form $\exp(g) \in \cOx{\an{Y}}$, and the latter are preserved by the definition of virtual morphisms.
\end{proof}

\subsection{Naturality of the comparison map}

We now prove the naturality of the Betti--de Rham comparison map with respect to arbitrary virtual morphisms.  

For this, we recall that in \cite{KatoNakayama}, Kato--Nakayama introduce the complex
\[
\forms[\bullet,\log]{X} := \cOlog{\an{X}}\otimes_{\tau^{-1}\cO{X}} \tau^{-1}\forms{X/\CC},
\]
of sheaves of ``holomorphic forms with logarithmic coefficients'' on $\KN{X}$, 
which they denote by $\omega^{\bullet,log}_X$.  It carries a natural de Rham differential $\dd$, and they prove that the natural maps
\begin{equation}
\begin{tikzcd}
    \tau^{-1}(\forms{X/\CC},\dd)\ar[r] & (\forms[\bullet,\log]{X},\dd) &  \CC_{\KN{X}} \ar[l]
\end{tikzcd}\label{eq:BdR-morphisms}
\end{equation}
are quasi-isomorphisms of complexes of sheaves, where the right hand side is the constant sheaf, viewed as a complex concentrated in degree zero.
Passing to derived global sections using a symmetric monoidal resolution functor for $\rsect{-}$ as in \autoref{sec:de-Rham}, we obtain a quasi-isomorphism
\[
\rsect{X,\forms{X}} \cong \rsect{\KN{X};\CC},
\] 
which induces the Betti--de Rham comparison isomorphism
\[
\HdR{X} \cong \HB{X}\otimes_\ZZ\CC.
\]
From the functoriality of $\cOlog{\an{X}}$ (\autoref{lem:functoriality-of-log-functions}), the nil-invariance (\autoref{lem:nil-invariance}), and the functoriality of $\forms{X}$ for reduced ideally smooth $X$ (\autoref{lem:differentiability}), it is immediate that the maps \eqref{eq:BdR-morphisms} are natural with respect to virtual morphisms and compatible with products, so that we have a weak equivalence
\[
\rsect{-,\forms{}} \cong \rsect{\KN{-};\CC}
\]
of lax symmetric monoidal functors up to homotopy on $\smWLogSch_\CC$.

\begin{corollary}
    Kato--Nakayama's comparison isomorphism gives a natural equivalence
    \[
    \HdR{-} \cong \HB{-}\otimes_\ZZ\CC
    \]
    of strong symmetric monoidal functors on $\smWLogSch_\CC$.
\end{corollary}

\section{Regularized pullbacks to strata}\label{sec:specialization}

Let us now explain how the cohomological formalism plays out in the concrete example of a normal crossing divisor.  The basic idea is that pulling back along a virtual morphism gives rise to a natural ``regularized pullback in cohomology'' that let us restrict classes in the complement of a normal crossing divisor to classes on the strata of the divisor, compatible with the comparison between Betti and de Rham cohomology.  The possibility of constructing such pullbacks is one of the motivations for Deligne's introduction of tangential basepoints.

We adopt the notation of \autoref{sec:tangential-basepoints}, in which  $X$ is the smooth log scheme over $\KK$ associated to a normal crossing divisor $(\uX,\uD)$.  We now further assume that $\KK$ is a subfield of $\CC$. Let $\coH{-} = \HB{-}$ or $\HdR{-}$, and recall that the inclusion $\uX\setminus \uD \hookrightarrow X$ induces a quasi-isomorphism on Betti and de Rham cochains, giving an isomorphism
\[
\coH{X}\cong \coH{\uX\setminus \uD},
\]
compatible with the comparison $\HB{-}\otimes\CC \cong \HdR{-}\otimes \CC$.

Now suppose that  $i \colon \uY \to \uX$ is a locally closed immersion given locally by an intersection of divisor components as in  \autoref{sec:strata}.  Then the immersion of log schemes $\widehat{X}^\circ_Y \to N^\circ_YX$ induces a homotopy equivalence of Kato--Nakayama spaces, with inverse given by the projection of the punctured normal bundle onto its corresponding circle bundle.  The diagram \eqref{eq:punctured-neighbourhood} of ordinary morphisms thus induces the following commutative diagram in cohomology, in which many of the maps are isomorphisms:
\begin{equation*}
\begin{tikzcd}[column sep=1.5em]
    \coH{N_{\uY}\uX \setminus N_{\uY}\uD}  & \coH{N^\circ_Y X}  \ar[l,"\sim"'] \ar[r,"\sim"] & \coH{\widehat X^\circ_Y} & \coH{X} \ar[l] \ar[r,"\sim"] & \coH{\uX \setminus \uD} \\
    \coH{\uY \setminus \uD_{Y}} \ar[u] & \coH{Y}\ar[l,"\sim"']\ar[u] &&& \coH{\uX}\ar[lll]\ar[u]
\end{tikzcd}\label{eq:punctured-neighbourhood-cohomology}
\end{equation*}
Tracing through the diagram, we see that there is a canonical induced map
\[
\sigma \colon \coH{X} \to \coH{N^\circ_Y X}
\]
that is canonically identified with the classical \defn{specialization map}
\[
\underline{\sigma} \colon \coH{\uX\setminus\uD} \to \coH{N_{\uY}\uX \setminus N_{\uY}\uD}.
\]
On Betti cohomology, the latter is induced by restriction to any choice of punctured $C^\infty$ tubular neighbourhood of $Y$.  Note that Levine \cite{LevineTubular} constructed a motivic lift of the specialization map, which is thus defined for any Weil cohomology theory.

Now suppose that $\phi \colon Y \to X$ is a virtual morphism lifting the immersion $i \colon \uY \to \uX$, and let $s \colon \uY \setminus \uD_Y \to N_{\uY}\uX \setminus N_{\uY}\uD$ be the corresponding inward pointing section of the normal bundle as in \autoref{prop:virtual-inclusions-of-strata}.  Since $\phi$ factors through the pullback log structure, we deduce from the above that its action on cohomology is as follows.
\begin{proposition}
    The pullback map
    \[
    \phi^* \colon \coH{X} \to \coH{Y}
    \]
    is naturally identified with the composition
    \[
    s^* \circ \underline{\sigma} \colon \coH{\uX\setminus \uD} \to \coH{\uY\setminus \uD_Y}.
    \]
\end{proposition}

\begin{remark}
    In de Rham cohomology, the map $\phi^*$ has the following description as a ``regularized pullback'', in which forms with logarithmic poles along $Y$ are first made nonsingular by formally subtracting their poles, and then pulled back to $\uY$ in the usual way.
    
    Recall that we have $\forms[1]{X/\KK} \cong \forms[1]{\uX}(\log\uD)$, the sheaf of differential forms with logarithmic poles, so that taking residues along the components of $\uD$ containing $\uY$ gives an exact sequence
    \begin{equation}
    \begin{tikzcd}
    0 \ar[r] & \forms[1]{\uY\setminus\uD_Y} \ar[r] &\forms[1]{X}|_{\uY\setminus\uD_Y} \ar[r,"\mathrm{Res}"] & \cO{\uY\setminus\uD_Y}^{\oplus j} \ar[r] & 0
    \end{tikzcd} \label{eq:log-residue-sequence}
    \end{equation}
    where $j = \mathrm{codim}(\uY,\uX)$.  The section $s$ trivializes the normal bundle of each divisor component along $\uY$, so it is equivalent to the data of the 1-jet of a collection of defining equations $y_1,\ldots,y_j$ for the components along $\uY\setminus\uD$.  The pullback $\phi^*\colon \forms[1]{X} \to\forms[1]{\uY\setminus \uD_{\uY}}$ is then the (left) splitting of the sequence \eqref{eq:log-residue-sequence} given by the formula
    \[
    \phi^* \omega = \uphi^*\rbrac{\omega - \sum_{i}\mathrm{Res}_{y_i=0}(\omega)\dlog{y_i}} \in \forms[1]{\uY\setminus\uD_Y}
    \]
    which induces the map on forms of all degrees by exterior product.
\end{remark}

Note that the argument works at cochain level, i.e.~in the derived category. Consequently, if $X_\bullet$ is any diagram in $\WLogSch$ whose objects are isomorphic to strata in normal crossing divisors equipped with the induced log structure, and whose morphisms are isomorphic to virtual morphisms lifting inclusions of strata, then $X_\bullet$ has a canonical realization in the derived category of motives over $\KK$.  It follows that the ``periods'' defined as the matrix coefficients of the  Betti--de Rham comparison map for the total cohomology of such diagrams are all classical periods over $\KK$ in the sense of \cite{KontsevichZagier}. 

\section{The little disks operad
}
\label{sec:little-disks}
In this final section, we explain how the topological operad of little disks can be lifted to an operad of ideally smooth log schemes over $\Spec{\ZZ}$, constructed as a moduli space of stable curves of genus zero with tangential basepoints.  The operadic compositions will be defined by virtual morphisms that transport tangential basepoints around the curves, following Beilinson's strategy in~\cite{BeilinsonLetterToKontsevich}. 

\subsection{Transporting tangential basepoints on a curve}

A remarkable property of smooth rational curves is that although their tangent bundles are nontrivial, there is a canonical recipe for transporting tangential basepoints between different points on the curve.  In geometric terms, one first applies a conformal transformation of a sphere to arrange that the points are antipodal, and then applies the antipodal involution.  Since the latter map reverses orientations, it has no direct algebraic analogue.  Our aim in this subsection is to give a purely algebraic construction of this map via virtual morphisms, and explain how to extend it to singular curves decorated with additional ``gluing data'' at the nodes, to be defined below.

Throughout the present subsection, we fix a tree of rational curves over a field $\KK$, which we denote by $C$, and we fix in addition a pair $p\neq q \in C$ of distinct smooth points of $C$.   In other words, $C$ is a connected projective curve of genus zero with at worst nodal singularities, whose irreducible components are all smooth, and the points $p$ and $q$ each lie on a unique irreducible component of $C$.  Note that $p$ and $q$ may or may not lie on the same component of $C$; this dichotomy  will be important in what follows.

\begin{lemma}\label{lem:s-tensor}
The following statements hold:
\begin{enumerate}
\item The space $\coH[0]{C,\cO{C}(q-p)}$ is one-dimensional. Any non-zero element $f$ of this space is regular near $p$ and its reciprocal $\frac{1}{f}$ is regular near $q$.
\item There is a unique  element $s_{p,q} \in T^\vee_pC \otimes T^\vee_q C$ such that
\[
s_{p,q} = \dd f|_p \otimes \dd(\tfrac{1}{f})|_q 
\]
for every nonzero element $f \in \coH[0]{C,\cO{C}(q-p)}$.  
\item We have $s_{p,q} \neq 0
$ if and only if $p$ and $q$ lie on the same irreducible component of $C$.
\end{enumerate}
\end{lemma}

\begin{proof}
(2) follows immediately from (1) since the formula defining $s_{p,q}$ is invariant under rescaling $f$. To prove (1) and (3) we treat the two cases separately: $p$ and $q$ on the same irreducible component of $C$, or on different components.

First assume that $p$ and $q$ are on the same irreducible component $\tilde{C}$ of $C$. Then since $\tilde{C}\cong \PP^1$ it is clear that $\coH[0]{\tilde{C},\mathcal{O}_{\tilde{C}}(q-p)}$ has dimension $1$ and that its non-zero elements $f$ have a zero of order $1$ at $p$ and a pole of order $1$ at $q$, thus satisfying $df|_p\neq 0$ and $d(\frac{1}{f})|_q\neq 0$. The claim follows from the fact that the restriction map $\coH[0]{C,\mathcal{O}_C(q-p)}\to \coH[0]{\tilde{C},\mathcal{O}_{\tilde{C}}(q-p)}$ is an isomorphism: a section of $\mathcal{O}_{\tilde{C}}(q-p)$ extends uniquely to a section of $\mathcal{O}_C(q-p)$, which is necessarily constant on all of the other irreducible components of $C$.

Now assume that $p$ and $q$ are on different irreducible components of $C$, denoted by $C_p$ and $C_q$ respectively. Any $f\in \coH[0]{C,\cO{C}(q-p)}$ is regular on $C_p$ and vanishes at $p$, hence vanishes identically on $C_p$. By the same reasoning, $f$ vanishes on every component of $C$ that touches $C_p$ but does not contain $q$. By iterating this reasoning, one sees that $f$ has to vanish at the node $\nu$ of $C_q$ that is on the only path in the dual tree of $C$ between $C_p$ and $C_q$. Therefore, we have a restriction map $\coH[0]{C,\cO{C}(q-p)}\to \coH[0]{C_q,\cO{C_q}(q-\nu)}$, which is easily seen to be an isomorphism: a section of $\cO{C_q}(q-\nu)$ extends uniquely to a section of $\cO{C}(q-p)$ which is constant on each irreducible component other than $C_q$. Since $f$ is zero near $p$, we have $df|_p=0$ and therefore $s_{p,q}= 0$.
\end{proof}

We conclude that if $p,q \in C$ lie on the same component, then $s_{p,q}$ induces a canonical linear isomorphism $T_pC \cong T_q^\vee C$. Composing it with the inversion map $(T_q^\vee C)^\times\cong (T_qC)^\times$, we obtain a bijection
    \[
    (T_pC)^\times \cong (T_qC)^\times
    \]
    between the sets of tangential basepoints on $C$ at $p$ at $q$, which is $\Gm$-equivariant with respect to the inversion automorphism of $\Gm$. In view of \autoref{prop:morphisms-as-monomial-maps}, it will be more convenient to think of this bijection as a virtual isomorphism between the lines $T_pC$ and $T_qC$, viewed as log structures on a point (both non canonically isomorphic to the log point $\logpt$). We therefore get the following.
 
    \begin{corollary}\label{cor:transport-tangential-basepoints}
     If $p,q$ lie on the same irreducible component, then we get a canonical virtual isomorphism
     \[
     s_{p,q}\colon T_pC \stackrel{\sim}{\to} T_qC.
     \]
    \end{corollary}

    \begin{example}\label{ex:transport-tangential-basepoints-easy}
    If $C=\PP^1$ with coordinate $z$ then we have, for each $p,q\in C\setminus\{\infty\}$ and $\lambda\in\KK^\times$,
    \[
    s_{p,\infty} (\lambda\partial_z|_p) = \lambda^{-1}\partial_{1/z}|_\infty \quad \mbox{ and } \quad s_{\infty,q}(\lambda\partial_{1/z}|_\infty) = \lambda^{-1}\partial_z|_q,
    \]
    explicitly exhibiting the $\Gm$-antilinearity coming from the inversion.
    \end{example}

    On the other hand, if $p$ and $q$ lie on distinct irreducible components, then $s_{p,q}$ is zero, so it no longer induces a correspondence between tangential basepoints.  In order to construct the desired correspondence we need to transport tangent vectors through the nodes lying between $p$ and $q$ via the following procedure.

Since $C$ is a tree of rational curves, the points $p$ and $q$ are joined by a unique chain of irreducible components $C_p = \tilde C_0 , \tilde C_1 ,\ldots, \tilde C_n = C_q$ such  that each component $\tilde C_i$ intersects the previous component $\tilde C_{i-1}$ in a node $\nu_i$ and all other pairwise intersections are empty. We would like to transport tangent vectors from $p$ to $q$ via the intermediate points $\nu_i$ using the chain of isomorphisms induced by $s_{p,\nu_1}$, $s_{\nu_1,\nu_2},\ldots,s_{\nu_n,q}$.  However, these maps do not compose: the image of $s_{\nu_{i-1},\nu_i}$ is $T_{\nu_i} C_i$, whereas the domain of $s_{\nu_i,\nu_{i+1}}$ is $T_{\nu_i} C_{i+1}$, i.e.~they correspond to tangent vectors at $\nu_i$ which lie on different components of $C$, and there is no canonical isomorphism between these tangent spaces---not even a virtual one.  We therefore need to specify such an isomorphism as part of the data, which motivates the following.
\begin{definition}
    Let $\nu \in C$ be a node, and let $C' \neq C''$ be the distinct irreducible components of $C$ containing $\nu$.   A \defn{gluing datum at $\nu$} is a nonzero element of the tensor product $T_\nu C' \otimes T_{\nu} C'' $.
\end{definition}

Thus a gluing datum $\eta \in T_\nu C' \otimes T_\nu C''$ defines a linear isomorphism $T_\nu C'\cong T_\nu^\vee C''$, and composing with the inversion $(T_\nu^\vee C'')^\times\cong (T_\nu C'')^\times$ we get a virtual isomorphism $T_\nu C'\stackrel{\sim}{\to} T_\nu C''$. Consequently, if $\eta_1,\ldots,\eta_n$ are gluing data at $\nu_1,\ldots,\nu_n$, we have a composable chain of virtual isomorphisms
\begin{equation}
\begin{tikzcd}[column sep=2.5em]
    T_pC = T_p \tilde C_0\ar[r,"s_{p,\nu_1}"] &T_{\nu_1}\tilde C_0 \ar[r,"\eta_1"]  & T_{\nu_1}\tilde C_1 \ar[r,"{s_{\nu_1,\nu_2}}"] & \cdots \ar[r,"{s_{\nu_n,q}}"] & T_q \tilde C_n = T_qC
\end{tikzcd}\label{eq:composite-transport}
\end{equation}
\begin{corollary}\label{cor:transport-tangential-basepoints-through-nodes}
    If $\nu_1,\ldots,\nu_n$ are the nodes lying between $p$ and $q$, and $\eta_1,\ldots,\eta_n$ are gluing data at these nodes, the composite \eqref{eq:composite-transport} defines a canonical virtual isomorphism  $T_pC \stackrel{\sim}{\to} T_qC$
\end{corollary}
In terms of classical multi-linear algebra, this virtual morphism is obtained by starting with the tensor
\[
s_{p,\nu_1}\otimes \eta_1 \otimes s_{\nu_1,\nu_2}\otimes \cdots \otimes s_{\nu_n,q} \in (T^\vee_p \tilde C_0 \otimes T_{\nu_1}^\vee\tilde C_0) \otimes (T_{\nu_1}\tilde C_0\otimes T_{\nu_1} \tilde C_1) \otimes \cdots \otimes  (T^\vee_{\nu_n}\tilde C_n\otimes T_q^\vee \tilde C_n)
\]
and contracting the adjacent factors $T^\vee_{\nu_i}\tilde C_i \otimes T_{\nu_i} \tilde C_i \cong \KK$ to obtain a nonzero element in $T_p^\vee C\otimes T_q^\vee C$.  The latter defines a linear isomorphism $T_pC\cong T_q^\vee C$ which we compose with the inversion $(T_q^\vee C)^\times\cong (T_q C)^\times$ to obtain the virtual isomorphism of \autoref{cor:transport-tangential-basepoints-through-nodes}.

\begin{example}
    Let $C_1=\PP^1$ with coordinate $z_1$, $C_2=\PP^1$ with coordinate $z_2$, and $C$ be the nodal curve obtained by identifying a point $p_1\in C_1\setminus\{\infty\}$ with $\infty\in C_2$. A gluing datum at the resulting node $\nu$ is given by an element
    \[
    \eta = a \, \partial_{z_1}|_{p_1}\otimes \partial_{1/z_2}|_\infty
    \]
    with $a\in\KK^\times$. We view it as the virtual isomorphism $\eta\colon T_\nu C_1 \stackrel{\sim}{\to} T_\nu C_2$ given by
    \[
    \lambda \partial_{z_1}|_{p_1} \mapsto a\lambda^{-1} \partial_{1/z_2}|_\infty.
    \]
    Therefore, if we let $p=\infty\in C_1$ and $q=q_2\in C_2\setminus\{\infty\}$ we see, thanks to the computation of \autoref{ex:transport-tangential-basepoints-easy}, that the virtual isomorphism $s_{p,q}^\eta\colon T_pC\to T_q C$ is computed by 
    \[
    \lambda \partial_{1/z_1}|_\infty \mapsto \lambda^{-1}\partial_{z_1}|_{p_1} \mapsto a\lambda \partial_{1/z_2}|_\infty \mapsto a^{-1}\lambda^{-1}\partial_{z_2}|_{q_2}.\qedhere
    \]
\end{example}

\subsection{Universal  transport  over the moduli space}

We now explain how to perform the construction of the previous subsection for \emph{families} of curves, explaining what happens when a smooth curve degenerates to a singular one.  In other words, we perform the construction universally over the moduli space of stable marked curves of genus zero.

\subsubsection{Moduli spaces of curves}
We refer the reader to \cite{Knudsen, Keel, GoncharovManin} for details on moduli spaces of curves. For a finite set $A$ with $n \ge 2$ elements, let $\bM{A} \cong \bM{0,n+1}$ denote the moduli space of stable curves of genus zero with marked points labelled by the set $A\sqcup \{\infty\}$.  It is a smooth projective scheme over $\Spec{\ZZ}$ parameterizing isomorphism classes of trees of rational curves as above, equipped with a collection of smooth points labelled by $A$, such that each irreducible component has at least three ``special points'', i.e.~nodes and/or marked points.

 We denote by $\M{A}\subset \bM{A}$ the open set parameterizing smooth curves; it is affine.  Its complement is a strict  normal crossing divisor $\bbM{A}\subset\bM{A}$ parameterizing singular curves.  The irreducible components of $\bbM{A}$ are indexed by partitions $A = A' \sqcup A''$ with $1<\# A'' < \#A$: to such a partition corresponds the locus 
 \begin{align}
 D_{A',A''} \subset \bbM{A} \label{eq:boundary-cpt}
 \end{align}
 of nodal curves for which the markings labelled by $A'\sqcup \{\infty\}$ are separated from the markings labelled $A''$ by a node. More generally, the higher-codimension strata of $\bbM{A}$ are indexed by the following data:
\begin{definition}
    The \defn{dual graph $\tau(C)$} of an element $C \in \bM{A}$ is the following rooted tree:
    \begin{itemize}
        \item The root is $\infty$
        \item The leaves are the elements of $A$
        \item The internal vertices are the irreducible components of $C$
        \item The edges are labelled by the special points of $C$
        \item An edge is incident to a vertex if and only if the corresponding special point lies on the corresponding irreducible component.
    \end{itemize}
   \end{definition}
 
 Thus, in particular, $D_{A',A''}$ is the locus of curves whose dual graph contracts to a tree with root labelled $\infty$ and leaves labelled with $A'$ and $A''$ as in  \autoref{fig:divisor-dual-tree}.
 
 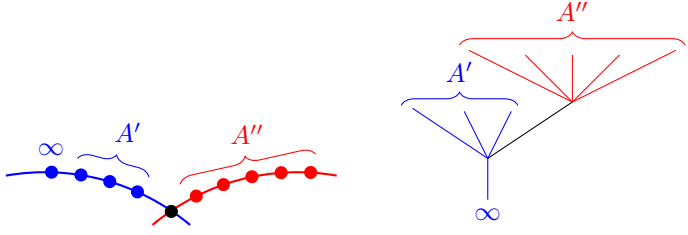
\begin{figure}
     \centering
     \begin{subfigure}{0.4\textwidth}
     \begin{tikzpicture}
     \draw [blue,
    decorate,decoration={brace,amplitude=5pt,raise=4ex}]
  (-0.65,0.3)--(0.3,0.05) node[midway,yshift=3em,xshift=1em,blue]{$A'$};
  \draw[blue] (-0.85,1) node {$\infty$};
        \draw[blue,thick,decoration={markings,mark=between positions 0.35 and 0.9 step 0.15 with \arrow{Circle}},postaction={decorate}] (1,0) arc (50:100:3) ;
        \draw[red,thick,decoration={markings,mark=between positions 0.3 and 0.9 step 0.15 with \arrow{Circle}},postaction={decorate}] (0.5,0) arc (130:80:3) ;
         \draw [red,
    decorate,decoration={brace,amplitude=5pt,raise=4ex}]
  (1,-0.05)--(2.8,0.35) node[midway,yshift=3em,xshift=-0.4em,red]{$A''$};
  \draw[fill] (0.75,0.18) circle (0.08);
     \end{tikzpicture}
     \end{subfigure}
     \begin{subfigure}{0.4\textwidth}
     \begin{tikzpicture}[scale=0.5]
  \node[blue] {$\infty$} [grow'=up]
    child[blue] { 
        child {node {}}
        child {node {}}
        child {node {}}
        child[black] {
            child[red] {node {}}
            child[red] {node {}}
            child[red] {node {}}
            child[red] {node {}}
            child[red] {node {}}
        }
    };
    \draw [blue,
    decorate,decoration={brace,amplitude=5pt,mirror,raise=4ex}]
  (0.8,1.7) -- (-2.3,1.7) node[midway,yshift=3em,blue]{$A'$};
    \draw [red,
    decorate,decoration={brace,amplitude=5pt,mirror,raise=4ex}]
  (5.25,3.3) -- (-0.75,3.3) node[midway,yshift=3em,red]{$A''$};
    \end{tikzpicture}
    \end{subfigure}
     \caption{A curve defining a point of $D_{A',A''}$ (left) and its dual rooted tree (right).}
     \label{fig:divisor-dual-tree}
 \end{figure}
 
\subsubsection{Tangential basepoints on the universal curve}

Let $\pi_A \colon \Crv{A} \to \bM{A}$ be the universal marked curve: if $C$ is a marked curve, then the fibre of $\Crv{A}$ at $[C]\in\bM{A}$ is isomorphic to $C$.  If $*$ is an additional symbol then $\Crv{A} \cong \bM{A\sqcup \{*\}}$ with the projection to $\bM{A}$ given by forgetting the location of $*$.  An element $s \in A\sqcup\{\infty\}$ corresponds to a section $\bM{A} \to \Crv{A}$ indicating the location of the marked point labelled $s$; we shall abuse notation and denote this section also by $s$.  We denote by
\[
\Tinf[s]{A} := \left.
\mathcal{T}_{\Crv{A}/\bM{A}}\right|_{s}
\]
the pullback of the relative tangent bundle of $\pi_A$ via $s$. It is a line bundle on $\bM{A}$ whose fibre at  $[C] \in \bM{A}$ is canonically identified with the tangent space $T_s C$. We also consider
\[
\Tinfx[s]{A} := \Tinf[s]{A} \setminus 0 ,
\]
the complement of the zero section.  Thus a point in $\Tinfx[s]{A}$ is a tangential basepoint at one of the marked points in the universal curve. 

\subsubsection{Universal transport maps}

Given distinct markings $p,q \in A \sqcup\{\infty\}$, we shall explain how to construct a virtual isomorphism between the line bundles $\Tinf[p]{A}$ and $\Tinf[q]{A}$ relative to the normal crossing divisor $(\bM{A},\bbM{A})$, which when applied to virtual points induces the transport of tangential basepoints on each fibre of $\Crv{A}$ for which $p,q$ lie on the same component.

Consider the sheaf $\cO{\Crv{A}}(q-p)$ on the universal curve, where we abuse notation and write $p$ and $q$ for the divisors in the total space of  $\Crv{A}$ given by the images of the corresponding sections $\bM{A} \to \Crv{A}$.  
\begin{lemma}
The sheaf  $\cO{\Crv{A}}(q-p)$ is a trivial line bundle on $\Crv{A}$.  Hence the same is true for the direct image $\pi_{A*}\cO{\Crv{A}}(q-p)$ on $\bM{A}$.  
\end{lemma}
\begin{proof}
    If $r\in A\sqcup \{\infty\}$ is another marked point with $p \neq r$, then we have a well-defined cross-ratio $f(x) := [qp|rx]$ for $x\in \Crv{A}$, giving a rational function on $\Crv{A}$ that vanishes when $x=p$ and has a pole when $x=q$.  This function provides the desired trivialization. The conclusion about the pushforward follows since $\Crv{A} \to \bM{A}$ is a proper morphism with connected fibres, so that $\pi_{A*}\cO{\Crv{A}} \cong \cO{\bM{A}}$.
\end{proof}

Applying the construction of  \autoref{lem:s-tensor} fibrewise on $\Crv{A}$, we deduce that there is a unique section
    \[
    s_{p,q}^A \in  \coH[0]{\bM{A}, T_p^\vee\Crv{A}\otimes T_q^\vee\Crv{A}}
    \]
    such that 
    \[
s^A_{p,q} = \dd f|_p \otimes \dd(\tfrac{1}{f})|_q,
\]
for every nonvanishing local section $f \in \pi_{A*}\cO{\Crv{X}}(q-p)$, where $\dd = \dd_{\Crv{A}/\bM{A}}$ is the relative differential.  Moreover, the vanishing set of $s^A_{p,q}$ is the subdivisor of $\bbM{A}$ consisting of stable curves for which the marked points $p$ and $q$ are separated by a node.  As a consequence, this section defines a virtual isomorphism
\begin{equation}\label{eq:transport-tangential-basepoints-global}
s^A_{p,q}\colon (\bM{A},\bbM{A},\Tinf[p]{A}) \stackrel{\sim}{\to} (\bM{A},\bbM{A},\Tinf[q]{A}).
\end{equation}
of Deligne--Faltings log schemes.

\begin{definition}
    The virtual isomorphisms $s^A_{p,q}$ for $p,q \in A \sqcup\{\infty\}$ are called the \defn{universal transport maps}.
\end{definition}

\subsection{Factorization on the boundary}
From the definition, it is immediate that the universal transport maps restrict to the maps constructed in \autoref{lem:s-tensor} on the smooth fibres of $\Crv{A}$.  We claim that the same is true for the singular fibres, i.e.~that the virtual isomorphism \eqref{eq:transport-tangential-basepoints-global}
factors over the boundary $\bbM{A}$ into compositions of such isomorphisms with appropriate gluing data as in \autoref{cor:transport-tangential-basepoints-through-nodes}. Note that by induction on the number of nodes, the problem reduces to understanding what happens at a single node $\nu$.  Such a node partitions the marked points into two groups: those that are closer to $\infty$ than $\nu$, and those that are farther. We may thus treat the resulting factorization universally as follows.

Let $A =A' \sqcup A''$ be a partition, and consider the divisor
\[
D = D_{A',A''}  \subset \bbM{A}
\]
from \eqref{eq:boundary-cpt} above. Suppose without loss of generality that $p \in A'\sqcup\{\infty\}$ and $q \in A''$; the opposite case follows immediately by inversion.

Let $\nu$ be an additional symbol representing the node.  Then we have an isomorphism
\[
\mapdef{\circ_\nu}{\bM{A'\sqcup\{\nu\}}\times\bM{A''}}{D}
{(C',C'')}{C'\sqcup_\nu C''}
\]
where $C'\sqcup_\nu C''$ is the $A$-marked curve obtained by gluing $\nu \in C'$ to $\infty \in C''$ to form a node, and  marking the images of all $a' \in A'$ and $a'' \in A''$.  To avoid confusion, when we write $\infty$ it will always mean the corresponding point in $C'$; when we need to refer to the point labelled $\infty$ on $C''$, we call it $\nu$ instead.

Let us denote by $\Crv{}=\Crv{A}$ the universal curve on $\bM{A}$, and by $\Crv{}'$ and $\Crv{}''$ the pullbacks to $D$ of the universal curves on $\bM{A'\sqcup\{\nu\}}$ and $\bM{A''}$. Then the recipe above gives a canonical isomorphism 
\[
\Crv{}|_D \cong  \Crv{}' \sqcup_\nu \Crv{}''.
\]
Note that we may view $\Crv{}'$ and $\Crv{}''$ as divisors $D'$ and $D''$ in the total space of $\Crv{}$, whose intersection is the section of $\Crv{}|_D$  indicating the location of the node; see \autoref{fig:degeneration}.  By the above isomorphism, we have the identity
\[
\pi^* D = D' + D''
\]
in the divisor group of $\Crv{}$.

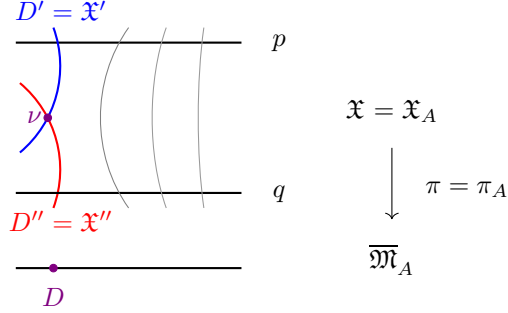
\begin{figure}
    \centering
    \begin{tikzcd}
        \draw[thick] (0,1) -- (3,1) ;
        \draw[blue,thick] (0.5,1.2) arc (20:-50:1.5);
        \draw[red,thick] (0.5,-1.2) arc (-20:50:1.5);
        \draw[thick] (0,-1) -- (3,-1);
        \draw[thick] (0,-2) -- (3,-2);
        \draw[fill,blue!50!red] (0.425,0) circle (0.05);
        \draw[blue!50!red] (0.25,-0.075) node {\nu};
        \draw[fill,blue!50!red] (0.5,-2) circle (0.05);
        \draw[blue!50!red] (0.5,-2.5) node {D};
        \draw[blue] (0.6,1.3) node {D'=\Crv{}'};
        \draw[red] (0.6,-1.5) node {D''=\Crv{}''};
        \draw (3.5,0.95) node {p};
        \draw (3.5,-1.05) node {q};
        \draw (5,0) node {\Crv{}=\Crv{A}};
        \draw[->] (5,-0.4) -- (5,-1.35);
        \draw (6,-1) node {\pi = \pi_A};
        \draw (5,-2) node {\bM{A}};
        \draw[gray] (1.5,1.2) .. controls (1,0.5) and (1,-0.5) .. (1.5,-1.2);
        \draw[black!40!white] (2,1.2) .. controls (1.75,0.5) and (1.75,-0.5) .. (2,-1.2);
        \draw[black!40!white] (2.5,1.2) .. controls (2.4,0.5) and (2.4,-0.5) .. (2.5,-1.2);
    \end{tikzcd}
    \caption{Degeneration of a smooth  curve to a nodal curve with the marked points $p$ and $q$ separated by a node $\nu$, viewed inside the universal curve.  The marked points other than $p$ and $q$ are omitted from the diagram.}
    \label{fig:degeneration}
\end{figure}
 It is a standard fact that the bundle $\Tinf[\nu]{}'\otimes\Tinf[\nu]{}''$ parameterizing the gluing data at $\nu$ is canonically identified  with the normal bundle of $N = \cO{}(D)|_D$, by taking the ``double derivative of the projection at the node''.  More precisely, we have the following.
\begin{lemma}\label{lem:eta-expression}
    There is a unique nonvanishing section
    \[
\eta  = \eta_\nu \in \coH[0]{D, \cO{\bM{A}}(-D)|_D \otimes
\Tinf[\nu]{}' \otimes \Tinf[\nu]{}'' }
\]
with the following property: for every triple of local trivializations
\[
g \in \cO{\bM{A}}(-D) \qquad g' \in \cO{\Crv{}}(-D') \qquad g'' \in \cO{\Crv{}}(-D'')
\]
such that $\pi^* g = g'g''$, we have
\[
\eta =   [g] \otimes  
 \cvf{g''}|_\nu \otimes \cvf{g'}|_\nu 
\]
where $g'$ (respectively $g''$) is viewed as a coordinate on the fibres of $\Crv{}''$ (resp.~$\Crv{}'$) and $[g] \in \cO{\bM{A}}(-D)|_D$ is the restriction of $g$ as a section of $\cO{\bM{A}}(-D)$.
\end{lemma}
\begin{remark}
    An important subtlety of the notation is that $g'$ is a coordinate on $\Crv{}''$, not $\Crv{}'$.  Indeed, by definition, $g'$ vanishes  transversally on all of $\Crv{}'$; it therefore gives a fibrewise coordinate on $\Crv{}''$ that vanishes at the node $\nu$.
\end{remark}

Meanwhile, we have three relevant universal transport maps, namely
\begin{align*}
s &:= s^A_{p,q} & s' &:= s^{A'\sqcup\{\nu\}}_{p,\nu} & s'' &:= s^{A''}_{\nu,q}
\end{align*}
which transport basepoints on $\Crv{}$, $\Crv{}'$ and $\Crv{}''$, respectively.  They are related by the universal gluing datum $\eta$, in the following sense.
\begin{proposition}\label{prop:s-factorization}
    We have the identity
    \[
    s|_D = s'\eta s''
    \]
    as sections of the line bundle $\left.\rbrac{T_p^\vee\Crv{}\otimes T_q^\vee\Crv{}\otimes\cO{\bM{A}}(-D)}\right|_D$
\end{proposition}

\begin{proof}
   Suppose that $f \in \coH[0]{\cO{\Crv{}}(q-p)}$, so that 
   \begin{align*}
   s &= \dd f|_p \otimes \dd(1/f)|_q \in T_p^\vee\Crv{}\otimes T_q^\vee\Crv{} \cong T_p^\vee\Crv{}' \otimes T_q^\vee\Crv{}''
    \end{align*}
    and suppose that  $g, g', g''$ are as in \autoref{lem:eta-expression}, so that
    \begin{align*}
   \eta &= [g] \otimes \cvf{g''}|_\nu \otimes \cvf{g'}|_\nu \in \cO{}(-D)|_D\otimes \Tinf[\nu]{}' \otimes \Tinf[\nu]{}''.
   \end{align*}
   The divisors of these functions (on their domains of definition) are as follows:
   \begin{align*}
      \divf(f) &= p-q+D' \\
      \divf(\pi^*g) &= \pi^* D = D'+D'' \\
      \divf(g') &= D' \\
      \divf(g'') &= D''.
   \end{align*}
   Thus, if we define functions on $\Crv{}'$ and $\Crv{}''$ by 
   \[
   f' := (f/\pi^* g)|_{\Crv{}'} \qquad f'' := f|_{\Crv{}''}
   \]
   then
   \[
   f' \in \pi'_*\cO{\Crv{}'}(\nu-p)\qquad \textrm{and} \qquad f'' \in \pi''_*\cO{\Crv{}''}(q-\nu)
   \]
   where $\pi' \colon \Crv{}'\to D$ and $\pi''\colon \Crv{}''\to D$ are the projections. We therefore have
   \begin{align*}
   s' &= \dd f'|_p \otimes \dd(\tfrac{1}{f'})|_\nu \in T_p^\vee\Crv{}' \otimes T_\nu^\vee\Crv{}'
   \end{align*}
   and 
   \begin{align*}
   s'' &= \dd f''|_\nu \otimes \dd(\tfrac{1}{f''})|_q \in T_\nu^\vee\Crv{}'' \otimes T_q^\vee\Crv{}''.
   \end{align*}
   The composition $s'\eta s''$ is obtained by combining the corresponding tensors via the contractions on the (co)tangent spaces at $\nu$.  Thus we have
   \[
   s'\eta s'' = g \otimes  (\partial_{g''}\tfrac{1}{f'})|_\nu \cdot (\partial_{g'}f'')|_\nu \cdot \dd f'|_p \otimes \dd f''|_q \in \cO{\bM{A}}(-D)|_D \otimes  T_p^\vee\Crv{}'\otimes T_q^\vee\Crv{}'
   \]
   We will prove that this expression is equal to $s|_D$ by a suitable coordinate change, as follows.
   
   Note that since $\Crv{A}$ is locally factorial and $f$ vanishes on $D'$, there is a unique regular function $h$ on $\Crv{A}$ such that
   \[
   f = hg'.
   \]
   on the domain of $g'$.  Then the divisor of $h$ is 
   \[
   \divf(h) = \divf(f)-\divf(g') = q-p
   \]
   so that $h$ and $h^{-1}$ are regular and nonvanishing at $\nu$.  Moreover, since $\pi^* g = g'g''$, the definitions of $f'$ and $f''$ give
   \begin{align*}
   f' &= \tfrac{h}{g''}|_{D'} & f'' &= hg'|_{D''}.
   \end{align*}
   Therefore we have
   \begin{align*}
   \cvf{g''}(\tfrac{1}{f'})|_\nu = \left.\rbrac{\tfrac{1}{h}+\tfrac{g''}{h^2}\cvf{g''}h}\right|_\nu = \tfrac{1}{h}|_\nu
   \end{align*}
   since $g''|_\nu=0$.  Similarly, we have
   \[
   \cvf{g'}(f'')|_\nu = (h + g'\cvf{g'}h)|_\nu = h|_\nu.
   \]
   Therefore we have
   \begin{align*}
   s'\eta s'' &= g \otimes   \tfrac{1}{h} \cdot h \cdot \dd f'|_p \otimes \dd(\tfrac{1}{f''})|_q \\
   &= g \otimes \dd f'|_p \otimes \dd f''|_q \\
   &= g \otimes \dd(\tfrac{f}{\pi^*{g}})|_p \otimes \dd(\tfrac{1}{f})|_q \\
   &= [g] \otimes \tfrac{1}{\pi^*g|_p} \dd f|_p \otimes \dd(\tfrac{1}{f})|_q \\
   &= \left.\rbrac{\frac{g}{\pi^*g|_p}\dd f|_p \otimes \dd(\tfrac{1}{f})|_q}\right|_D \\
   &= \left.\rbrac{\dd f|_p \otimes \dd(\tfrac{1}{f})|_q}\right|_D \\
   &= s|_D
   \end{align*}
   as desired, where for the fourth equality we have used that $d$ is the fibrewise differential, so that $\dd(\pi^*g)=0$, and for the sixth equality, we have used that $p$ is a section, so that $\pi^* g|_p = g$.
\end{proof}

\subsection{Operads of curves with tangential basepoints}\label{sec:operads-moduli-spaces}
We are now in a position to formulate our construction of a logarithmic lift of the little disks operad.

\begin{definition}
    For a finite set $A$, the \defn{logarithmic Fulton--MacPherson space} is the Deligne--Faltings log scheme
    \[
    \FM{A} := (\bM{A},\bbM{A},\Tinf{A})
    \]
    and the \defn{framed logarithmic Fulton--MacPherson space} is the Deligne--Faltings log scheme
    \[
    \FMfr{A} := \rbrac{\bM{A},\bbM{A},\bigoplus_{p \in A \sqcup\{\infty\}} \Tinf[p]{A}}
    \]
\end{definition}
Note that these log schemes are strata of normal crossing divisors, and hence their virtual points are tangential basepoints of $(\bM{A},\bbM{A})$ together with nonzero elements in the corresponding line bundles.  Since the normal bundles of the components of $\bM{A}$ parameterize gluing data by \autoref{lem:eta-expression}, it follows immediately that these log schemes have the following modular interpretation:

\begin{proposition}
    If $\KK$ is any field, then the virtual $\KK$-points $\FM{A}(\KK)$ are in natural bijection with isomorphism classes of stable $(A \sqcup\{\infty\})$-marked curves of genus zero over $\KK$, equipped with a gluing datum at every node, and a tangential basepoint at $\infty$.  Virtual points of $\FMfr{A}(\KK)$ are the same, but with the additional datum of a tangential basepoint at every marked point $a \in A$.
\end{proposition}

The framed versions were defined previously by Vaintrob~\cite{Vaintrob2021}, who denoted them by $FLC$, and  explained how to lift the gluing maps for stable curves to \emph{ordinary} morphisms between the log schemes $\FMfr{A}$, making them into an operad in $\fsLogSch$, which we recall below.  We shall do something similar for the unframed versions $\FM{A}$, but now the gluing maps will be \emph{virtual}.

\subsubsection{Composition of stable curves}
Given finite sets $A$ and $B$, and an additional symbol $\nu$, we have the gluing map of schemes
\begin{equation}\label{eq:operadic-composition-level-of-schemes}
\circ_\nu \colon \bM{A\sqcup\{\nu\}} \times \bM{B} \to \bM{A\sqcup B}
\end{equation}
where we attach the marked point $\nu$ on the first curve to the marked point $\infty$ on the second as above. It is the closed embedding of the component $D_{A,B}$ of $\bbM{A\sqcup B}$.

\subsubsection{Framed composition}

Following Vaintrob \cite{Vaintrob2021} we explain how to upgrade \eqref{eq:operadic-composition-level-of-schemes} to an ordinary morphism
\[
\circ_\nu \colon \FMfr{A\sqcup\{\nu\}} \times \FMfr{B} \to \FMfr{A\sqcup B}.
\]
This amounts to defining an ordinary morphism between two Deligne--Faltings log structures over the normal crossing divisor $(\bM{A\sqcup\{\nu\}},\bbM{A\sqcup\{\nu\}}) \times (\bM{B},\bbM{B})$: the product structure $\FMfr{A\sqcup\{\nu\}} \times \FMfr{B}$ and the pullback log structure $\circ_\nu^* \FMfr{A\sqcup B}$.  As explained in \autoref{ex:pullback-log-structure-to-divisor}, the latter is induced by the pullbacks of the line bundles defining $\FMfr{A\sqcup B}$, along with the normal bundle $N$ of the embedding \eqref{eq:operadic-composition-level-of-schemes}.
By \autoref{prop:morphisms-as-monomial-maps} we are therefore looking for an ordinary monomial map
\begin{equation*}
\left(\bigoplus_{x\in A\sqcup\{\nu\}\sqcup\{\infty\}}T_x\Crv{A\sqcup\{\nu\}}\right)  \oplus \left(\bigoplus_{y\in B\sqcup\{\infty\}}T_y\Crv{B}\right)  \to N\oplus \left(\bigoplus_{z\in A\sqcup B\sqcup \{\infty\}} (\circ_\nu)^* T_z\Crv{A\sqcup B}\right)
\end{equation*}
between split vector bundles.
Here we abuse notation and denote by the same symbol a line bundle on a factor of a product and its pullback to the product. Taking into account the isomorphisms $(\circ_\nu)^*T_x\Crv{A\sqcup B}\cong T_x\Crv{A\sqcup\{\nu\}}$ for $x\in A\sqcup\{\infty\}$ and $(\circ_\nu)^*T_y\Crv{A\sqcup B}\cong T_y\Crv{B}$ for $y\in B$, the sought-for ordinary monomial map looks like:
\begin{equation*}
\begin{array}{c}
\Tinf[\nu]{A\sqcup\{\nu\}}  \oplus T_\infty\Crv{B} \\ 
\oplus \\
    \bigoplus_{x\in A \sqcup\{\infty\}}\Tinf[x]{A\sqcup\{\nu\}} \\
    \oplus \\
    \bigoplus_{y\in B}T_y\Crv{B}
\end{array} \to \begin{array}{c}
    N \\
    \oplus \\
    \bigoplus_{x\in A\sqcup\{\infty\}}\Tinf[x]{A\sqcup \{\nu\}} \\
    \oplus \\
    \bigoplus_{y\in B}T_y\Crv{B} 
\end{array}
\end{equation*}
We define it by combining the obvious identity maps in the bottom two rows with the quadratic monomial map
\[
T_\nu\Crv{A\sqcup\{\nu\}} \oplus T_\infty\Crv{B} \to T_\nu\Crv{A\sqcup\{\nu\}}\otimes T_\infty\Crv{B} \stackrel{\eta_\nu}{\cong} N
\]
induced by the universal gluing datum $\eta_\nu$. The naturality of $\eta$, and the fact that it only depends on the local behaviour at each node, ensures that these morphisms satisfy the associativity conditions required to give the collection of log schemes $\FM{}$ the structure of an operad in the category $\fsLogSch$ of fs log schemes and ordinary morphisms.

\begin{remark}
The above construction can be performed for moduli spaces of curves of any genus, as already explained in a different language by Levine \cite[\S 11]{LevineTubular}, see also \cite[\S 8]{BDPW}. In contrast, the \emph{unframed} composition that we will now define crucially uses the fact that we are working with genus zero curves, for which tangential basepoints can be transported.
\end{remark}

\subsubsection{Unframed composition}

The gluing map of Vaintrob cannot be applied to the unframed versions $\FM{A}$, since the line bundle $\Tinf[\nu]{A\sqcup\{\nu\}}$ of tangent vectors at the gluing point is not part of the log structure.  To deal with this, we use the universal transport of tangential basepoints, which gives a \emph{virtual} isomorphism
\[
s^{A\sqcup\{\nu\}}_{\infty,\nu} \colon \Tinf{A\sqcup\{\nu\}} \to \Tinf[\nu]{A\sqcup\{\nu\}}
\]
over the pair $(\bM{A\sqcup\{\nu\}},\bbM{A\sqcup\{\nu\}})$. We now upgrade \eqref{eq:operadic-composition-level-of-schemes} to a \emph{virtual} morphism
\[
\circ_\nu \colon \FM{A\sqcup\{\nu\}} \times \FM{B} \to \FM{A\sqcup B}.
\]
Similarly to the framed case, this amounts by \autoref{prop:morphisms-as-monomial-maps} to a \emph{virtual} monomial map
\begin{equation*}\label{eq:operadic-composition-virtual-monomial-map}
T_\infty\Crv{A\sqcup\{\nu\}} \oplus T_\infty\Crv{B} \to  T_\infty\Crv{A\sqcup\{\nu\}} \oplus N
\end{equation*}
between split vector bundles on the pair $(\bM{A\sqcup\{\nu\}},\bbM{A\sqcup\{\nu\}})\times (\bM{B},\bbM{B})$. It is obtained from combining the identity map of $T_\infty\Crv{A\sqcup\{\nu\}}$ and the composite
\[
\begin{tikzcd}[column sep = 2em]
    \Tinf{A\sqcup\{\nu\}} \oplus \Tinf{B} \ar[rr,"s^{A\sqcup\{\nu\}}_{\infty,\nu}\oplus \id{}"] && \Tinf[\nu]{A\sqcup\{\nu\}}\oplus \Tinf{B} 
    \stackrel{\eta_\nu}{\cong} N.
\end{tikzcd}
\]
The following statement is readily seen to follow from the factorization property of the universal transport maps proven in \autoref{prop:s-factorization}:
\begin{proposition}
    The morphisms $\circ_\nu$ are associative, making $\FM{}$ into an operad in the category $\fsWLogSch$.
\end{proposition}

\begin{remark}\label{rem:FM-as-free-set-operad}
The virtual monomial map \eqref{eq:operadic-composition-virtual-monomial-map} is a virtual isomorphism, which implies that the operadic composition induces a virtual isomorphism of log schemes
\[
\FM{A\sqcup\{\nu\}} \times \FM{B} \stackrel{\sim}{\to} \FM{A\sqcup B}|_{D_{A,B}} \hookrightarrow \FM{A\sqcup B}.
\]
lifting the inclusion $\bM{A\sqcup \{\nu\}}\times\bM{B} \stackrel{\sim}{\rightarrow} D_{A,B} \hookrightarrow \bM{A\sqcup B}$.
\end{remark}

Moreover, the maps
\[
s^A_{\infty,a} \colon \Tinf{A} \to \Tinf[a]{A}
\]
for $a \in A$, together with the identity map on $\Tinf{A}$, give a virtual morphism
\[
s^A \colon \FM{A} \to \FMfr{A}.
\]
At the level of virtual points, this takes a curve with gluing data and a tangential basepoint at $\infty$, and further decorates it with tangential basepoints at all other marked points, by transporting the basepoint at $\infty$ using the gluing data.  It is therefore compatible with the operadic compositions:
\begin{proposition}
    The universal transport maps $s^A_{\infty,a}$ for $a \in A$ induce a morphism  $\FM{} \to \FMfr{}$ of operads in $\fsWLogSch$.
\end{proposition}

\subsection{Relation to configuration spaces and little disks}
We now give a more explicit description of the log schemes $\FM{A}$  in terms of configuration spaces, which explains their relation to the usual Fulton--MacPherson model for the little disks operad.   

\subsubsection{The operad of Kato--Nakayama spaces}

Note that since the functor $\KN{-}$ is symmetric monoidal, the spaces $\KN{\FM{}}$ form a topological operad. As for the Kato--Nakayama space of any Deligne--Faltings log scheme over a normal crossing divisor (\autoref{ex:DF-KN}), the points of $\KN{\FM{A}}$ are the quotient of the set of virtual $\CC$-points $\FM{A}(\CC)$ by real dilations in all the line bundles and normal bundles.  Concretely, this means that elements in $\KN{\FM{A}}$ are in bijection with isomorphism classes of the following data:
\begin{itemize}
    \item A stable $A$-marked curve $C$ of genus zero over $\CC$.
    \item A real ray in the tangent space $T_\infty C$
    \item A real ray in the gluing space $T_\nu C_\nu' \otimes T_\nu {C''_\nu}$ at every node $\nu \in C$
\end{itemize}
The operadic composition then descends from that of $\FM{A}$.  Namely, to compose an element $C_1 \in \KN{\FM{A\sqcup\{\nu\}}}$ with an element $C_2 \in \KN{\FM{B}}$, one glues the curves $C_1$ and $C_2$ at the points $\nu \in C_1$ and $\infty \in C_2$ to obtain a nodal curve $C$.  Then  one transports the tangent ray at $\infty$ on $C_1$ to a tangent ray at $\nu$ using $s_{\infty,\nu}$ and ``tensors'' it with the tangent ray at $\infty$ on $C_2$ to obtain a ray in the gluing space $T_\nu C_1 \otimes T_\infty C_2$ of $C$.

\begin{example}\label{ex:gluing-rays}
    For $j=1,2$, let $C_j=\PP^1(\CC)$ with coordinate $z_j$.  Then a tangent ray $\rho_j \subset T_\infty C_j$ has the form
    \[
    \rho_j = \RR_{>0} e^{\iu \theta_j}\cvf{1/z_j} \subset T_{\infty}C_j
    \]
    for some angle $\theta_j \in \RR/2\pi\ZZ$.   From the formula in \autoref{ex:transport-tangential-basepoints-easy} for the transport of tangential basepoints, we have
    \[
    s_{\infty,0}(\rho_1) = \RR_{>0} e^{-\iu \theta_1}\cvf{z_1} \subset T_0 C_1
    \]
    Hence if we glue $0 \in C_1$ to $\infty \in C_2$ and combine their tangent rays by the operadic composition above, the resulting gluing datum at the node is given by
    \[
    \eta = s_{\infty,0}(\rho_1)\otimes\rho_2 = \RR_{>0} e^{\iu(\theta_2-\theta_1)} \rbrac{\cvf{z_1}|_0 \otimes \cvf{1/z_2}|_\infty}.
    \]
    It can thus be thought of as a measure of the difference in angle between the tangent rays on $C_1$ and $C_2$.
\end{example}

\subsubsection{The topological Fulton--MacPherson operad}

Let us recall the construction of the topological Fulton--MacPherson operad $\FMtop{}$, which is discussed for instance in \cite{KontsevichOperadsMotives,LambrechtsVolic,Markl1999,Salvatore2001}.

Let us denote by 
\[
\Conf_A(\CC) := \set{(z_a)_{a \in A} \in \CC^{A} }{z_a \neq z_b \textrm{ if }a\neq b}
\]
the configuration space parameterizing distinct points in $\CC$ labelled by the elements of $A$, i.e.~the set of injections $A \hookrightarrow \CC$.  It carries a canonical action of the group $\RR_{>0} \ltimes \CC$ of translations and real dilations of $\CC$, and the quotient $\Conf_A(\CC)/(\RR_{>0}\ltimes \CC)$ has a canonical Fulton--MacPherson compactification
\[
\FMtop{A} := \overline{\Conf_A(\CC)/(\RR_{>0}\ltimes \CC)}
\]
which is a real manifold with corners, whose boundary strata  parameterize ways in which configurations of points can degenerate.  The points of $\FMtop{A}$ are given by the following data:
\begin{itemize}
    \item A rooted tree $\tau$ whose leaves are labelled by the elements of $A$
    \item For each internal vertex $v$ of $\tau$, an element $Z_v \in \Conf_{A(v)}(\CC)/(\RR_{>0}\ltimes \CC)$, where $A(v)$ denotes the set of edges that point away from the root at $v$.
\end{itemize}
These spaces assemble into an operad $\FMtop{}$ whose operadic composition is given simply by the grafting of trees (inserting the root of one into a leaf of another).

\subsubsection{The relation between them}
We now relate the topological operads $\KN{\FM{A}}$ and $\FMtop{A}$ discussed above.

Note that the configuration \emph{space} $\Conf_A(\CC)$ is the space of $\CC$-points (with the classical topology) of the configuration \emph{scheme}
\[
\Conf_A(\AF^1) := \set{(z_a)_{a\in \AF^1}}{z_a \neq z_b \textrm{ if } a \ne b}.
\]
The latter is a smooth scheme of finite type over $\Spec{\ZZ}$ and comes equipped with an action of the group scheme $\Gm\ltimes \Ga$ by translations and dilations.

Given a configuration $(z_a)_{a \in A}$ in $\AF^1$, we may form a smooth marked curve of genus zero by adding the point at infinity to obtain $C:=(\PP^1,(z_a)_{a \in A},\infty)\in\M{A}$.  We may further equip it with the tangent vector $\cvf{1/z}|_\infty \in T_\infty \PP^1$ where $z$ is the standard coordinate on $\AF^1$.  Evidently every smooth marked curve with a tangential basepoint at infinity is isomorphic to one obtained in this way.  Moreover, the only automorphisms of $\PP^1$ fixing a nonzero tangent vector at $\infty$ are the translations, so two such curves with a tangential basepoint are isomorphic if and only if the corresponding configurations differ by a translation.  Thus the quotient of $\Conf_A(\AF^1)$ by translations is identified with the moduli space of smooth $A$-marked curves with a nonzero tangent vector at $\infty$, i.e.~we have a canonical isomorphism
\begin{equation}
\begin{tikzcd}
\Conf_A(\AF^1)/\Ga \ar[r,"\sim"] &\Tinfx{A}|_{\M{A}}
\end{tikzcd} \label{eq:conf-to-T}
\end{equation}
of schemes over $\Spec{\ZZ}$.

In addition, we have the canonical virtual map $\Tinfx{A}|_{\M{A}}  \to \FM{A}|_{\M{A}}$ which trades the fibre coordinate on $\Tinfx{A}|_{\M{A}}$ for the corresponding phantom coordinate on $\FM{A}$, and is thus a bijection on the level of virtual points.  On the level of Kato--Nakayama spaces, it gives the map
\begin{equation}
\begin{tikzcd}
    (\Tinfx{A}|_{\M{A}})(\CC) \ar[r] &\KN{\FM{A}|_{\M{A}}}
\end{tikzcd}\label{eq:T-to-FM}
\end{equation}
that sends a pair $(C,v)$ of a smooth complex curve $C \in \M{A}(\CC)$ and a nonzero tangent vector $v \in (T_\infty C)^\times$ to the pair $(C,\RR_{>0}v)$ given by the same curve and the tangent ray $\RR_{>0}v$ spanned by $v$.  It is thus the quotient by the canonical action of $\RR_{>0}$ by rescaling the tangent vector.  The latter corresponds to the action of $\RR_{>0}$ on $\Conf_A(\CC)$ by dilation of configurations, and hence the maps \eqref{eq:conf-to-T} and \eqref{eq:T-to-FM} combine to induce a homeomorphism
\begin{equation}
\begin{tikzcd}
\Conf_A(\CC)/(\RR_{>0}\ltimes \CC) \ar[r,"\sim"] & \KN{\FM{A}|_{\M{A}}}.
\end{tikzcd} \label{eq:conf-to-FM}
\end{equation}

\begin{theorem}\label{prop:conf-vs-fm}
    The map \eqref{eq:conf-to-FM} extends uniquely to the Fulton--MacPherson space $FM_A =\overline{\Conf_A(\CC)/(\RR_{>0}\ltimes \CC)}$, giving an isomorphism of topological operads
    \[
    \FMtop{} \cong \KN{\FM{}}.
    \]
\end{theorem}

\begin{proof}[Sketch of proof]   
    We will follow a tradition in the literature and describe the resulting isomorphism of operads of sets, while omitting the routine verification of its compatibility with the topology.

    Note that by definition, the set operad $FM$ is the free operad generated by the sequence $\Conf_\bullet(\CC)/(\RR_{>0}\ltimes\CC)$. Also, \autoref{rem:FM-as-free-set-operad} implies that the set operad $\KN{\FM{}}$ is the free operad generated by the sequence $\KN{\FM{}|_{\M{}}}$. Therefore, \eqref{eq:conf-to-FM} extends uniquely to an isomorphism of set operads 
    \begin{equation}\label{eq:FM-to-KN-in-proof}
    FM \stackrel{\sim}{\to} \KN{\FM{}},
    \end{equation}
    and one checks that its components are continuous, hence homeomorphisms because all spaces involved are compact and Hausdorff.
\end{proof}

\begin{remark}
The isomorphism from \autoref{prop:conf-vs-fm} can be described explicitly as follows. Given a point of $\FMtop{A}$ consisting of a tree $\tau$ and a configuration $Z_{v} \in \Conf_{A(v)}(\CC)$ at each vertex $v$, we obtain a collection of smooth curves $C_v := (\PP^1,Z_v,\infty)$ which we decorate with the tangent ray spanned by the unit tangent vector $\RR_{>0} \cvf{1/z}|_\infty \subset T_\infty \PP^1$.  In the notation of \autoref{ex:gluing-rays}, this corresponds to the case in which all angles $\theta_j$ are equal to zero. Therefore, the image of $(\tau,(Z_v)_v)$ by \eqref{eq:FM-to-KN-in-proof} is the curve obtained by gluing the $C_v$ according to $\tau$ and decorating with the gluing ray $\RR_{>0}\cdot (\cvf{z}|_\nu\otimes \cvf{1/z}|_\infty)$ at every node.  Every point of $\KN{\FM{}}$ is represented by a decorated curve of this type, unique up to dilations and translations of the configurations $Z_v$.
\end{remark}

\subsection{de Rham cohomology, Arnol'd forms and formality}
Parallel to the ``Betti'' description above, we may give an explicit description of the de Rham algebra of $\FM{A}$ (even over $\Spec{\ZZ}$) as follows. 

\subsubsection{Forgetful maps}
Recall that if $A' \subset  A$ is a subset with at least two elements, we have a forgetful map $\bM{A} \to \bM{A'}$ defined by forgetting the location of the marked points in $A\setminus A'$ and stabilizing the resulting curve.  Over the locus of smooth curves, this identifies the tangent spaces at infinity on $\Crv{A}$ and $\Crv{A'}$, and thus it induces a virtual morphism $\FM{A} \to \FM{A'}$. 

\subsubsection{Logarithmic differentials}
In particular, if $\{a,b\}  \in A$ is a subset with two elements, we have a canonical forgetful map
\[
f_{\{a,b\}} : \FM{A} \to \FM{\{a,b\}}.
\]
Note that $\FM{\{a,b\}}$ is the log structure over $\bM{\{a,b\}}\cong \pt$ associated to the line $T_\infty\Crv{\{a,b\}}$.  A choice of $\ZZ$-linear coordinate $t$ on this tangent space gives an isomorphism $\FM{\{a,b\}} \to  \logpt$, so that
$
\forms[1]{}(\FM{\{a,b\}}) \cong \ZZ\cdot \dlog{t}.
$
Note that the generator $\dlog{t}$ is invariant under change of the phantom coordinate $t$.  We deduce that the pullback
\[
\omega_{\{a,b\}}  := f_{\{a,b\}}^*\ddlog{t} \in \forms[1]{}(\FM{A})
\]
is independent of the choice of $t$,  i.e.~it depends only on the choice of the two-element subset $\{a,b\}\subset A$ as suggested by the notation.  We can think of it as the ``logarithmic derivative of the universal tangential basepoint relative to the positions of $a$ and $b$''.

Now consider the pullback of these elements along the canonical virtual morphism
\[
\phi : \Conf_A(\AF^1) \to \FM{A}
\]
constructed in the previous section, sending the configuration $(z_a)_{a \in A}$ to the curve $(\PP^1,(z_a)_{a\in A},\infty)$ equipped with the tangential basepoint $e:=\cvf{1/z}|_\infty$.  The composition $f_{\{a,b\}} \circ \phi$ then sends a configuration $(z_a)_{a \in A}$ to the marked curve $(\PP^1,z_a,z_b,\infty)$ with the same tangential basepoint $e$.  Translating by $-z_a$ and dilating by $\frac{1}{z_b-z_a}$, the latter curve is isomorphic to  $(\PP^1,0,1,\infty)$ with the tangential basepoint $(z_b-z_a)e$.  Taking the phantom coordinate $t$ dual to $e$, we deduce that 
\[
\phi^* f_{\{a,b\}}^* t = z_b-z_a
\]
and therefore
\[
\phi^*\omega_{\{a,b\}} = \frac{\dd(z_b-z_a)}{z_b-z_a} \in \forms[1]{\Conf_A(\AF^1)}.
\]
Note that these differential forms are exactly the ones considered by Arnol'd in \cite{Arnold1969}, where he proved that (after dividing each by $\tipi$), the subalgebra they generate projects isomorphically to the integral cohomology of the configuration space $\Conf_A(\CC)$.  

\subsubsection{Formality}
\label{sec:formality}
 Recall that the functor $X \mapsto \sect{X,\forms{X/\KK}}$ of global log forms over a field $\KK$ is strong  symmetric monoidal.  Thus, by base change of $\FM{}$ to $\QQ$, we  obtain a differential graded Hopf cooperad $\Omega^\bullet(\FM{}) := \sect{\FM{}/\QQ,\forms{\FM{}/\QQ}}$. It turns out that its differential vanishes:

\begin{proposition}
The following statements hold:
\begin{enumerate}
    \item All global logarithmic forms on $\FM{}$ are closed, and the natural map
    \begin{equation}\label{eq:formality-FM-de-Rham}
    \Omega^\bullet(\FM{})\to \HdR{\FM{}}
    \end{equation}
    is an isomorphism of graded Hopf cooperads.
    \item The elements $\omega_{a,b} \in \forms[1]{}(\FM{A})$ generate $\forms{}(\FM{A})$ as an algebra over $\QQ$.  
\end{enumerate}
\end{proposition}

\begin{proof}
The vector bundle $T_\infty\Crv{A}$ on $\bM{A}$ is trivial in the interior $\M{A}$. Indeed, choosing two distinct marked points $p,q\in A$, the cross-ratio $f(z)= [z p | \infty q]$ is a global function on $\Crv{A}|_{\M{A}}$ whose vanishing locus is $\infty$, so that its fibrewise differential trivializes $\Tinf{A}|_{\M{A}}$. Therefore, by \autoref{ex:automorphisms-DF-log-structures} we get a virtual isomorphism of log schemes
\[
\FM{A} \cong (\bM{A},\bbM{A})\times \logpt.
\]
It is a standard fact that all global log forms on $(\bM{A},\bbM{A})$ are closed and that the natural map
\[
\Omega^\bullet(\bM{A},\bbM{A})\to \HdR{\bM{A},\bbM{A}}\cong \HdR{\M{A}}
\]
is an isomorphism, as a classical consequence of the fact that the mixed Hodge structure on $\coH{\M{A}}$ is pure of weight $2k$ for all $k$. The same holds for $\logpt$, and statement (1) follows. Statement (2) then follows from Arnol'd's computation of the de Rham cohomology of the configuration spaces $\Conf_A(\AF^1)$.
\end{proof}

On the other hand, the cochains
\[
\rsect{\KN{\FM{}};\CC} \cong \rsect{\FMtop{};\CC}
\]
form a weak dg Hopf cooperad; see, e.g.~\cite[\S8.5]{CiriciHorelMHS} for a precise definition of this concept.  
Applying  Betti and de Rham cochains and their comparison equivalence, the isomorphism \eqref{eq:formality-FM-de-Rham} fits into the following diagram of homotopy equivalences  of weak dg Hopf co-operads:
\begin{equation}
\begin{tikzcd}
\Omega^\bullet(\FM{})\otimes_\QQ\CC \ar[r,"\cong"] \ar[d,"\sim"] & \HdR{\FM{}}\otimes_\QQ\CC 
    \ar[dd,"\cong"] \\
    \rsect{\FM{},\forms{\FM{}/\QQ}} \otimes_\QQ \CC \ar[d,"\sim"]\\
    \rsect{\KN{\FM{}};\CC} & \HB{\FM{}}\otimes_\QQ\CC \ar[l,dashed]
\end{tikzcd}\label{eq:formality-square}
\end{equation}
The dashed horizontal map is the one that makes the diagram commute; it induces the identity on cohomology and thus establishes the formality of the weak dg Hopf cooperad  $\rsect{\KN{\FM{}};\CC} \cong \rsect{\FMtop{};\CC}$.   Combined with the homotopy equivalence between $\FMtop{}$ and the little disks operad~\cite{KontsevichOperadsMotives,LambrechtsVolic,Markl1999,Salvatore2001}, we arrive at the following statement, known in the literature as ``formality of little disks''; see the introduction for a discussion of its history.

\begin{corollary}
The weak dg Hopf co-operad of cochains on the little disks operad, with $\CC$-coefficients, is formal.
\end{corollary}

\begin{remark}
There is a canonical integral structure  $\forms{}(\FM{})_\ZZ \subset \Omega^\bullet(\FM{})$ given by the global log forms on $\FM{}$ defined over $\ZZ$, or equivalently the $\ZZ$-subalgebra generated by the forms $\omega_{\{a,b\}}$. With this choice of integral structure, the Betti--de Rham comparison gives an isomorphism  $\forms[k]{}(\FM{})_\ZZ  \cong \HB[k]{\FM{};\ZZ (\tipi)^k}$, where the factor $(\tipi)^k$ appears in degree $k$ because the integral of $\ddlog{z_a-z_b}$ over a loop is a multiple of $\tipi$.  Hence the comparison isomorphism in cohomology can be renormalized to be defined over $\ZZ$.  (The corresponding statement for $\FMfr{}$ was observed also by Vaintrob~\cite{Vaintrob2021}.)  At cochain level, the situation is more subtle, since the integral of $\ddlog{z_a-z_b}$ along a non-closed path could be any complex number.  However, using sheaves of polylogarithms as in \cite{BPP,Brown2009,Goncharov2001}, one can show that the formality morphism \eqref{eq:formality-square} with $\CC$-coefficients can be lifted, in a canonical way, to one whose coefficients are multiple zeta values.
\end{remark}

\begin{remark}
Under the identification of $\forms{}(\FM{})$ with Arnol'd's forms on the configuration spaces $\Conf(\AF^1)$, the cooperadic cocomposition amounts to a canonical way to pull back the forms $\tfrac{\dd(z_i-z_j)}{z_i-z_j}$ the loci of collisions $z_i=z_j$, even though they have poles there.  Note that the existence of such a ``regularized pullback'' is a consequence of the homotopy equivalence of the configuration spaces with the spaces of little disks and Arnol'd's presentation of $\coH{\Conf(\AF^1)}$.  What is new here is that our formalism gives a direct algebro-geometric construction of the regularized pullback, which  implies the formality of little disks.
\end{remark}

\subsection{The operad of integral points}
Note that if $X$ is an operad of log schemes over $\ZZ$, then by taking virtual $\ZZ$-points we obtain an operad $X(\ZZ)$ in the category of sets; the operadic structure is naturally induced from that of $X$ because the functor of virtual points is symmetric monoidal.  Our aim now is to give a purely combinatorial description of the operads of virtual $\ZZ$-points $\FM{}(\ZZ)$.

\subsubsection{Binary trees}

Let $A$ be a finite set with $|A|\geq 2$. Consider the subscheme $\overline{\mathfrak{M}}_A^{\mathrm{bin}}\subset \bM{A}$ consisting of those stable marked curves $C$ whose dual graph $\tau(C)$ is a \emph{binary} tree, or in other words such that all the irreducible components of $C$ have exactly three ``special points'' (nodes and/or markings). It is the intersection of $|A|-2$ irreducible components of the divisor $\bbM{A}$, and is a disjoint union of copies of $\Spec{\ZZ}$ indexed by all binary rooted trees whose leaves are labelled by $A$ and whose root is labelled $\infty$.

\begin{proposition}\label{prop:virtual-Z-points-of-FM-are-binary}
Any virtual $\ZZ$-point of $\FM{A}$ has an underlying morphism which factors through $\overline{\mathfrak{M}}_A^{\mathrm{bin}}(\ZZ)$.
\end{proposition}

\begin{proof}
Since there is an ordinary morphism of log schemes from $\FM{A}$ to the divisorial log scheme $(\bM{A},\bbM{A})$ whose underlying morphism of schemes is the identity, it is enough to prove the analogous statement for $(\bM{A},\bbM{A})$.

Let $s:\Spec{\ZZ}\to (\bM{A},\bbM{A})$ be a virtual morphism of log schemes. For $S\subset A$ with $|S|=3$, the forgetful map $\bM{A}\to \bM{S}$ lifts to an ordinary morphism of log schemes $f_S\colon (\bM{A},\bbM{A})\to (\bM{S},\bbM{S})\cong (\PP^1_\ZZ,\{0,1,\infty\})$. \autoref{lem:virtual-Z-points-of-projective-line-three-points} then implies that $s\circ f_S$ has an underlying morphism which lands at one of the three $\ZZ$-points $0, 1,\infty$, i.e., in the locus $\bbM{S}$ of singular marked curves. This implies the claim since for a closed point $C$ of $\bM{A}$ whose dual graph is \emph{not} a binary tree, there exists a $3$-element subset $S\subset A$ such that the forgetful map $\bM{A}\to \bM{S}$ sends $C$ to a smooth marked curve.
\end{proof}

\subsubsection{Tangential basepoints and orderings}

We have a projection $\FM{A} \to \bM{A}$, defined by forgetting the log structure.  For $C \in \overline{\mathfrak{M}}_A^{\mathrm{bin}}(\ZZ)$, a lift of $C$ to an element of $\FM{A}(\ZZ)$ is equivalent to the specification of a tangential basepoint at $\infty$ on $C$ and a gluing datum at every node, all defined over $\Spec{\ZZ}$.  By transport of tangential basepoints through the nodes, this is equivalent to specifying, for each component $C_i$ of $C$, a tangential basepoint at the special point of  $C_i$ that is closest to $\infty$, which we denote by $\infty_i\in C_i(\ZZ)$.

Note that for each component $C_i$, the real locus $C_i(\RR) \cong \PP^1(\RR)$ (with the analytic topology) is a topological circle. This circle comes equipped with an orientation, obtained by declaring that the tangential basepoint at $\infty_i$ points in the positively oriented direction.  In terms of the dual graph $\tau(C)$, this corresponds to a cyclic order on the edges at each internal vertex of $\tau(C)$, which in turn gives the tree a planar structure $\rho(v)$.  Namely, the planar structure is determined by requiring that in a planar embedding, the edges are traversed counterclockwise in cyclic order; see \autoref{fig:trees} for an example.

In summary, we have the following combinatorial description of the operad $\FM{}(\ZZ)$.
\begin{proposition}\label{prop:planar-trees}
    The assignment $(C,v) \mapsto (\tau(C),\rho(v))$ gives an isomorphism from the operad $\FM{}(\ZZ)$ to the operad of planar rooted trees.  The forgetful map $\FM{}(\ZZ) \to \bM{}(\ZZ)$ is identified with the map $(\tau(C),\rho(v))\mapsto \tau(C)$ that forgets the planar structure.
\end{proposition}

\begin{figure}
    \centering
    \begin{subfigure}{0.3\textwidth}\centering
        \begin{tikzpicture}
        \draw(0,0) circle (0.5);
        \draw[fill] (30:0.5) circle (0.05);
        \draw (30:0.75) node {$c$};
        \draw[fill] (150:0.5) circle (0.05);
        \draw (150:0.75) node {$a$};
        \draw[fill] (0,-0.5) circle (0.05);
        \draw[->] (0,-0.5) -- (0.4,-0.5);
        \draw(0,-1) circle (0.5);
        \draw[fill] (0,-1.5) circle (0.05);
        \draw[->] (0,-1.5) -- (0.4,-1.5);
        \draw (0,-1.7) node {$\infty$};
        \draw[fill] (-0.5,-1) circle (0.05);
        \draw (-0.7,-1) node {$b$};
    \end{tikzpicture}
    \end{subfigure}
    \begin{subfigure}{0.3\textwidth}\centering
        \begin{tikzpicture}[scale=0.5]
  \node {$\infty$} [grow'=up]
    child { 
        child {node {$b$}}
        child {
            child {node {$a$}}
            child {node {$c$}}
        }
    };
\end{tikzpicture}
    \end{subfigure}
    \begin{subfigure}{0.3\textwidth}\centering
        \[(b(ac))\]
    \end{subfigure}
    
    \caption{Virtual $\ZZ$-points of $\FM{A}$ are in bijection with planar rooted trees, or equivalently parenthesized monomials built from the elements of $A$; the correspondence is illustrated here with the set of labels $A = \{a,b,c\}$.}
    \label{fig:trees}
\end{figure}
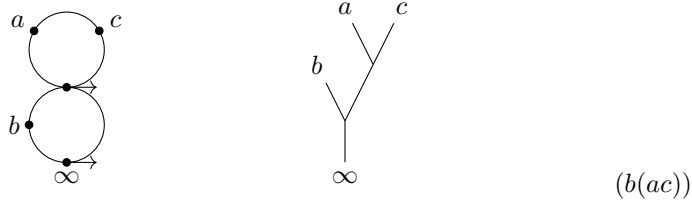

\subsection{The parenthesized braid operad}
Finally, we explain how Bar-Natan's formalism of parenthesized braids from \cite[\S2]{Bar-Natan1998} can be interpreted as the ``fundamental groupoid of $\FM{A}$ with basepoints defined over $\Spec{\ZZ}$''.

By the discussion in  \autoref{sec:KN-points-over-subrings}, we have a canonical injection
\[
\FM{A}(\ZZ)\cong \KN[\ZZ]{\FM{A}}\hookrightarrow \KN{\FM{A}} \cong \FMtop{A}.
\]
This allows us to view $\FM{A}(\ZZ)$ as a collection of basepoints in $\FMtop{A}$.  

\begin{definition}
    We denote by $\Pi_1(\FM{A}) := \Pi_1(\FMtop{A},\FM{A}(\ZZ))$ the fundamental groupoid of the topological space $\FMtop{A}$ based at the subset $\FM{A}(\ZZ)$.  
\end{definition}
Since the inclusion of $\ZZ$-points in the Kato--Nakayama space is a morphism of symmetric monoidal functors, the operadic structure of $\FM{}$ makes $\Pi_1(\FM{})$ into an operad of groupoids.  

In light of \autoref{prop:planar-trees}, the operad of basepoints of $\Pi_1(\FM{A})$ is the operad of planar rooted trees.  Recall that a planar rooted tree with leaves labelled by $A$ can equivalently be encoded in a parenthesized monomial built from the elements of $A$, as illustrated in \autoref{fig:trees}.  Namely, the planar structure induces a total order on the elements of $A$, given by traversing the leaves clockwise.  We then interpret each internal vertex as the bracketing of its edges pointing away from the root.  In particular, if $A = \{1,\ldots,n\}$, then this structure is what Bar-Natan calls a parenthesized permutation.  

The projection $\Conf_A(\CC)\to \FMtop{A}\cong \KN{\FM{A}}$ is a homotopy equivalence, so the fundamental group of $\FMtop{A}$ is the pure braid group on strands labelled by $A$.  More generally if $x,y \in \FM{A}(\ZZ)$, then comparing the total orders on $A$ induced by $x,y$ induces a bijection $\sigma_{x,y} \colon A \to A$, and homotopy classes of paths from $x$ to $y$ are in canonical bijection with lifts of $\sigma_{x,y}$ to an element of the braid group on strands labelled by $A$.  Thus the groupoids $\Pi_1(\FM{A})$ are exactly Bar-Natan's groupoids of parenthesized braids, which were observed by Tamarkin~\cite{TamarkinDisks} to form an operad. The operadic compositions clearly agree with those of $\Pi_1(\FM{})$.  Hence we have the following.

\begin{proposition}
    The operad of groupoids $\Pi_1(\FM{})$ is canonically isomorphic to the operad of parenthesized braids.
\end{proposition}

\bibliographystyle{hyperamsalpha}
\bibliography{virtual-mor-biblio}

\end{document}